\documentclass[11pt]{amsart}

\usepackage[T1]{fontenc}
\usepackage[utf8]{inputenc}

\usepackage[english]{babel}
\usepackage{geometry}
\usepackage{graphicx}

\usepackage{xcolor}
\usepackage{amsmath,amssymb,amsthm,amsfonts,stmaryrd,mathtools}
\usepackage{bbm}
\usepackage{eurosym}
\usepackage{array,dcolumn,mathdots,multirow}
\usepackage{rotating}
\usepackage{fancyvrb}
\usepackage[activate={true,nocompatibility},spacing,kerning]{microtype}

\graphicspath{%
{./Images/}
}

\definecolor{darkgreen}{rgb}{0,0.5,0}
\definecolor{darkred}{rgb}{0.5,0,0}
\definecolor{darkblue}{rgb}{0.004,0.396,0.741}
\definecolor{gcolor}{rgb}{0.004,0.396,0.741}	
\definecolor{warning}{rgb}{1.0,0.6,0.0}	
\definecolor{gray}{rgb}{0.6,0.6,0.6}

\usepackage[pdftex,
            colorlinks,
            hyperfootnotes=false,
            pdffitwindow=true,
            plainpages=false,
            pdfpagelabels=true,
            pdfpagemode=UseOutlines,
            pdfpagelayout=TwoPageRight,
            pdftitle={Asymptotic Expansions of the Limit Laws of Gaussian and Laguerre (Wishart) Ensembles at the Soft Edge},
            pdfauthor={Folkmar Bornemann, TU Munich},
            hyperindex,
            setpagesize=false,
            filecolor=purple,
            urlcolor=darkred,
            citecolor=gcolor,
            linkcolor=gcolor]{hyperref}
\usepackage{remreset}
\usepackage{bm}

\usepackage{listings}
\definecolor{mygreen}{rgb}{0,0.6,0}
\definecolor{mygray}{rgb}{0.5,0.5,0.5}
\definecolor{lightgray}{rgb}{0.925,0.925,0.925}
\definecolor{mymauve}{rgb}{0.58,0,0.82}
\definecolor{gcolor}{rgb}{0.004,0.396,0.741}
\usepackage{caption}
\theoremstyle{plain}

\newtheorem{theorem}{Theorem}[section]
\newtheorem{lemma}{Lemma}[section]
\newtheorem{corollary}{Corollary}[section]
\theoremstyle{definition}

\theoremstyle{remark}
\newtheorem{remark}{Remark}[section]

\DeclareMathOperator{\arcosh}{arcosh}

\numberwithin{equation}{section}

\input{macros.sty}
\VerbatimFootnotes
\makeindex
\begin{document}

\title[Asymptotic Expansions of the Limit Laws at the Soft Edge]{Asymptotic Expansions of the Limit Laws of Gaussian and Laguerre (Wishart) Ensembles at the Soft Edge}
\author{Folkmar Bornemann}
\address{Department of Mathematics, Technical University of Munich, 80290 Munich, Germany}
\email{bornemann@tum.de}


\begin{abstract}
The large-matrix limit laws of the rescaled largest eigenvalue of the orthogonal, unitary, and symplectic $n$-dimensional Gaussian ensembles -- and of the corresponding Laguerre ensembles (Wishart distributions) for various regimes of the parameter $\alpha$ (degrees of freedom~$p$) -- are known to be the Tracy--Widom distributions $F_\beta$ ($\beta=1,2,4$). We establish (paying particular attention to large or small ratios~$p/n$) that, with careful choices of the rescaling constants and of the expansion parameter $h$,  the limit laws embed into asymptotic expansions in powers of $h$, where $h \asymp n^{-2/3}$ resp. $h \asymp (n\,\wedge\,p)^{-2/3}$. We find explicit analytic expressions of the first few expansion terms as linear combinations of higher-order derivatives of the limit law $F_\beta$ with rational polynomial coefficients. The parametrizations are fine-tuned so that the expansion coefficients in the Gaussian cases are, for given~$n$,  the limits $p\to\infty$ of those of the Laguerre cases. Whereas the results for $\beta=2$ are presented with proof, the discussion of the cases $\beta=1,4$ is based on some hypotheses, focusing on the algebraic aspects of actually computing the polynomial coefficients. For the purposes of illustration and validation, the various results are checked against simulation data with large sample sizes.
\end{abstract}

\maketitle

\tableofcontents

\section{Introduction}\label{sect:intro}

\subsection{Background and notation}\label{subsect:background} We start with recalling some basic notions from random matrix theory (cf., e.g., the  monographs \cite{MR2760897,MR2641363,Mehta04}), thereby defining the normalizations and the notation used in this paper. 
For $\F$ being one of the (skew) fields of reals  $\R$, complex numbers $\C$ or quaternions $\HH$, we draw 
random matrices $Y\in \F^{n,p}$ with iid standard normals (in the quaternionic case, however, rescaled to variance two for being compatible with the interrelations studied in \cite{MR1842786}, see Sections~\ref{subsect:introAlgebraic} and \ref{subsect:inter}). The different fields are conveniently encoded by their real dimensions
\[
\beta:= \dim_{\R} \F \in \{1,2,4\}.
\] 
Then, if $p=n$, the random $n\times n$ matrices $X = (Y+Y^*)/2$
define the Gaussian orthogonal ensemble (GOE, $\F=\R$), the Gaussian unitary ensemble (GUE, $\F=\C$), and the Gaussian symplectic ensemble (GSE, $\F=\HH$), addressing the ensembles  by the name of their symmetry groups. In multivariate statistical analysis, one considers the random $n\times n$ matrices $X = Y Y^*$
which define the standard $n$-variate Wishart distribution of $p$ degrees of freedom (numbers of observations) with identity covariance matrix (the `null' case).

In this paper, we study the induced probability distribution $\prob(\lambda_{\max} \leq x)$ of the largest eigenvalue $\lambda_{\max}$ of the random matrix $X$ (the spectrum is real in all cases), which we denote~by
\begin{equation}\label{eq:Ebeta}
E_\beta(n;x),\quad E_\beta(n,p; x).
\end{equation}
Since there is no upper bound on the values of $\lambda_{\max}$, the largest eigenvalue marks what is called the soft edge of the spectrum. In particular, we will study the distributions \eqref{eq:Ebeta} in the large matrix limit $n\to \infty$ -- noting that, when addressing the real Wishart case, comparatively small (or, because of the symmetry \eqref{eq:WishartSymmetry} below,  large) values of the ratio  $p/n$ are of particular interest in modern statistics, cf. \cite[p.~296]{MR1863961}. 

From linear algebra, it is known that, for $p\geq n$, the eigenvalues of $Y^*Y$ are exactly the eigenvalues of $YY^*$ padded by $p-n$ additional zeros. Therefore, if we discard  those eigenvalues that are zero for algebraic reasons,  the eigenvalue distribution of  Wishart matrices is invariant under the permutation $p \leftrightarrow n$. In particular, there is the symmetry
\begin{equation}\label{eq:WishartSymmetry}
E_\beta(n,p; x) = E_\beta(p,n;x).
\end{equation}
Because of this symmetry, it does not really matter that,  in the statistics literature, the letters $n$ and $p$ are often used with their roles reversed. 

The joint probability densities of the (unordered) eigenvalues of $X$ are known to be\footnote{The normalization constants (`partition functions') are spelt out in \eqref{eq:Z}.}
\begin{equation}\label{eq:pdfBetaEnsemble}
p(x_1,\ldots,x_n) \propto |\Delta(x)|^\beta \prod_{j=1}^n w_\beta(x_j) \qquad (x \in \R^n),
\end{equation}
where $\Delta(x)$ denotes the Vandermonde determinant 
\[
\Delta(x) := \prod_{j>k} (x_j-x_k),
\]
while the weight functions $w_\beta(\xi)$ are given, for the Gaussian ensembles, by
\begin{equation}\label{eq:GaussianWeight}
w_\beta(\xi) = e^{-\gamma_\beta \xi^2} \qquad (\xi\in\R)
\end{equation}
and, for the Wishart distribution with  $p\geq n$ (i.e., when breaking the symmetry by taking the case with a.s. no zero eigenvalues), by\footnote{We use $\chi_A$ to denote the indicator function of a set $A$.}
\begin{equation}\label{eq:LaguerreWeight}
w_\beta(\xi) = \xi^\alpha e^{-\gamma_\beta \xi}\chi_{\{\xi>0\}}\qquad (\xi\in\R).
\end{equation}
Here, the parameters are\footnote{Without rescaling the variance in the quaternionic (symplectic) case, one would get $\gamma_\beta = \beta/2$ in all three cases. We deviate from that common choice to facilitate a simple formula for interrelating the orthogonal and symplectic ensembles; see Sections~\ref{subsect:introAlgebraic} and~\ref{subsect:inter}. 
As a consequence, the orthogonal and symplectic cases share exactly the same polynomial expansion coefficients; see Eq.~\eqref{eq:FbetaLinForm} -- illustrating a general duality between $\beta$ and $4/\beta$, which is revealed in \cite{MR2522665,MR2461989}.}
\begin{equation}\label{eq:alphabetagamma}
\alpha = \beta(p-n+1)/2-1, \qquad
\gamma_\beta = \begin{cases}
1/2 &\quad \beta = 1,\\*[1mm]
1 & \quad \beta = 2,4.
\end{cases}
\end{equation}
The Laguerre weight function \eqref{eq:LaguerreWeight} renders \eqref{eq:pdfBetaEnsemble} a probability density on $\R^n$ for any real number $\alpha>-1$. The such obtained $n$-level distributions define the Laguerre orthogonal ensemble (LOE, $\beta=1$), Laguerre unitary ensemble (LUE, $\beta=2$) and Laguerre symplectic ensemble (LSE, $\beta=4$). Dumitriu and Edelman \cite{MR1936554} constructed tridiagonal random matrix models that exhibit accordingly distributed eigenvalues (these, and the corresponding models for the Gaussian ensembles, are used to sample efficiently from the distributions).

The induced probability distributions $\prob(\lambda_{\max} \leq x)$ of the largest eigenvalue $\lambda_{\max}$ of a Laguerre $\beta$-ensemble will be denoted by $E_{\LbE,\alpha}(n;x)$.
This way, 
\begin{equation}\label{eq:EbetaLnp}
E_\beta(n,p; x) := E_{\LbE,\alpha}(n;x)\big|_{\alpha=\beta(p-n+1)/2-1}
\end{equation}
extends the definition of the distribution $E_\beta(n,p; x)$ from integer $p$ to any real $p>n-1$. In the course of the paper, we will extend the definition of $E_2(n,p;x)$ even further to any real values $n,p>0$ while upholding the symmetry in $n$ and $p$.

\subsection{Prior work} 

In the large matrix limit $n\to\infty$, with  carefully chosen rescaling constants $\mu_{n'}$, $\sigma_{n'}$, $\mu_{n',p'}$, $\sigma_{n',p'}$ (where $n'$ and $p'$ denote certain modifications of $n$ and $p$ that are necessary in the cases $\beta=1,4$), the following limit law is known to hold uniformly in~$t$ (with certain restrictions on $p$):\footnote{This result was established for the GUE, and the LUE with fixed  parameter $\alpha$ (i.e., $p/n \to 1$),  by Forrester \cite{MR1236195} (with Tracy and Widom \cite{MR1257246} adding the representation of $F_2$ in terms of Painlevé II); for the GOE, GSE and LOE, LSE with fixed  $\alpha$ by Tracy and Widom \cite{MR1385083}, see also Forrester \cite{MR2275509}; for the LUE with $p/n \to \gamma \in (0,\infty)$ by Johansson \cite{MR1737991}, under the same condition on $p$ for the LOE by Johnstone \cite{MR1863961}; for the LOE with $n,p\to\infty$ such that $p/n \to 0,\infty$ by El Karoui \cite{karoui2003largest}. The uniformity in $t$ follows from \cite[Lemma~2.11]{MR1652247}.}
\begin{align}\label{eq:TWlimit}
\lim_{n\to\infty} E_\beta(n;\mu_{n'} + \sigma_{n'} t) = \lim_{n\to\infty} E_\beta(n,p;\mu_{n',p'} + \sigma_{n',p'} t) = F_\beta(t), 
\end{align}
where the functions $F_\beta$ denote the Tracy–Widom distributions\footnote{Because of the scaling of the symplectic case, which leads to $\gamma_4 = 1$ in \eqref{eq:alphabetagamma}, the original $\beta=4$ Tracy–Widom distribution is $F_4(\sqrt{2} s)$ in the notation used here (cf. \cite[Eq.~(54)]{MR1385083}, \cite[Eq.~(2.10)]{MR2895091} and \cite[Eq.~(35)]{MR2165698}).} for $\beta=1,2,4$. 
Exercising even more care in choosing proper rescaling constants (that is, modifications $n',p'$), Johnstone and Ma~\cite{MR3025686}
established a convergence rate of $O(n^{-2/3})$ for the GUE of any dimension and the GOE of even dimension, whereas El Karoui \cite{MR2294977} and Ma \cite{MR2888709}
established a rate of $O((n\wedge p)^{-2/3})$ for the LUE and LOE when $p/n \to \gamma \in (0,\infty)$.\footnote{We use the infix notation $n\wedge p := \min(n,p)$, $n\vee p:=\max(n,p)$.} More details about the genesis of the work of El Karoui,  Johnstone, and Ma on convergence rates and their mutual influence can be found in \cite[§3.2.2]{Taljan}.  

The study of asymptotic expansions (which, in reference to expansions of the central limit theorem, are sometimes called Edgeworth expansions in the probabilistic setting) was initiated for the GUE, and the LUE with a fixed parameter $\alpha$, by Choup \cite{MR2233711,MR2406805}: he established the functional form of the first finite-size correction in terms of Painlevé II. His subsequent study~\cite{MR2492622} of the GOE led to an inconclusive result; see the discussion in \cite[p.~1963]{MR3025686}. Finally, for the LUE with a parameter $\alpha$ proportional to $n$ (which implies $p/n\to \gamma \in (0,\infty)$), Forrester and Trinh \cite{MR3802426}  established  the functional form of the 
first finite-size correction in operator theoretic terms (and, alternatively, in terms of a perturbation system of the $\sigma$-form of Painlevé~II), which is amenable for evaluation by the numerical methods that we developed in \cite{MR2895091, MR2600548,MR3647807}.

\subsection{The new findings of the paper}

In the context of the increasing subsequence problem in random permutations, we studied, in \cite{arxiv:2301.02022} for $\beta=2$ and in \cite{2306.03798} for $\beta=1,4$,  asymptotic expansions of the related Forrester--Borodin hard-to-soft edge transition law~\cite{MR1986402}. We developed a method that allows us to compute the functional form of as many terms of an asymptotic expansion as we care to explicitly look at -- thereby unfolding  a particular layer of additional integrability:\footnote{A magnificent definition of `integrability', attributed to Flaschka, is given by McKean in \cite[Footnote~2]{MR1953973}. See  Footnote~\ref{fn:integrable} for more context.}  the expansion terms are linear combinations of the higher-order derivatives of the limit law $F_\beta$, with coefficients in the space $\Q[t]$ of rational polynomials.

Using this methodology, in the Gaussian cases,  the proper choices of the rescaling parameters and  the expansion parameter are\footnote{The factor $1/4$ fixes the scale of  $h_n$ so that $h_n = \lim_{p\to\infty} h_{n,p}$, where $h_{n,p}$ is the expansion parameter in the Laguerre cases -- whose scale is being fixed according to \eqref{eq:hscale}.} 
\[
\mu_n = \sqrt{2n},\quad  \sigma_n = 2^{-1/2} n^{-1/6}, \quad h_n = 4^{-1} n^{-2/3}.
\]
With the modifications
\begin{equation}\label{eq:nprimeIntro}
n' := 
\begin{cases}
n - \frac12, &\quad \beta = 1,\\*[1mm]
n,           &\quad \beta = 2,\\*[1mm]
2n+ \frac12, &\quad \beta = 4,
\end{cases}
\end{equation}
we will present  asymptotic expansions of the form 
\begin{subequations}\label{eq:GbetaEexpansion}
\begin{equation}
E_\beta(n;\mu_{n'}+\sigma_{n'} t) = F_\beta(t) + \sum_{j=1}^m E_{\beta,j}(t) h_{n'}^j + O(h_{n'}^{m+1}).
\end{equation}
We observe that\footnote{We write $h  \asymp n^{-2/3}$ to denote that there are absolute positive  constants $c, C$ with $c n^{-2/3} \leq h \leq C n^{-2/3}$.} $h_{n'} \asymp n^{-2/3}$.
Moreover, also the expansion terms in \eqref{eq:GbetaEexpansion} exhibit  that additional layer of integrability in taking the multilinear form
\begin{equation}
E_{\beta,j}(t) = \sum_{k=1}^{2j} p_{\beta,jk}(t) F^{(k)}_\beta(t), \quad p_{\beta,jk} \in \Q[t];
\end{equation}
\end{subequations}
and we will provide concrete formulas for the polynomials $p_{\beta,jk}$, up to $m=3$. 

Likewise, in the Laguerre cases, using the Wishart parameter $p$ (`degrees of freedom'),  the proper choices of the rescaling parameters 
and of the expansion parameter  are
\begin{gather*}
\mu_{n,p} = \big(\sqrt{n}+\sqrt{p}\,\big)^2,\quad \sigma_{n,p}= \big(\sqrt{n}+\sqrt{p}\,\big) \bigg(\frac{1}{\sqrt{n}} + \frac{1}{\sqrt{p}}\bigg)^{1/3},\quad 
h_{n,p} = \frac{1}{4} \bigg(\frac{1}{\sqrt{n}} + \frac{1}{\sqrt{p}}\bigg)^{4/3}.
\end{gather*}
To encode the ratio $p/n$, we introduce the additional quantity
\[
\tau_{n,p} = \frac{\sigma_{n,p}}{h_{n,p}\,\mu_{n,p}} = \frac{4}{\big(\sqrt{n}+\sqrt{p}\,\big) \left(\dfrac{1}{\sqrt{n}} + \dfrac{1}{\sqrt{p}}\right)},
\]
which will enter the rational polynomial coefficients of the expansion terms as an additional indeterminate (see \eqref{eq:rationalInTau}) and satisfies
\[
0 < \tau_{n,p} \leq 1, \qquad \lim_{p/n\to 0,\,\infty} \tau_{n,p} = 0.
\] 
Note that all the parameters take a form that is symmetric in $n$ and $p$. The arbitrary scales of $h_{n,p}$ and $\tau_{n,p}$ have been fixed so that
\begin{equation}\label{eq:hscale}
 \sigma_{n,p} = \tau_{n,p} h_{n,p} \mu_{n,p}, \qquad \tau_{n,n}=1.
\end{equation}
We will present asymptotic expansions of the form 
\begin{subequations}\label{eq:LbetaEexpansion}
\begin{equation}
E_\beta(n,p;\mu_{n',p'}+\sigma_{n',p'} t) = F_\beta(t) + \sum_{j=1}^m E_{\beta,j;\tau_{n',p'}}(t) h_{n',p'}^j + O(h_{n',p'}^{m+1}),
\end{equation}
where $n',p'$ are the $\beta$-dependent modifications of $n,p$ as in \eqref{eq:nprimeIntro}. We observe that 
\[
h_{n',p'} \asymp (n\wedge p)^{-2/3}.
\]
Once again, the expansion terms in \eqref{eq:LbetaEexpansion} take the multilinear form
\begin{equation}\label{eq:rationalInTau}
E_{\beta,j;\tau}(t) = \sum_{k=1}^{2j} p_{\beta,jk}(t,\tau) F^{(k)}_\beta(t), \quad p_{\beta,jk} \in \Q[t,\tau];
\end{equation}
\end{subequations}
and we will provide concrete formulas of the polynomials $p_{\beta,jk}$, up to $m=3$. 

As we will see, the expansion terms of the
Gaussian cases are related to those of the Laguerre cases by the simple equations
\[
E_{\beta,j}(t) = E_{\beta,j;\tau}(t)\big|_{\tau=0}.
\]
Because of $\lim_{p\to\infty} h_{n',p'}=h_{n'}$ and $\lim_{p\to\infty}\tau_{n',p'} =0$, this equality of 
the expansion terms is consistent with the  Laguerre-to-Gaussian transition law
\[
\lim_{p\to\infty} E_\beta\big(n,p;\mu_{n',p'} + \sigma_{n',p'} t\big)  = E_\beta\big(n;\mu_{n'} + \sigma_{n'} t\big),
\]
which we establish in Appendix \ref{sect:Laguerre2Gauss}.

Whereas in Part I, for the unitary cases ($\beta=2$), all these results  are presented with  proof (spelling out all the necessary uniform dependencies and tail estimates of the error terms), in Part II, the discussion of the orthogonal and symplectic cases ($\beta=1,4$)  is based on some hypotheses, focusing on the algebraic aspects of actually computing the polynomials~$p_{\beta,jk}$. 

For the purposes of illustration (Part I) and validation (Part II), each of the expansions~is checked, with close matches, against simulation data with a sample size of a thousand million.

\subsection{Proof strategy for unitary ensembles}\label{subsect:strategy}

In the case $\beta=2$, the eigenvalues of $X$ are  known to constitute a determinantal point process with correlation kernel\footnote{In this Section~\ref{subsect:strategy}, we use a simplified notation  which resembles the Gaussian case, suppressing the dependence on $p$, or $\tau$ for that matter, in the Laguerre case.}
\[
K_n(x,y) = \sum_{k=0}^{n-1} \phi_k(x) \phi_k(y) = c_n \frac{\phi_n(x)\phi_{n-1}(y) - \phi_{n-1}(x)\phi_n(y)}{x-y},
\]
where the $\phi_k$ form the $L^2$-orthonormal system of `wave functions'  belonging to the Hermite polynomials (for the Gaussian ensembles) and Laguerre polynomials (for the Laguerre ensembles and the Wishart distribution); the second equality  follows from the corresponding Christoffel--Darboux formulas. The joint probability distribution of the eigenvalues can then be expressed as
\[
p(x_1,\ldots,x_n) = \det_{j,k=1}^n K_n(x_j,x_k)
\]
and the distribution of the largest eigenvalue is  given by
the Fredholm determinant
\begin{equation}\label{eq:En}
E_2(n;x) = \det\big(I-K_n\big)\big|_{L^2(x,\infty)}.
\end{equation}
Since the eigenvalues of a unitary ensemble of dimension $n$ are, in the large matrix limit, a.s. uniformly approximated by the zeros of 
the wave function $\phi_n$ (see \cite{MR2320657}), the region of the largest eigenvalue marks the transition between oscillatory behavior (`spectrum') and exponential decay (`no spectrum') of $\phi_n$.
Therefore, the rescaling $x = \mu_n + \sigma_n t$ is constructed in such a way that this transition happens in the new variable at about $t \approx 0$ with a fixed scale. Because  the wave functions satisfy certain linear differential equations of second order, they can be matched to the one of the Airy function $\Ai(t)$, which is the standard model of a function exhibiting such a transition. This is spelt out in Part III of this paper, where we will go to great length finding asymptotic formulas of the form $(n_+:=n+1/2)$
\begin{equation}\label{eq:phiExpansion}
c h_{n_+}^\gamma\phi_n\big(\mu_{n_+} + \sigma_{n_+} s\big) = \Ai(s) \sum_{k=0}^{m} p_k(s) h_{n_+}^{k} + \Ai'(s) \sum_{k=0}^{m} q_k(s) h_{n_+}^{k} + h_{n_+}^{m+1} O(e^{-s})
\end{equation}
as $n\to\infty$, uniformly when $s$ is bounded from below; in the Laguerre case, we can arrange for the expansion to be uniformly valid for the Laguerre parameter $\alpha$ bounded proportional to $n$. Here,
\[
h_{n_+} \asymp n^{-2/3}
\] 
is a carefully chosen expansion parameter which prevents the expansion from being polluted\footnote{See Remarks \ref{rem:Her_m01} and \ref{rem:Lag_m01} for examples of such pollutions if one chooses to expand in powers of $n^{-1/3}$.} by  terms of the order of odd positive powers of $n^{-1/3}$, and the $p_k, q_k$ turn out to be certain polynomials starting with $p_0(s)=1$ and $q_0(s)=0$, $c>0$ and $\gamma$ are certain constants. The catch is that we are not just interested in the existence and validity of such an expansion but that we need to provide a method of actually computing the coefficients $p_k, q_k$ by, say, using a computer algebra system. 
 
Plugging the expansions \eqref{eq:phiExpansion} for $\phi_n(x)$ and $\phi_{n-1}(x)$ etc. into the correlation kernel, while re-expanding them at $x = \mu_{n} + \sigma_{n} \xi$ etc., gives, by the symmetry of $n$ being the midpoint of $n_+$ and $n_-:=(n-1)_+ = n-1/2$, a kernel expansion of the form
\[
\sigma_n  K_n(\mu_n + \sigma_n \xi,\mu_n + \sigma_n \eta) 
 = K_{\Ai}(\xi,\eta) + \sum_{j=1}^m K_j(\xi,\eta) h_n^j + h_n^{m+1}\cdot O\big(e^{-(\xi+\eta)}\big)
 \]
such that:
\begin{itemize}\itemsep=3pt
\item it is uniformly valid as $h_n\to 0^+$ when $\xi,\eta$ are bounded from below,
\item  the leading order is given by the Airy kernel (for the integral formula, see \cite[Eq.~(4.5)]{MR1257246})
\begin{equation}\label{eq:AiryKernel}
K_{\Ai}(x,y) := \frac{\Ai(x)\Ai'(y)-\Ai'(x)\Ai(y)}{x-y} = \int_0^\infty \Ai(x+t)\Ai(y+t)\,dt,
\end{equation}
\item the expansion kernels $K_j$ are  finite rank kernels of the form
\begin{equation}\label{eq:finiteRank}
K_j(x,y) = \sum_{\kappa,\lambda\in\{0,1\}}  p_{j,\kappa\lambda}(x,y) \Ai^{(\kappa)}(x)\Ai^{(\lambda)}(y),
\end{equation}
with certain polynomials $p_{j,\kappa\lambda}(x,y)$. This important structure follows from some intricate concrete cancellation and divisibility properties of the coefficients in \eqref{eq:phiExpansion}.
\end{itemize}

\begin{remark}\label{rem:mstar}
Because  our method of proof is based on checking those intricate cancellation and divisibility properties of the polynomial coefficients in \eqref{eq:phiExpansion} by  explicitly inspecting all the cases $m=1,\ldots, m_*$, up to some  $m_*$ (we have to stop somewhere and did so at $m_*=10$), we have to impose  a bound $m_*$ on the number of expansion terms in this paper. In the related case \cite[Theorem~3.1]{arxiv:2301.02022} of the hard-to-soft edge transition for $\beta=2$, Yao and Zhang \cite{arXiv:2309.06733} were able to remove such a restriction by a Riemann--Hilbert analysis which infers the necessary properties of the polynomials from a certain algebraic structure of the Airy function and its derivatives. We expect their analysis to carry over to our present study so that the restriction on $m$ can be removed here, too.
However, other restrictions on $m$, relating to establishing the functional form of the expansion terms $E_{\beta,j}$, would still apply. Throughout the paper, we will use the same bound $m\leq m_*$ for those restrictions, too.
\end{remark}

Now, since the expansion of the wave functions, and therefore also of the kernel, can be differentiated term-wise, lifting the kernel expansion to the Fredholm determinant \eqref{eq:En} is admissible by \cite[Theorem~2.1]{arxiv:2301.02022} and gives
\begin{equation}\label{eq:introE2expan}
\begin{aligned}
E_2(n;\mu_n + \sigma_n t) &= F_2(t) + \sum_{j=1}^m E_{2,j}(t) h_n^j + h_n^{m+1} O(e^{-2t}),\\*[1mm]
F_2(t) &=\det\big(I-K_{\Ai}\big)\big|_{L^2(t,\infty)},
\end{aligned}
\end{equation}
which is uniformly valid as $h_n\to 0^+$ when $t$ is bounded from below. Generally, the functions $E_{2,j}(t)/F_2(t)$ can be evaluated as highly nonlinear expressions in terms of operator traces involving the resolvent of $K_{\Ai}$ and the expansion kernels $K_1,\ldots,K_j$ (see \cite[§2.3]{arxiv:2301.02022} and \eqref{eq:G1G3}). However, because of the particular finite-rank structure \eqref{eq:finiteRank} of the expansion kernels,  first trading  powers of $x$, $y$ for higher derivatives of the Airy function (see \eqref{eq:airyrule}), and then applying the algorithm of \cite[Appendix~B]{arxiv:2301.02022} (which generalizes the work of Shinault and Tracy~\cite{MR2787973}),
yields the multilinear structure
\[
E_{2,j}(t) = \sum_{k=1}^{2j} p_{2,jk}(t) F_2^{(k)}(t),
\]
with certain concretely computable polynomials $p_{jk}$ -- see Appendix~\ref{app:airykernel} for the general mechanics of this type of integrability, and Theorems~\ref{thm:softGUE} and  \ref{thm:softLUE} for the concrete results.

\subsection{The algebraic method for  orthogonal and symplectic ensembles}\label{subsect:introAlgebraic}

Here, in analogy to our study of the hard-to-soft edge transition in \cite{2306.03798}, we start with the Forrester--Rains interrelations \cite{MR1842786}, symbolically written as
\[
\UE_{n} = \even(\OOE_{n} \cup \OOE_{n+1}), \quad \SE_n = \even(\OOE_{2n+1}),
\]
between the unitary, orthogonal, and symplectic ensembles.
For the Laguerre ensembles it is understood that the Wishart parameters $p$ are connected in just the same way as the dimensions  $n$. In Section~\ref{subsect:inter}, we  show that these interrelations can be recast as the formulas\footnote{As in Section~\ref{subsect:strategy}, we use a simplified notation  which resembles the Gaussian cases, suppressing the dependence on $p$, or $\tau$ for that matter, in the Laguerre cases.}
\begin{equation}\label{eq:introEpm}
\begin{gathered}
E_1(n;x) = E_+(n;x),\quad E_2(n;x) = \frac{1}{2} \sum_{\nu=0,1} E_+(n+\nu;x)E_-(n+1-\nu;x),\\*[2mm]
 \quad E_4(n;x) = \frac{1}{2} \big(E_+(2n+1;x) + E_-(2n+1;x)\big),
\end{gathered}
\end{equation}
where $E_-(n;x)$ is a certain function related to the orthogonal ensemble of dimension $n$, involving the second largest eigenvalue. The form of these formulas closely resembles the corresponding limit interrelations 
\begin{equation}\label{eq:introFpm}
F_1(t) = F_+(t),\quad F_2(t) = F_+(t) F_-(t), \quad F_4(t) = \frac{1}{2} \big(F_+(t) + F_-(t)\big)
\end{equation}
between the various Tracy--Widom distributions, see \cite[Eq.~(53/54)]{MR1385083}. In Section~\ref{subsect:selfconsistent}, we show that the `self-consistent expansion hypothesis', asserting that there are expansions of the particular form
\begin{equation}\label{eq:selfconsistent}
E_\pm(n;\mu_{n_-} + \sigma_{n_-} t) = F_\pm(t) + \sum_{j=1}^m E_{\pm,j}(t) h_{n_-}^j + O(h_{n_-}^{m+1}),
\end{equation}
would consolidate, by the symmetry of $n$ being the midpoint of $n_-$ and $(n+1)_-$, the interrelations in \eqref{eq:introEpm} and \eqref{eq:introFpm} with the already established expansion for $\beta=2$. 

Now, by performing all the additional Taylor expansions that are necessary to express the formulas with reference to a common variable $t$, this gives  (differential) relations between the expansion terms $E_{2,j}$ and the $E_{\pm,k}$ ($k=1,\ldots,j$) -- the first of them being simply
\begin{equation}\label{eq:1steq}
E_{2,1} = F_+ E_{-,1} + E_{+,1}F_-.
\end{equation}
In Section~\ref{subsect:multilinear}, we show that these relations suffice to uniquely determine the $E_{\pm,j}$ under the `linear form hypothesis', asserting that the
expansion terms enjoy a multilinear structure of the form\footnote{In the analogous case of the hard-to-soft edge transition, we could prove that multilinear structure to be true for the first case $j=1$, see \cite[§3.3.3]{2306.03798}.}
\[
E_{\pm,j}(t) = \sum_{k=1}^{2j} p_{\pm,jk}(t) F_\pm^{(k)}(t),
\]
with certain polynomials $p_{\pm,jk} \in \Q[t]$ (resp. $p_{\pm,jk} \in \Q[t,\tau]$ in the Laguerre case). For example, the case $m=1$, as displayed in \eqref{eq:1steq}, becomes
\begin{equation}\label{eq:1stA}
p_{2,11} \frac{F_2'}{F_2} + p_{2,12} \frac{F_2''}{F_2} =  \sum_{\sigma = \pm} \left(p_{\sigma,11} \frac{F_\sigma'}{F_\sigma} + p_{\sigma,12} \frac{F_\sigma''}{F_\sigma}\right).
\end{equation}
From equations such as this one, the polynomials $p_{\pm,jk}$ are computed as follows. The Tracy--Widom theory \cite{MR1385083} yields
\[
\frac{F_2^{(k)}}{F_2},\; \frac{F_\pm^{(k)}}{F_\pm} \in \Q[t][q,q'],
\]
where $q$ is the Hastings--McLeod solution of Painlevé II. By the theory of Painlevé transcendents \cite{MR1960811}, it is known that $q,q'$ behave as indeterminates over $\Q[t]$, so that $\Q[t][q,q']$ can be viewed as generic multivariate rational polynomials in three indeterminates. If written in $\Q[t,\tau][q,q']$ (with $\tau=0$ in the Gaussian case), the relations between the $E_{2,j}$ and the $E_{\pm,k}$ ($k=1,\ldots,j$)
give, inductively, a linear equation in $\Q[t,\tau][q,q']$ for the $4m$ unknowns
\[
p_{\pm,m\mu} \in \Q[t,\tau] \quad (\mu=1,\ldots,2m).
\]
By comparing coefficients of the $(q,q')$-monomials, we get an overdetermined linear system of size $(16m^2-8m+5)\times 4m$ over $\Q[t,\tau]$. For example, the $13\times 4$ linear system, which is equivalent to \eqref{eq:1stA} and is displayed in all its glory in \eqref{eq:exampleSystem}, yields
\[
p_{\pm,11}=p_{2,11},\quad p_{\pm,12} = 2p_{2,12}.
\]
The much more involved solutions for $m=2,3$ are presented in Section~\ref{subsect:multilinear} on pp.~\pageref{eq:exampleSystem}--\pageref{eq:plusMinusSame}. By explicit computation
for $m\leq m_*$, these linear systems are  uniquely solvable in  $\Q[t,\tau]^{4m}$ and their solutions satisfy
\[
p_{+,m\mu} = p_{-,m\mu}\quad (\mu=1,\ldots,2m).
\]
In summary, under the hypotheses made
for the functions $E_\pm(n;x)$, using \eqref{eq:introEpm} and \eqref{eq:introFpm}, we obtain the expansions 
\[
E_\beta(n;\mu_{n'} + \sigma_{n'} t) = F_\beta(t) + \sum_{j=1}^m E_{\beta,j}(t) h_{n'}^j + O(h_{n'}^{m+1})\qquad (\beta=1,4),
\]
where $n' = n_- = n-1/2$ for $\beta=1$ and $n'=(2n+1)_- = 2n+1/2$ for $\beta=4$ -- which justifies the particular modifications  in \eqref{eq:nprimeIntro}. Here, the expansion terms take the form
\[
E_{\beta,j}(t) = \sum_{k=1}^{2j} p_{+,jk}(t) F_\beta^{(k)}(t) \qquad (\beta=1,4),
\]
so that, most remarkably, the orthogonal and symplectic ensembles share the same polynomial coefficients. The concrete results up to $m=3$, displayed in \eqref{eq:Fbeta2} and \eqref{eq:F22LbetaE}, are in perfect agreement with the fine properties of  simulation data with a sample size of a thousand million, see Figures~\ref{fig:GOEGSESimulation} and \ref{fig:LOELSESimulation}. Besides the self-consistency of the expansions and the unique solvability of the linear systems over $\Q[t,\tau]$, such a close agreement with actual data provides further strong evidence for the correctness of the underlying hypotheses -- this way, our study resembles theoretical physics: explaining  accurate experiments with a  quantitative theory. 

\part*{Part I. Unitary Ensembles}

\section{Expansion at the soft edge of the GUE}\label{sect:GUE}

The eigenvalues of the GUE are known (cf., e.g., \cite[§4.2]{MR2760897} and \cite[Chap.~5]{MR2641363}) to constitute a determinantal point process, with a correlation kernel given by the finite-rank projection kernel
\begin{equation}\label{eq:KnGUE}
K_n^{\GUE}(x,y) := \sum_{k=0}^{n-1} \phi_k(x)\phi_k(y) = \sqrt\frac{n}{2} \frac{\phi_n(x)\phi_{n-1}(y)-\phi_{n-1}(x)\phi_n(y)}{x-y},
\end{equation}
where the $\phi_n$ are the $L^2$-orthonormal system of Hermite `wave functions'\/\footnote{Lacking a better name, we take this one in reference to quantum mechanics where the $\phi_n$ are known as `oscillator wave functions'. They are sometimes called `Hermite functions' or `Hermite--Gaussian functions' in the literature.} (with $H_n$ denoting the Hermite  polynomials, as defined in \cite[§18.3]{MR2723248})
\begin{equation}\label{eq:HermitePhiN}
\phi_n(x) := \frac{e^{-x^2/2}}{\pi^{1/4}\sqrt{2^n n!}} H_n(x)\quad (-\infty < x < \infty).
\end{equation}
The distribution of the largest eigenvalue is given by the Fredholm determinant (cf., e.g., \cite[§3.4]{MR2760897} and \cite[Chap.~9]{MR2641363})
\begin{equation}\label{eq:EGUE}
E_2(n;x) = \det\big(I- K_n^{\GUE}\big)\big|_{L^2(x,\infty)}.
\end{equation}
By following the strategy of Section~\ref{subsect:strategy}, we will arrive at the following theorem.

\begin{theorem}\label{thm:softGUE} 
With the scaling and expansion parameters\footnote{In view of the later discussion, leading to \eqref{eq:GUEnaturalScaling}, these are, up to a possible rescaling of $h_n$, the unique `natural' choices of parameters such that the leading order is normalized to $F_2(t)$ and the expansion proceeds in powers of $h_n\asymp n^{-2/3}$. Any other choice would pollute the expansion with terms that scale like odd positive powers of order $n^{-1/3}$, or even worse. Here, the particular scaling of $h_n$ has been chosen to facilitate the simple relation to the Laguerre ensemble, which is discussed in Remark~\ref{rem:LUE2GUE}.}
\begin{equation}\label{eq:GUEscaling}
\mu_n=\sqrt{2n}, \quad \sigma_n = 2^{-1/2}n^{-1/6}, \quad h_n =\frac{\sigma_n}{2\mu_n}= 4^{-1} n^{-2/3},
\end{equation}
there holds the large matrix expansion of the $\GUE$ at the soft edge,
\begin{equation}\label{eq:GUE}
E_2\big(n;\mu_n + \sigma_n t\big) = F_2(t) + \sum_{j=1}^m E_{2,j}(t) h_n^j + h_n^{m+1}\cdot O(e^{-2t}),
\end{equation}
which is uniformly valid when $t_0\leq t$ as $h_n\to0^+$, with $m$ being any fixed integer in the range\footnote{For the  reasons explained in Remark~\ref{rem:mstar}, we have to impose a bound $m_*$ on the number of expansion terms (we have explicitly checked the cancellation and divisibility properties in the proof of Lemma~\ref{lem:GUEkernexpan},  up to $m_*=10$); we expect this bound to be removed by extending the analysis of \cite{arXiv:2309.06733}.} 
$0\leq m\leq m_*$ and $t_0$ any fixed real number. Preserving uniformity, the expansion can be repeatedly differentiated w.r.t. the variable $t$.
Here, the $E_{2,j}$ are certain smooth functions -- the first instances are
\begin{subequations}\label{eq:F22}
\begin{align}
E_{2,1}(t) &= \frac{t^2}{5}  F'(t)-\frac{3}{10} F''(t),\label{eq:F21}\\*[1mm]
E_{2,2}(t) &=-\bigg(\frac{141}{350} + \frac{8t^3}{175}\bigg) F_2'(t) + \bigg(\frac{39t}{175} + \frac{t^4}{50}\bigg) F_2''(t) - \frac{3t^2}{50} F_2'''(t) + \frac{9}{200} F_2^{(4)}(t),\\*[2mm]
E_{2,3}(t) &= 
\bigg(\frac{2216 t}{7875} + \frac{148 t^4}{7875}\bigg)  F_2'(t)
-\bigg( \frac{53t^2}{210}  + \frac{8t^5}{875} \bigg) F_2''(t) \\*[0.5mm]
&\qquad +\bigg(\frac{10403}{31500} +\frac{51 t^3}{875} + \frac{t^6}{750}\bigg) F_2'''(t)
-\bigg(\frac{117t}{1750} +\frac{3 t^4}{500} \bigg)F_2^{(4)}(t)\notag\\*[0.5mm]
&\qquad +\frac{9 t^2 }{1000}F_2^{(5)}(t)
-\frac{9}{2000} F_2^{(6)}(t).\notag
\end{align}
\end{subequations}
\end{theorem}

\begin{figure}[tbp]
\includegraphics[width=0.325\textwidth]{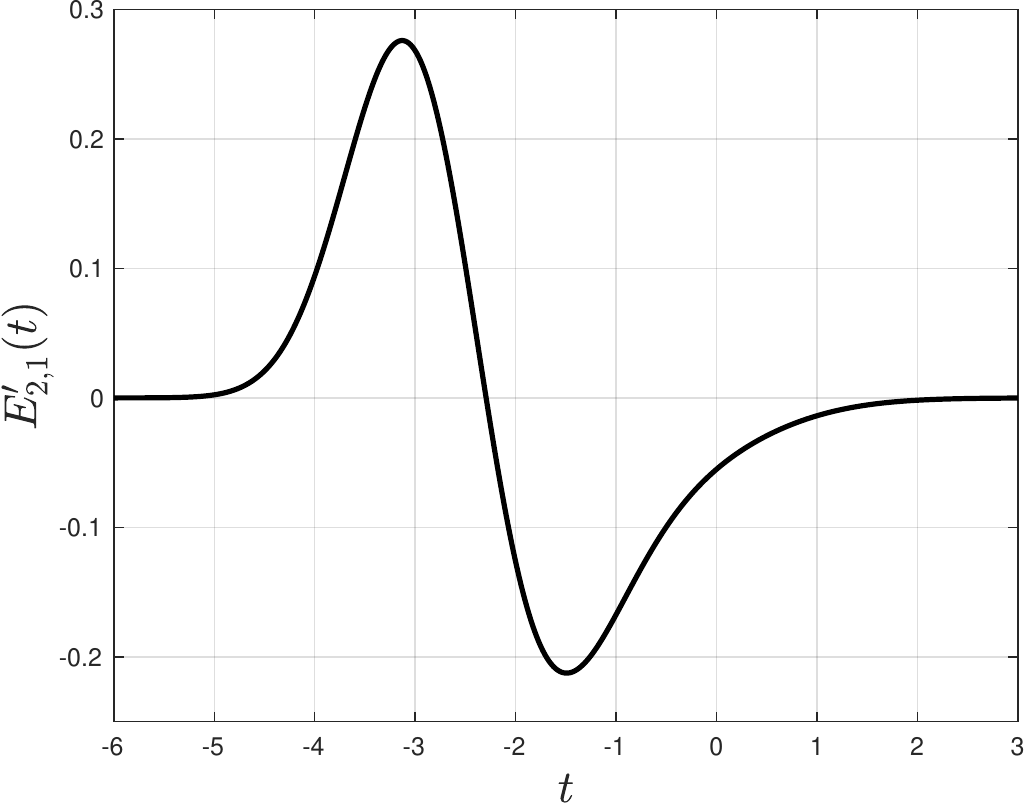}\hfil\,
\includegraphics[width=0.325\textwidth]{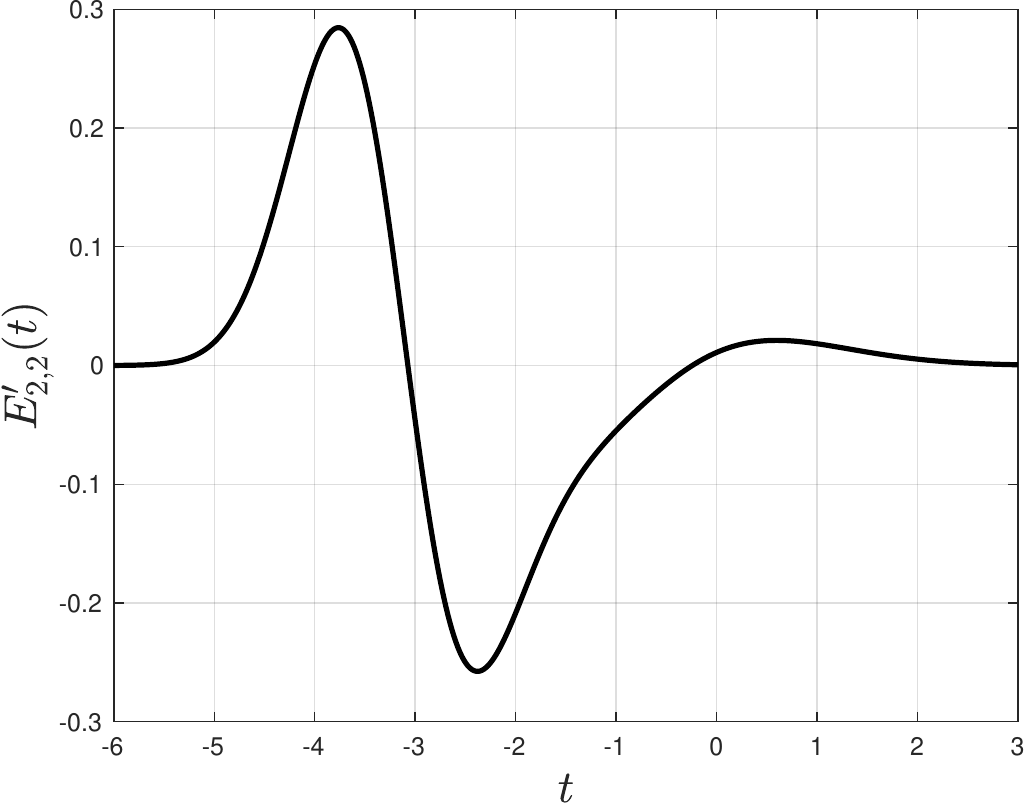}\hfil\,
\includegraphics[width=0.325\textwidth]{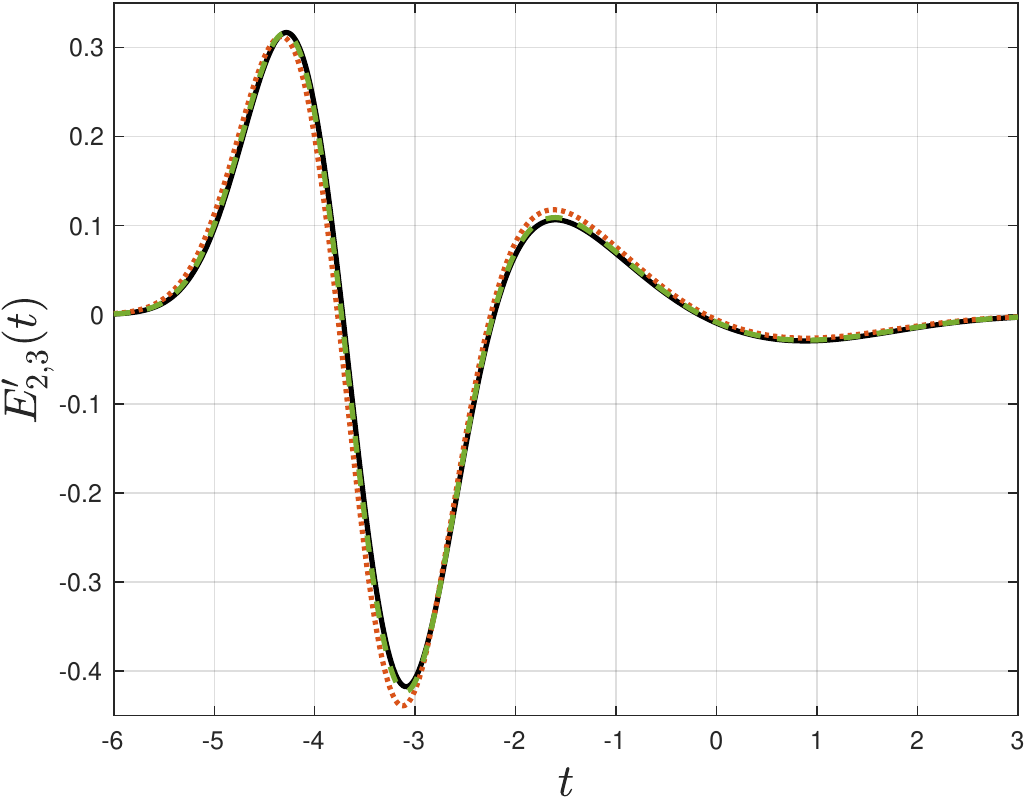}
\caption{{\footnotesize Plots of $E_{2,1}'(t)$ (left panel) and $E_{2,2}'(t)$ (middle panel).
The right panel shows $E_{2,3}'(t)$ (black solid line) with the approximations \eqref{eq:F23} for $n=10$ (red dotted line) and $n=80$ (green dashed line): the close match validates the functional forms displayed in~\eqref{eq:F22} and the differentiability of  the expansion \eqref{eq:GUE}. Details about the numerical method can be found in \cite{MR2895091, MR2600548,arxiv.2206.09411,MR3647807}.}}
\label{fig:GUE}
\end{figure}

The proof of Theorem~\ref{thm:softGUE} is split into several steps and will be given in Sections~\ref{subsect:kernelGUEexpan}--\ref{sect:functionalformGUE}. We will comment on prior work on the cases $m=0$ and $m=1$ in Section~\ref{subsect:previousGUE}.

Using numerical methods we can validate the formulas in (\ref{eq:F22}), and the differentiability of the expansion \eqref{eq:GUE}: Figure~\ref{fig:GUE} plots  $E_{2,1}'(t)$, $E_{2,2}'(t)$, $E_{2,3}'(t)$ next to the approximation
\begin{equation}\label{eq:F23}
E_{2,3}'(t) \approx h_n^{-3} \cdot \big( \tfrac{d}{dt}E_2\big(n;\mu_n + \sigma_n t\big) - F_2'(t) - E_{2,1}'(t)h_n - E_{2,2}'(t) h_n^2\big) 
\end{equation}
for matrix dimensions $n=10$ and $n=80$ and we observe a close match. Note that the absolute scale of the expansion terms  shown in Figure~\ref{fig:GUE} is in itself already quite small, which explains that in actual statistical applications the accuracy of the limit law is mostly sufficient. In particular, we get for  $n\geq 2$ the following numerical bounds for the first approximations of the probability density of the largest eigenvalue of the $\GUE$:
\begin{subequations}
\begin{gather}
\max_{t\in\R} \big|\tfrac{d}{dt}E_2\big(n;\mu_n + \sigma_n t \big) - F_2'(t)\big| < 0.07 \cdot n^{-2/3},\\*[2mm]
\max_{t\in\R} \big|\tfrac{d}{dt}E_2\big(n;\mu_n + \sigma_n t \big) - F_2'(t) - \tfrac14 E_{2,1}'(t) n^{-2/3} \big| < 0.02 \cdot n^{-4/3}.
\end{gather}
\end{subequations}
For a matrix dimension as small as $n=10$,
the latter estimate gives an absolute error bound of the first finite correction of about $10^{-3}$, which is more than an order of magnitude larger than the sampling error of a statistics with $N=10^9$ draws from the $\GUE$. As shown in Figure~\ref{fig:GUESimulation}, for a sample size that large, the graphical forms of the first two finite-size corrections are perfectly matched by the simulation data, with the sampling errors just starting to become visible for the second finite-size correction. In contrast, Figure~\ref{fig:LUESimulation} shows that, in a parameter regime of the $\LUE$ where the absolute size scale of the expansion terms is about one order of magnitude larger than for the $\GUE$, the sampling errors are even less pronounced.

\begin{figure}[tbp]
\includegraphics[width=0.325\textwidth]{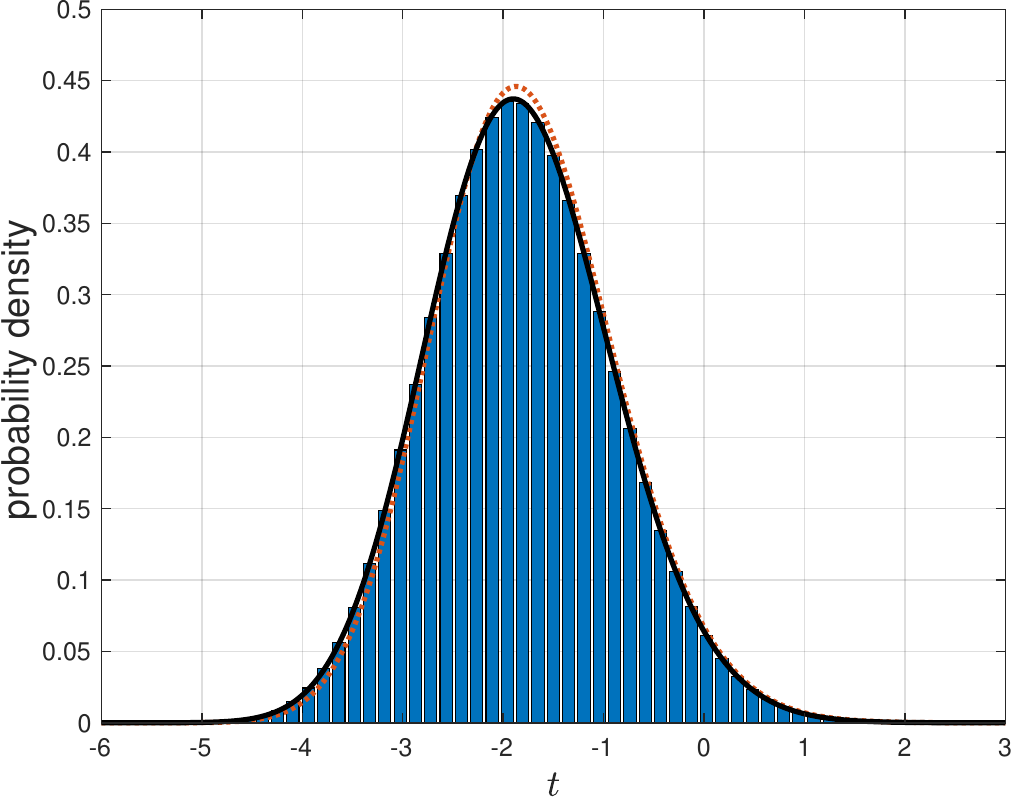}\hfil\,
\includegraphics[width=0.325\textwidth]{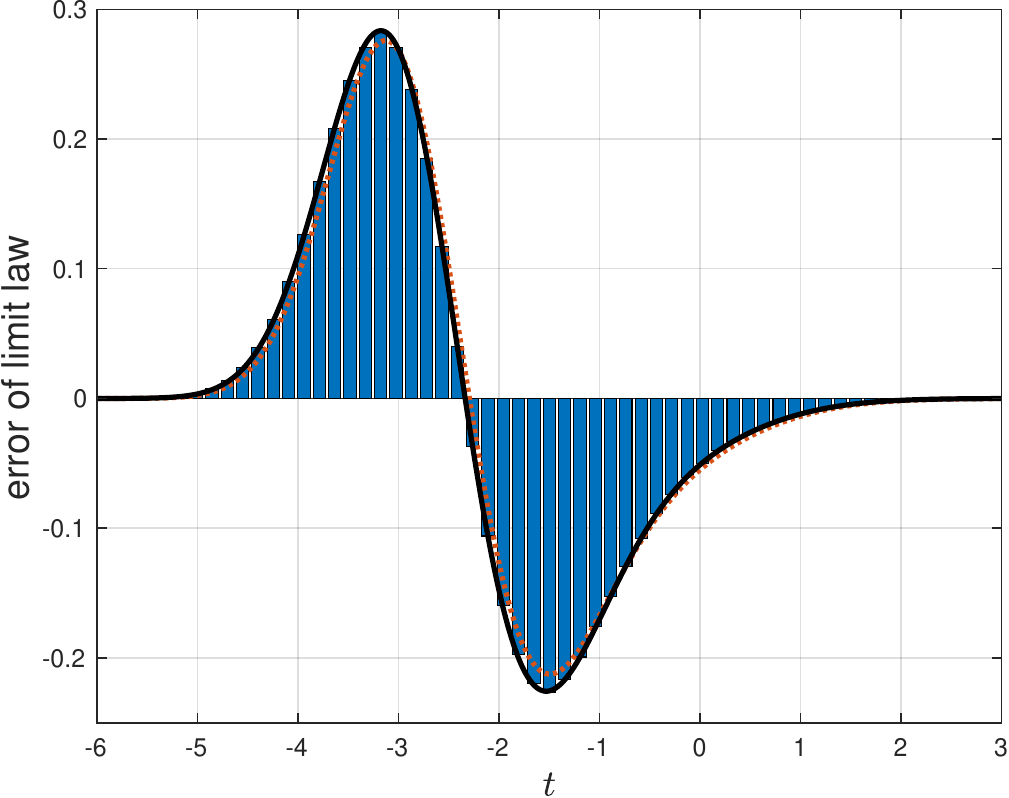}\hfil\,
\includegraphics[width=0.325\textwidth]{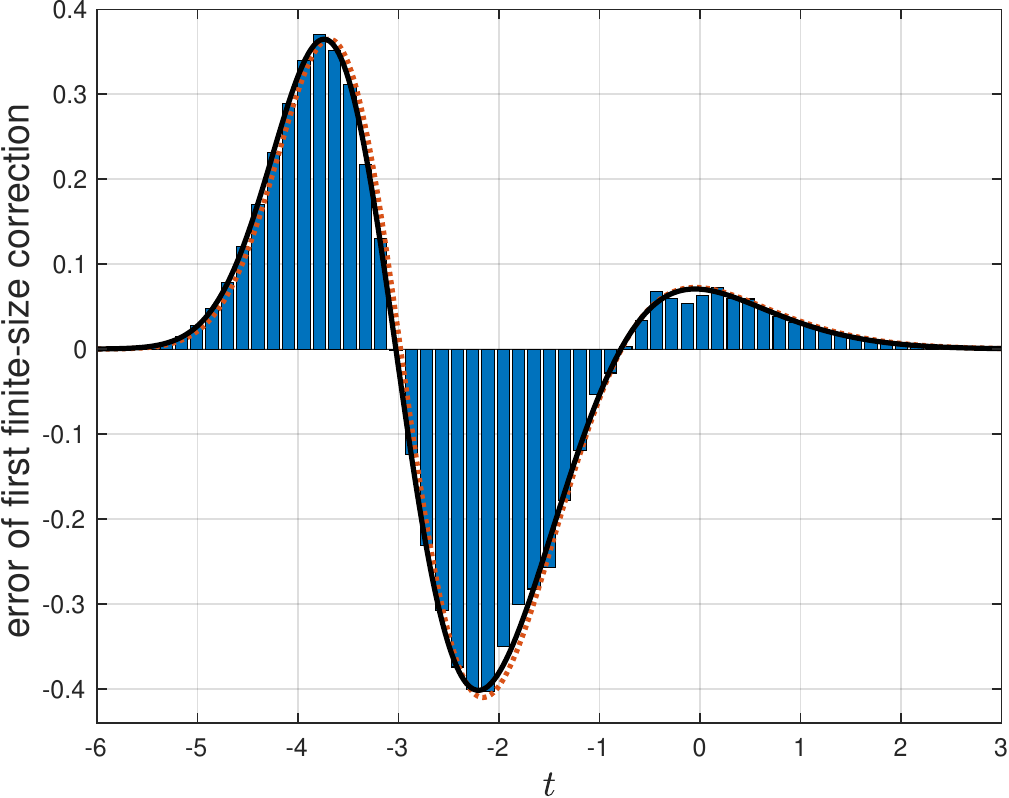}
\caption{{\footnotesize [Expansion terms with a tilde are  `histogram adjusted', see Appendix~\ref{app:histo}.] Left panel: histogram  of the scaled largest eigenvalues $(\lambda_{\max} - \mu_n)/\sigma_n$, for $N=10^9$ draws from the $\GUE$ for matrix dimension $n=10$ (blue bars) vs. the Tracy--Widom limit density $F_2'$ (red dotted line) 
and its first finite-size correction $F_2' + h_n E_{2,1}'$ (black solid line). Middle panel: difference, scaled by $h_n^{-1}$,  between the histogram midpoints and the limit density (blue bars) vs. the correction terms $E_{2,1}'$ (red dotted line) and $E_{2,1}' + h_n \tilde E_{2,2}'$ (black solid line). Right panel: difference, scaled by $h_n^{-2}$, between 
the histogram midpoints and the first finite-size correction (blue bars) vs. the correction terms $\tilde E_{2,2}'$ (red dotted line) and $\tilde E_{2,2}' + h_n \tilde E_{2,3}'$ (black solid line). Note that sampling errors are about to be visible here.}}
\label{fig:GUESimulation}
\end{figure}

\subsection{Kernel expansion}\label{subsect:kernelGUEexpan}

We start with expanding the induced transformation of the kernel. 

\begin{lemma}\label{lem:GUEkernexpan} With the parameters $\mu_n, \sigma_n, h_n$ as introduced in \eqref{eq:GUEscaling}, the  linear change of variables $s=\mu_n + \sigma_n t$ induces the transformed kernel 
\begin{equation}\label{eq:GUEkernexpan}
\begin{aligned}
\hat K_n^{\GUE}(x,y) &:= \sigma_n \, K_n^{\GUE}(\mu_n + \sigma_n x,\mu_n + \sigma_n y) \\
 &= K_{\Ai}(x,y) + \sum_{j=1}^m K_j(x,y) h_n^j + h_n^{m+1}\cdot O\big(e^{-(x+y)}\big),
\end{aligned}
\end{equation}
where the expansion is uniformly valid when $x,y \geq t_0$ as $h_n\to 0^+$, $m$ being any fixed integer in the range $0\leq m\leq m_*$ and $t_0$ any fixed real number. Here, the expansion kernels $K_j$ are certain finite rank kernels of the form
\[
K_j(x,y) = \sum_{\kappa,\lambda\in\{0,1\}}  p_{j,\kappa\lambda}(x,y) \Ai^{(\kappa)}(x)\Ai^{(\lambda)}(y)
\]
with polynomial coefficients $p_{j,\kappa\lambda}(x,y)\in \Q[x,y]$ -- the first instances are
\begin{subequations}\label{eq:K2GUE}
\begin{align}
K_1(x,y) &= \frac{1}{10} \Big( -2\big(x^2+xy+y^2\big)\Ai(x)\Ai(y) \label{eq:K1GUE}\\*[0mm]
&\qquad  + 3 \big(\!\Ai(x)\Ai'(y) + \Ai'(x)\Ai(y)\big)+ 2(x+y)\Ai'(x)\Ai'(y) \Big),\notag\\*[2mm]
K_2(x,y) &= \frac{1}{1400}\Big(6\Big(20(x^3+y^3)+6 x y
   (x+y)+21\Big) \Ai(x)\Ai(y) \\*[0mm]
&\qquad   +\Big(28(x^4+x^3 y-x^2 y^2-x y^3-y^4)-135 x-261 y\Big) \Ai(x)\Ai'(y) \notag\\*[0mm]
&\qquad\quad   -\Big(28(x^4+x^3 y+x^2 y^2-x y^3-y^4)+261 x+135 y\Big) \Ai'(x)\Ai(y) \notag\\*[0mm]
 &\qquad\qquad     +4\Big(5(x^2+y^2)-16 x y\Big) \Ai'(x)\Ai'(y)\Big).\notag
\end{align}
\end{subequations}
Preserving uniformity, the kernel expansion \eqref{eq:GUEkernexpan} can repeatedly be differentiated w.r.t. $x$, $y$.%
\end{lemma}

\begin{proof} 
Upon using $n=h_n^{-3/2}/8$,  the transformed kernel can be rewritten as
\[
\hat K_n^{\GUE}(x,y) =  \frac{1}{2\sqrt{h_n}} \frac{a_n(x,y)-a_n(y,x)}{x-y}
\]
(where the singularities at $x=y$ are removable) with 
\[
a_n(x,y) := \frac{\phi_n(\mu_n + \sigma_n x)}{2^{1/2}h_n^{1/8}}\frac{ \phi_{n-1}(\mu_n + \sigma_n y)}{2^{1/2}h_n^{1/8}}.
\]
According to Corollary~\ref{cor:Hermite}, the transformed oscillator wave function expands as
\[
\frac{\phi_n(\mu_{n+1/2} + \sigma_{n+1/2}\theta)}{2^{1/2}h_{n+1/2}^{1/8}} = \Ai(\theta) \sum_{k=0}^m p_k(\theta) h_{n+1/2}^k + \Ai'(s) \sum_{k=0}^m q_k(\theta) h_{n+1/2}^k  + h_{n+1/2}^{m+1} O(e^{-\theta}),
\]
which is uniformly valid for $\theta$ being bounded from below. If we change variables so that we evaluate $\phi_n(\xi)$ and $\phi_{n-1}(\xi)$ both at a point of the form
\[
\xi = \mu_n + \sigma_n t,
\]
a  re-expansion in the  variable $t$ (arguing as in the proof of Theorem~\ref{thm:localExpan} for the tail),  using~$h_n$ as the expansion parameter, yields, by  the symmetry of $n\pm 1/2$ w.r.t. $n$, that
\begin{align*}
\frac{\phi_n(\mu_n + \sigma_n t)}{2^{1/2}h_n^{1/8}} &= \Ai(t) \sum_{k=0}^{2m+2} \hat p_k(t) h_{n}^{k/2} + \Ai'(t) \sum_{k=0}^{2m+1} \hat q_k(t) h_{n}^{k/2}  + h_{n}^{m+3/2} O(e^{-t}),\\*[2mm]
\frac{\phi_{n-1}(\mu_n + \sigma_n t)}{2^{1/2}h_n^{1/8}}  &= \Ai(t) \sum_{k=0}^{2m+2} (-1)^k \hat p_k(t) h_{n}^{k/2} + \Ai'(t) \sum_{k=0}^{2m+1} (-1)^k \hat q_k(t) h_{n}^{k/2}  + h_{n}^{m+3/2} O(e^{-t}),
\end{align*}
with certain polynomials $\hat p_{k}, \hat q_{k} \in \Q[t]$. These expansions are uniformly valid for $t\geq t_0$ as $h_n\to 0^+$. The first instances of the polynomial coefficients are
\begin{equation}\label{eq:GUEkernelPolynomials}
\begin{gathered}
\hat p_0(t) = 1, \quad\hat p_1(t) = 0,\quad  \hat p_2(t) =\frac{3 t}{10},\quad \hat p_3(t) = - \frac{t^3}{5}-\frac{3}{10},\\*[1mm]
\quad \hat p_4(t) = \frac{t^5}{50}-\frac{27 t^2}{280}, \quad \hat p_5(t) = \frac{3 t^4}{35},\\*[2mm]
\hat q_0(t) = 0,\quad \hat q_1(t) = -1,\quad \hat q_2(t) = \frac{t^2}{5}, \quad \hat q_3(t) =  \frac{3 t}{10},\\*[1mm]
\quad \hat q_4(t) = \frac{t^3}{70}+\frac{69}{140},
\quad \hat q_5(t) = -\frac{t^5}{50}-\frac{69 t^2}{280}.
\end{gathered}
\end{equation}
Using this and similarly obtained expressions for larger indices, a routine calculation shows, at least in the range $m\leq m_*$, that  in evaluating
\[
\frac{a_n(x,y)-a_n(y,x)}{x-y}
\]
the even powers of $h_n^{1/2}$ cancel while the positive odd powers generate polynomial factors of the Airy terms $\Ai^{(\kappa)}(x) \Ai^{(\lambda)}(y)$ which belong to $\Q[x,y]$ and which are divisible by $x-y$. Finally, dividing by $2h_n^{1/2}$ proves, first, the claim about the structure of the expansion kernels. Next, the expressions in \eqref{eq:GUEkernelPolynomials} allow us to straightforwardly compute $K_1$, $K_2$. Since all expansions can repeatedly be differentiated under the same conditions while preserving their uniformity, the same holds for the resulting expansion of the kernel.
\end{proof}

\subsection{Simplifying transformation}\label{subsect:simplifyGUE} When studying the $\beta=2$ hard-to-soft edge transition, we showed in \cite[Lemma~3.5]{arxiv:2301.02022} that the expansion kernels are considerably simplified if a certain {\em nonlinear} transformation is applied.
Proceeding analogously here, we observe from \eqref{eq:ts} and~\eqref{eq:sHer} that the scaling and expansion parameters of  Lemma~\ref{lem:GUEkernexpan} are uniquely constructed in such a way that, with $h=h_n$ and  $u=2n$,
\begin{equation}\label{eq:GUEnaturalScaling}
\mu_n +\sigma_n t = u^{1/2}(z_*+2h t), \qquad u^{2/3} \zeta = t + O(h),
\end{equation}
where the variable $\zeta$ is defined as in Theorem~\ref{thm:genHer}(II) and $z_*=1$ denotes the position of the right turning point of the underlying Weber differential equation.
It is precisely that $O(h)$ error term, induced by choosing a {\em linear} scaling transform in the first place, which generates a spray of additional terms in the expansion kernels.\footnote{These terms are, effectively, generated by Taylor expanding the factors $\Ai(u^{2/3}\zeta)$, $\Ai'(u^{2/3}\zeta)$ as in in the proof of Theorem~\ref{thm:localExpan}. On the other hand,  with the variable $s$ in \eqref{eq:psihGUE}, these factors are simply   $\Ai(s)$, $\Ai'(s)$.} Hence, we can reverse that effect by introducing the new variable 
\begin{equation}\label{eq:psihGUE}
s = u^{2/3} \zeta = t +\frac{t^2}{5}h -\frac{8t^3}{175}h^2 + \frac{148t^4}{7875}h^3 + \cdots \quad(|h t|<1),
\end{equation}
where the power series is taken from \eqref{eq:sHer}. In doing so, the next lemma significantly reduces the complexity of the resulting expressions.

\begin{lemma}\label{lem:KhGUEexpan} For $h>0$ and the first expansion kernels $K_1$, $K_2$, $K_3$ from Lemma~\ref{lem:GUEkernexpan}, we consider the kernel
\[
K_{(h)}(x,y) := K_{\Ai}(x,y) + K_1(x,y) h + K_2(x,y)h^2 + K_3(x,y)h^3
\]
and the transformation $s = \psi_h^{-1}(t)$ defined in \eqref{eq:psihGUE}, which maps
the interval 
\[
\rho_0 h^{-1} < t < \rho_1 h^{-1} \quad\text{with $\rho_0 = -0.57903\cdots$, $\rho_1 = 0.46124\cdots$}
\]
monotonically increasing onto $-h^{-1}/2 < s < h^{-1}/2$. The induced  transformed kernel
\begin{subequations}\label{eq:auxkernelexpanGUE}
\begin{align}
\tilde K_{(h)}(x,y) &:= \sqrt{\psi_h'(x)\psi_h'(y)} \, K_{(h)}(\psi_h(x),\psi_h(y))\\
\intertext{has the expansion}
\tilde K_{(h)}(x,y) &= K_{\Ai}(x,y) + \tilde K_1(x,y)h + \tilde K_2(x,y)h^2 + \tilde K_3(x,y)h^3 + h^4 \cdot O\big(e^{-(x+y)}\big),
\end{align}
which is uniformly valid in $s_0 \leq x,y < h^{-1}/2$ as $h\to 0^+$, $s_0$ being a fixed real number. The three expansion kernels are
\begin{align}
\tilde K_1 &= \frac{3}{10}(\Ai\otimes \Ai' + \Ai'\otimes \Ai)\label{eq:K1tilde},\\*[2mm]
\tilde K_2 &= \frac{117\Ai\otimes\Ai}{175} -\frac{81(\Ai\otimes\Ai'''+\Ai'''\otimes\Ai)}{280} -\frac{159(\Ai'\otimes\Ai''+\Ai''\otimes\Ai')}{1400},\\*[2mm]
\tilde K_3 &= -\frac{3713 (\Ai\otimes \Ai'' + \Ai''\otimes \Ai)}{2250} +\frac{1171(\Ai\otimes \Ai^{\rm V} + \Ai^{\rm V}\otimes \Ai)}{3600}  -\frac{3743\Ai'\otimes \Ai'}{7875} \\*[1mm]
& \qquad  + \frac{3359(\Ai'\otimes \Ai^{\rm IV} + \Ai^{\rm IV}\otimes \Ai')}{42000}  +\frac{271(\Ai''\otimes \Ai''' + \Ai'''\otimes \Ai'')}{1800}.\notag
\end{align}
\end{subequations}
Preserving uniformity, the kernel expansion can repeatedly be differentiated w.r.t. $x$, $y$.
\end{lemma}

\begin{proof} Reversing the power series \eqref{eq:psihGUE} gives
\[
t = s-\frac{s^2}{5}h +\frac{22 s^3}{175}h^2 -\frac{823 s^4}{7875}h^3 + \cdots,
\]
which is uniformly convergent for $s_0 \leq s < h^{-1}/2$ (if $h$ is small enough such that $s_0 > -h^{-1}/2$) since then $|ht| \leq \rho_0 < 1$. Taking the expressions for $K_1$, $K_2$ given in \eqref{eq:K2GUE}, and the unwieldy one for the kernel $K_3$, which is not shown here, a routine calculation with truncated power series gives formula \eqref{eq:K1tilde} for $\tilde K_1$, and the expansion kernels 
\[
\begin{gathered}
\tilde K_2(x,y) = \frac{9}{100} \Ai(x)\Ai(y) -\frac{3 (53 x+135 y)}{1400}\Ai(x)\Ai'(y) -\frac{3 (135 x+53 y)}{1400} \Ai'(x)\Ai(y),\\*[2mm]
\tilde K_3(x,y) = -\frac{1787 (x+y)}{9000}\Ai(x)\Ai(y) + \frac{10077 x^2+18970 x y+40985 y^2}{126000}\Ai(x)\Ai'(y)\\*[1mm]
 \qquad + \frac{40985 x^2+18970 x y+10077 y^2}{126000}\Ai'(x)\Ai(y) -\frac{979}{6300} \Ai'(x)\Ai'(y)
\end{gathered}
\]
-- much simpler expressions than those for $K_1, K_2, K_3$ (e.g., the degree of the polynomial coefficients in $K_3$ is reduced from seven to two).

The Airy differential equation $\xi \Ai(\xi) = \Ai''(\xi)$ implies the replacement rule
\begin{equation}\label{eq:airyrule}
\xi^j \Ai^{(k)}(\xi) = \xi^{j-1} \Ai^{(k+2)}(\xi) - k \xi^{j-1} \Ai^{(k-1)}(\xi) \qquad (j\geq 1, \; k\geq 0)
\end{equation}
which, if repeatedly applied to a kernel of the given structure, allows us to absorb any powers of $x$ and $y$ into higher-order derivatives of $\Ai$. This process yields the asserted form of $\tilde K_2$ and $\tilde K_3$, which will be the preferred form in the course of the calculations in Section~\ref{sect:functionalformGUE}.

Since we stay within the range of uniformity of the power series expansions, and calculations with truncated powers series are amenable to repeated differentiation, the result now follows from the bounds in \cite[§2.1]{arxiv:2301.02022}.
\end{proof}

\subsection{Proof of the general form of the expansion}\label{sect:generalformGUE}

By Lemma~\ref{lem:GUEkernexpan} and \cite[Theorem~2.1]{arxiv:2301.02022},  we get for the change of variable $x=\mu_n + \sigma_n t$
\begin{align*}
E_2(n;x) &= \det(I - K_n^{\GUE})|_{L^2(x,\infty)}  \\*[2mm] 
&= \det(I - \hat K_n^{\GUE})|_{L^2(t,\infty)}
= F_2(t) + \sum_{j=1}^m E_{2,j}(t) h_n^j + h_n^{m+1} O(e^{-2t}),
\end{align*}~\\*[-2mm]
uniformly for $t\geq t_0$ as $h_n\to0^+$; preserving uniformity, this expansion can be repeatedly differentiated w.r.t. the variable $t$. By \cite[Theorem~2.1]{arxiv:2301.02022}, the $E_{2,j}(t)$ are certain smooth functions that can be expressed in terms of traces of integral operators. This finishes the proof of \eqref{eq:GUE}.

\subsection{Functional form of the expansion terms}\label{sect:functionalformGUE}

Instead of calculating $E_{2,j}$ directly from the trace formulas in \cite[Theorem~2.1]{arxiv:2301.02022} applied to the increasingly unwieldy expressions for $K_1, K_2, \ldots$  as in (\ref{eq:K2GUE}), we will first compute the corresponding functions $\tilde E_{2,j}$ induced by the much simpler kernels $\tilde K_1$, $\tilde K_2$, $\tilde K_3$ in (\ref{eq:auxkernelexpanGUE}).

This computation is based on a powerful generalization of the Tracy--Widom theory \cite{MR1257246}, which was suggested by  Shinault and Tracy \cite{MR2787973} and which we turned into a general algorithm in our work on the $\beta=2$ hard-to-soft transition limit \cite[Appendix~B]{arxiv:2301.02022}. In Appendix~\ref{app:airykernel}, we recall the results of that work when
applied to kernels of the form $\tilde K_1$, $\tilde K_2$, $\tilde K_3$ with general coefficients. Now, a direct application of formulas \eqref{eq:K1K3} and \eqref{eq:G1G3nice}  to (\ref{eq:auxkernelexpanGUE}) yields
\begin{align*}
\tilde E_{2,1}(t) &= -\frac{3}{10} F''(t), \\*[1mm]
\tilde E_{2,2}(t) &= -\frac{141}{350}F'(t) +\frac{39}{175} t F''(t) + \frac{9}{200} F^{(4)}(t),\\*[3mm]
\tilde E_{2,3}(t) &= \frac{2216 }{7875}t F'(t)-\frac{568 }{2625} t^2F''(t)+\frac{10403}{31500}F'''(t)-\frac{117 }{1750}tF^{(4)}(t)-\frac{9 }{2000}F^{(6)}(t). 
\end{align*}~\\*[-1mm]
The relation between $E_{2,1}(t)$, $E_{2,2}(t)$, $E_{2,3}(t)$ and their counterparts with a tilde is established in Lemma~\ref{lem:KhGUEexpan}: with $t$ being fixed, $h$ sufficiently small and $s=\psi_h^{-1}(t)$, we get by applying \cite[Theorem~2.1]{arxiv:2301.02022} 
\begin{multline*}
F_2(t) + E_{2,1}(t) h + E_{2,2}(t)h^2 + E_{3,2}(t)h^3 + O(h^4)\\*[1mm]
= \det\big(I-K_{(h)}\big)|_{L^2(t, \rho_1 h^{-1})}  = \det\big(I-\tilde K_{(h)}\big)|_{L^2(s, h^{-1}/2)}\\*[1mm]
 = F_2(s) + \tilde E_{2,1}(s) h + \tilde E_{2,2}(s)h^2 + \tilde E_{2,3}(s)h^3 + O(h^4),
\end{multline*}
where we have absorbed the exponentially small error contributions $e^{-\rho_1 h^{-1}} O(e^{-t})$ and $e^{-h^{-1}/2} O(e^{-s})$
into the $O(h^4)$ error term.
Upon inserting the power series \eqref{eq:psihGUE},
followed by  Taylor expansions and comparing coefficients, we obtain the expressions displayed for $E_{2,1}$, $E_{2,2}$, $E_{2,3}$ in~\eqref{eq:F22}. This finishes the proof of Theorem~\ref{thm:softGUE}.

\subsection{Comments on prior work}\label{subsect:previousGUE}
We discuss the relevant prior work in chronological order.

\subsubsection{The case $m=1$}\label{subsect:ChoupGUE} In 2005, Garoni, Forrester and Frankel \cite[Eq.~(72)]{MR2178598} found the diagonal term $K_1(x,x)$ of \eqref{eq:K1GUE} as the finite-size correction to the one-point function at the soft edge of the GUE.

In 2006, upon establishing  formula \eqref{eq:K1GUE} for $K_1(x,y)$, Choup \cite[Eq.~(3.5)]{MR2233711} gave an incomplete\footnote{\label{fn:Choup}The argument given by Choup \cite[pp.~7--8]{MR2233711} does not justify the  error term in \cite[Thm~1.2]{MR2233711}: it misses the subtle technical point that, in general,  uniform kernel bounds do not  lift to trace class, see  \cite[§3]{MR3025686} and \cite[§2]{arxiv:2301.02022}. A later, different proof offered by Choup  \cite[pp.~7]{MR2406805} suffers a similar fate: the Neumann series \cite[Eq.~(2.17)]{MR2406805} does  not, as stated, converge in trace class since the identity operator acting on $L^2(t_0,\infty)$ is not trace class.}  proof that 
\begin{equation}\label{eq:GUEChoup}
E_2\big(n;\mu_n + \sigma_n t\big) = F_2(t) + F_2(t) \det\big(I-h_n(I - K_{\Ai})^{-1} K_1 \big)\big|_{L^2(t,\infty)} + O(n^{-1}),
\end{equation}
 uniformly  for bounded $t$. Initially,  by expressing the Fredholm determinant of a finite-rank perturbation of the identity  in terms of a matrix determinant, and later, by expanding the resolvent kernel, Choup 
inferred from this the expansion
\[
E_2\big(n;\mu_n + \sigma_n t\big) = F_2(t) + E_{2,1}(t) h_n + O(n^{-1}),
\]
see \cite[Theorem~1.4]{MR2233711} and \cite[Theorem~1.3]{MR2406805},  representing 
\begin{equation}\label{eq:E21Choup}
E_{2,1}/F_2 = -\frac{1}{5}\big(2w_1 - 3u_2 +3 v_0 + u_1 v_0 - u_0 v_1 + u_0v_0^2 -u_0^2 w_0\big)
\end{equation}
in terms of the following quantities (the forms involving the Hastings--McLeod solution $q(\xi)$ of Painlevé II follow from the Tracy--Widom theory \cite{MR1257246}):
\begin{align*}
u_k(t) &= \langle(I - K_{\Ai})^{-1}\Ai, \xi^k \Ai \rangle_{L^2(t,\infty)} = \int_t^\infty q(\xi)\, \xi^k \Ai(\xi)\,d\xi,\\*[1mm]
v_k(t) &= \langle(I - K_{\Ai})^{-1}\Ai, \xi^k \Ai' \rangle_{L^2(t,\infty)}= \int_t^\infty q(\xi) \,\xi^k \Ai'(\xi)\,d\xi,\\*[1mm]
w_k(t) &= \langle(I - K_{\Ai})^{-1}\Ai', \xi^k \Ai' \rangle_{L^2(t,\infty)}= \int_t^\infty q'(\xi)\,\xi^k \Ai'(\xi)\,d\xi + u_0(t)v_k(t).
\end{align*}
There are basically two ways to simplify Choup's findings to our expression \eqref{eq:F21} for $E_{2,1}$. First, when recursively applying the replacement rule \eqref{eq:airyrule} to trade the multiplication with powers of $\xi$ for higher order derivatives of the Airy function,  we can rewrite \eqref{eq:E21Choup} in the form (see \eqref{eq:ujk} for the definition of $u_{jk}$)
\[
E_{2,1}/F_2 = -\frac{1}{5} \big(7 u_{01}-3 u_{04}+2 u_{13}+u_{00}^2-u_{00}u_{03} +u_{01} u_{02} - u_{00}^2u_{11}+u_{00} u_{01}^2  \big)
\]
-- an expression which the algorithm developed in \cite[Appendix~B]{arxiv:2301.02022} then readily reduces to~\eqref{eq:F21}.
Second, as we have already sketched in \cite[§3.4]{arxiv:2301.02022}, the expansion \eqref{eq:GUEChoup} allows us to deduce the expression \eqref{eq:F21} rather directly -- without harnessing the power of the Tracy--Widom theory at all; see also \eqref{eq:detexpand2} and Remark~\ref{rem:u01}.

\subsubsection{The case $m=0$}

 In 2012, Johnstone and Ma \cite{MR3025686} studied the convergence rates of the Tracy--Widom limit laws for the GUE and GOE, obtaining an optimal $O(n^{-2/3})$ rate for suitable specific centering and scaling constants. 
By fully addressing the technical difficulties in lifting kernel bounds to trace class bounds, emphasizing differentiability, they proved \cite[Theorem~1]{MR3025686}
\[
E_2\big(n;\mu_n + \sigma_n t\big) = F_2(t) + n^{-2/3} O(e^{-t})
\]
to be uniformly valid as $n\to\infty$ when $t$ is bounded from below. Except for a slightly weaker tail bound, this result is the case $m=0$ of our Theorem~\ref{thm:softGUE}, and their proof uses a technique that is similar to ours, see Remark~\ref{rem:Her_m01}.

\section{Expansion at the soft edge of the LUE and complex Wishart ensembles}\label{sect:LUE}

Proceeding as in Section~\ref{sect:GUE}, we consider the $L^2$-orthonormal system of Laguerre `wave functions' \/(with $L_n^{(\alpha)}$ denoting the generalized Laguerre polynomial with parameter $\alpha>-1$, as defined in \cite[§18.3]{MR2723248})\footnote{Note that, even though the $\phi_n^{(\alpha)}$ belong to $L^2(0,\infty)$ if and only if $\alpha>-1$, the functions themselves are well-defined for $n+\alpha+1 > 0$ (that is, $p>-1$ if $\alpha=p-n$).} 
\begin{equation}\label{eq:LaguerrePhiN}
\phi_n^{(\alpha)}(x) := \sqrt{\frac{n!}{\Gamma(n+\alpha+1)}} x^{\alpha/2} e^{-x/2} (-1)^n L_n^{(\alpha)}(x) \quad (0<x<\infty)
\end{equation}
and the induced finite-rank projection kernel (cf. \cite[§§5.1.2/5.4.1]{MR2641363})
\begin{equation}\label{eq:KnLUE}
K_n^{\LUE,\alpha}(x,y) := \sum_{k=0}^{n-1} \phi_k^{(\alpha)}(x)\phi_k^{(\alpha)}(y) = \sqrt{n(n+\alpha)} \frac{\phi_n^{(\alpha)}(x)\phi_{n-1}^{(\alpha)}(y)-\phi_{n-1}^{(\alpha)}(x)\phi_n^{(\alpha)}(y)}{x-y}.
\end{equation}
The distribution of the largest eigenvalue of the LUE is then given by
\[
E_{\LUE,\alpha}(n;x) = \det\big(I- K_n^{\LUE,\alpha}\big)\big|_{L^2(x,\infty)}.
\]
However, with the complex Wishart distribution and its symmetry \eqref{eq:WishartSymmetry} in mind, we prefer a setup that extends these definitions to general real $n>0$ and $p =  n + \alpha >0$. As explained in Section~\ref{subsect:symmetric}, using Whittaker's confluent hypergeometric function $W_{\kappa,\mu}$, we can extend the definition of the wave functions
\[
\phi_{n,p}(x) :=\phi^{(p-n)}_n(x) 
\]
to all real $n, p >-1$ -- an extension, which now enjoys the symmetry (a consequence of Kummer's transformation for confluent hypergeometric functions, see~\eqref{eq:Kummer})
\[
\phi_{n,p}(x) = \phi_{p,n}(x).
\]
By likewise extending the kernel $K_{n,p}^{\LUE}(x,y):= K_n^{\LUE,p-n}(x,y)$ from integer $n$ to real $n$, we define the kernel\footnote{Because of differentiability and exponential decay the kernel $K_{n,p}^{\LUE}$ is easily seen to induce 
a trace class operator on the spaces $L^2(t,\infty)$, $t>0$.}
\begin{equation}\label{eq:KnpLUE}
K_{n,p}^{\LUE}(x,y) := \sqrt{n p} \,\frac{\phi_{n,p}(x)\phi_{n-1,p-1}(y)-\phi_{n-1,p-1}(x)\phi_{n,p}(y)}{x-y} \quad (n,p>0),
\end{equation}
which satisfies the symmetry 
\[
K_{n,p}^{\LUE}(x,y) = K_{p,n}^{\LUE}(x,y).
\]
In this fashion, the largest eigenvalue distribution $E_{\LUE,p-n}(n;x)$, and $E_2(n,p;x)$ for that matter, generalizes from integer $n$ to
\begin{equation}\label{eq:E2np}
E_2(n,p;x) := \det\big(I- K_{n,p}^{\LUE}\big)\big|_{L^2(x,\infty)}
\end{equation}
which  satisfies the symmetry relation \eqref{eq:WishartSymmetry} now for any real $n,p>0$. We note, however, that 
the probabilistic interpretation of $E_2(n,p;x)$  is, so far, only known to be valid  if (i) $n$ is an integer and $p> n-1$ or vice versa (the LUE case with real $\alpha=p-n>-1$) or (ii) $n,p$ are integers (the complex Wishart case).

In pursuing the strategy of Section~\ref{subsect:strategy}, we took particular care that the centering, scaling, and expansion parameters $\mu, \sigma, h$ are symmetric in $n$ and $p$ and that the ratio $p/n$ is encoded by yet another symmetric parameter $\tau$, such that the coefficients of all the expansion terms are rational polynomials in $\tau$. This way, we are led to the following result.

\begin{theorem}\label{thm:softLUE} For any real $n, p >0$, consider the symmetric parameters\/\footnote{In view of the later discussion, leading to \eqref{eq:LUEnaturalScaling}, these are the unique `natural' choices of parameters.}
\begin{subequations}\label{eq:LUEscaling}
\begin{gather}
\mu_{n,p} = \big(\sqrt{n}+\sqrt{p}\,\big)^2,\quad \sigma_{n,p}= \big(\sqrt{n}+\sqrt{p}\,\big) \bigg(\frac{1}{\sqrt{n}} + \frac{1}{\sqrt{p}}\bigg)^{1/3},\\*[2mm]
\tau_{n,p} = \frac{4}{\big(\sqrt{n}+\sqrt{p}\,\big) \left(\dfrac{1}{\sqrt{n}} + \dfrac{1}{\sqrt{p}}\right)},\quad 
h_{n,p} = \frac{\sigma_{n,p}}{\tau_{n,p}\cdot\mu_{n,p}}= \frac{1}{4} \bigg(\frac{1}{\sqrt{n}} + \frac{1}{\sqrt{p}}\bigg)^{4/3},
\end{gather}
\end{subequations}
which satisfy $h_{n,p}\asymp (n\wedge p)^{-2/3}$ and $0< \tau_{n,p}\leq 1$, where $\tau_{n,p}=1$ corresponds to $n=p$.
Then there holds the large matrix expansion of the $\LUE$ at the soft edge,
\begin{equation}\label{eq:LUE}
E_2\big(n,p;\mu_{n,p} + \sigma_{n,p} t\big) = F_2(t) + \sum_{j=1}^m E_{2,j;\tau_{n,p}}(t) h_{n,p}^j + h_{n,p}^{m+1}\cdot O(e^{-2t}),
\end{equation}
which is uniformly valid as $h_{n,p}\to0^+$ when $t\geq t_0$ and $\tau_{n,p}\geq \tau_0 >0$,
with~$m$ being any fixed integer in the range\footnote{For the  reasons explained in Remark~\ref{rem:mstar}, we have to impose a bound $m_*$ on the number of expansion terms (we have explicitly checked the cancellation and divisibility properties in the proof of Lemma~\ref{lem:LUEkernexpan},  up to $m_*=10$); we expect this bound to be removed by extending the analysis of \cite{arXiv:2309.06733}.}  $0\leq m\leq m_*$ and  $t_0,\tau_0$  any fixed real numbers. Preserving uniformity, the expansion can be repeatedly differentiated w.r.t. the variable $t$.
Here, the $E_{2,j;\tau}(t)$ are certain smooth functions -- the first instances are\/\footnote{Figure~\ref{fig:LUE} plots the density correction terms $E_{2,1;\tau}'(t)$, $E_{2,2;\tau}'(t)$, $E_{2,3;\tau}'(t)$ for various values of $0\leq \tau \leq 1$; note that (if we break the symmetry by considering $p\geq n$) the boundary cases correspond to $p=\infty$, $p=n$.}
\begin{subequations}\label{eq:F22LUE}
\begin{align}
E_{2,1;\tau}(t) &= -\frac{2 \tau -1}{5} t^2 F_2'(t) + \frac{\tau -3}{10} F_2''(t),\label{eq:F21LUE}\\*[1mm]
E_{2,2;\tau}(t) &=\bigg(\frac{\tau ^2+94 \tau -141}{350} +\frac{43 \tau ^2-18 \tau -8 }{175}t^3\bigg) F_2'(t)\\*[0.5mm]
&\qquad\qquad  -\bigg(\frac{4 \tau ^2+26 \tau -39}{175}t  - \frac{(2 \tau -1)^2}{50} t^4\bigg)F_2''(t)\notag\\*[0.5mm]
&\qquad\qquad\qquad -\frac{(\tau -3)(2 \tau -1) }{50}t^2 F_2'''(t) + \frac{(\tau -3)^2}{200} F_2^{(4)}(t),\notag\\*[1mm]
E_{2,3;\tau}(t) &= -\bigg(\frac{ 2(4 \tau ^3-161 \tau ^2+1108 \tau -1108)}{7875}t\\*[0.5mm]
&\qquad\quad +\frac{1384 \tau ^3-551 \tau ^2-212 \tau -148}{7875}t^4\bigg)F'_2(t)\notag\\*[0.5mm]
&\hspace*{-2.5cm} + \bigg(\frac{ 12 \tau ^3-77 \tau ^2+328 \tau -265}{1050}t^2
-\frac{(2 \tau -1)(43 \tau ^2-18 \tau -8)}{875}t^5\bigg) F_2''(t)\notag\\*[0.5mm]
& +\bigg(\frac{61 \tau ^3+1351 \tau ^2-10403 \tau +10403 }{31500}\notag\\*[0.5mm]
&\qquad\quad +\frac{59 \tau ^3-51 \tau ^2-162 \tau +102}{1750} t^3-\frac{(2 \tau -1)^3 }{750} t^6\bigg)F_2'''(t)\notag\\*[0.5mm]
&\qquad-\bigg( \frac{(\tau -3) \left(4 \tau ^2+26 \tau -39\right)}{1750}t
-\frac{(\tau -3) (2 \tau -1)^2 }{500} t^4\bigg)F_2^{(4)}(t)\notag\\*[0.5mm]
&\qquad\qquad\qquad\qquad\qquad -\frac{ (\tau -3)^2 (2 \tau -1)}{1000}t^2 F_2^{(5)}(t) + \frac{(\tau -3)^3 }{6000}F_2^{(6)}(t).\notag
\end{align}
\end{subequations}
\end{theorem}

\begin{figure}[tbp]
\includegraphics[width=0.325\textwidth]{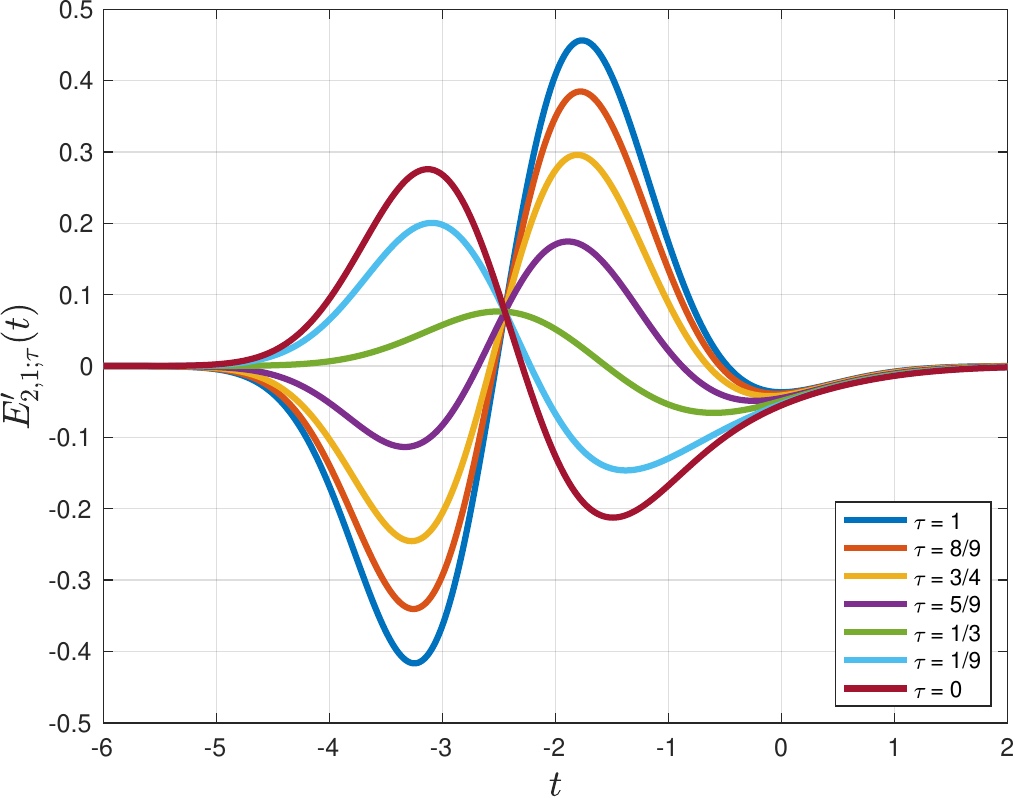}\hfil\,
\includegraphics[width=0.325\textwidth]{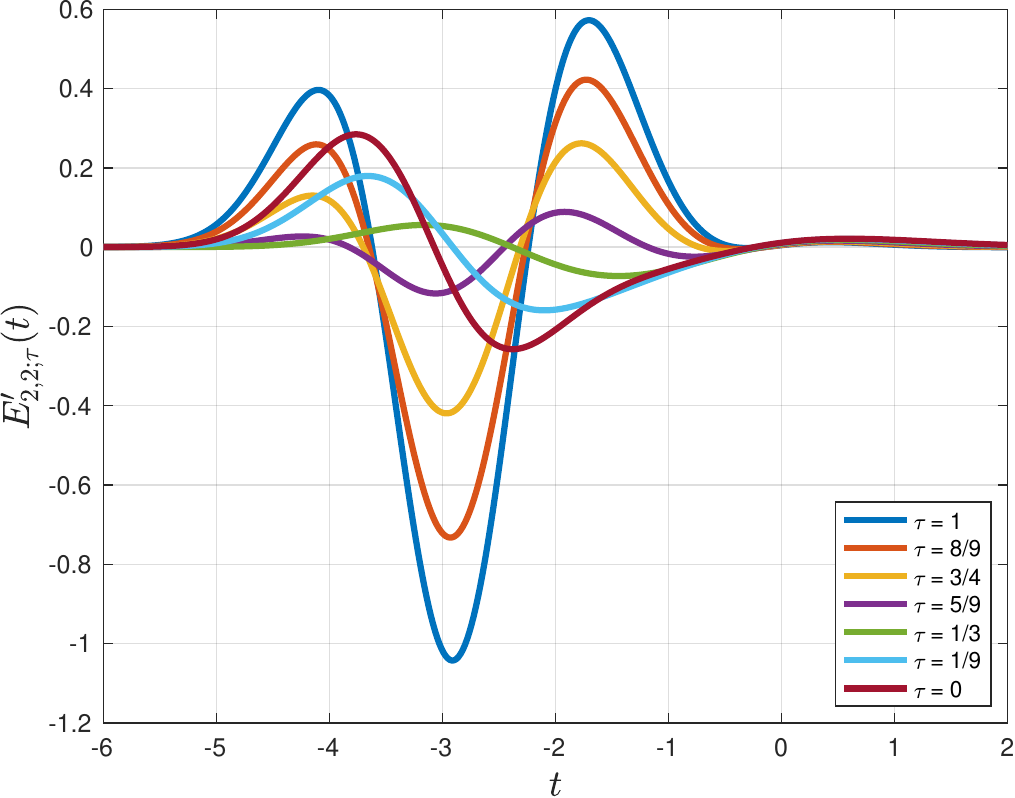}\hfil\,
\includegraphics[width=0.325\textwidth]{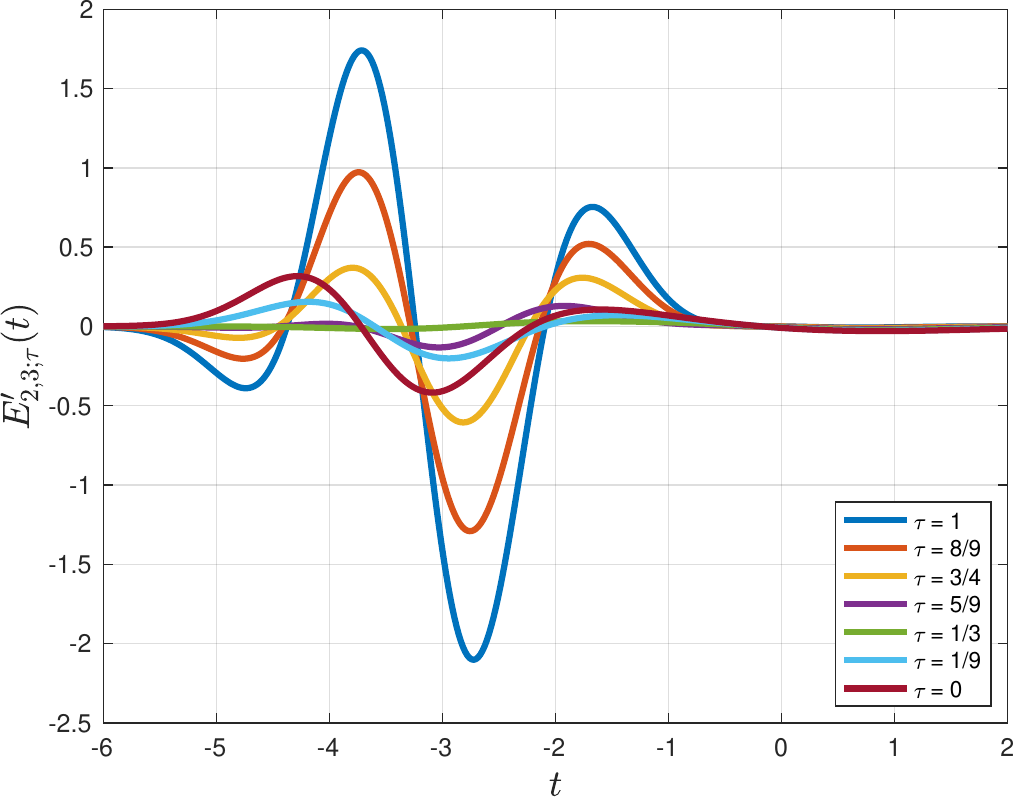}
\caption{{\footnotesize Plots of $E_{2,1;\tau}'(t)$ (left panel), $E_{2,2;\tau}'(t)$ (middle panel) 
and $E_{2,3;\tau}'(t)$ (right panel) for $\tau\in\{1, 8/9, 3/4, 5/9,1/3,1/9,0\}$. In the case $p\geq n$, these values of $\tau$ correspond to the ratios $p=n$, $p=4n$, $p=9n$,  $p= 25 n$, $p\approx 100n$, $p\approx 1000 n$, $p=\infty$. See Figure~\ref{fig:LUESimulation} for a comparison with sampling data in the case $p=4n$.}}
\label{fig:LUE}
\end{figure}

The proof of Theorem~\ref{thm:softLUE} is split into several steps and will be given in Sections~\ref{subsect:kernelLUEexpan}--\ref{sect:functionalformLUE}. We will comment on previous work on the cases $m=0$ and $m=1$ in Section~\ref{subsect:previousLUE}. 

\begin{remark}\label{rem:LUE2GUE} Though we can prove the uniform validity of the expansion \eqref{eq:LUE} only for 
\[
\tau_{n,p}\geq \tau_0 >0
\]
bounded away from zero, numerical experiments (see the example following this remark) show robustness also for very small values of $\tau_{n,p}$. We therefore conjecture that the expansion is, in fact, uniformly valid for all $\tau_{n,p}>0$
-- a conjecture which is also supported by the following observation. Upon comparing Theorem~\ref{thm:softGUE} with Theorem~\ref{thm:softLUE}  (see also Figure~\ref{fig:GUE} and Figure~\ref{fig:LUE}), we get by inspecting the cases $j=1,\ldots,m_*$ that, for given $n$,
\[
h_{n,\infty} := \lim_{p\to\infty} h_{n,p} = h_n,\quad \lim_{p\to\infty} E_{2,j;\tau_{n,p}}(t)  = E_{2,j;0}(t) =  E_{2,j}(t).
\]
On the other hand, in Appendix~\ref{sect:Laguerre2Gauss}, Corollary~\ref{cor:LUE2GUE}, we will establish the limit law
\[
\lim_{p\to\infty} E_2\big(n,p;\mu_{n,p} + \sigma_{n,p} t\big)  = E_2\big(n;\mu_n + \sigma_n t\big),
\]
uniformly in $t$. Thus, by combining these  limit relations between the GUE and the LUE, the expansion \eqref{eq:GUE} implies (at least for $m\leq m_*$) that\footnote{Note that this implies the limit law
\[
\lim_{n\to\infty}\lim_{p\to\infty} E_2\big(n,p;\mu_{n,p} + \sigma_{n,p} t\big) = \lim_{p\to\infty}\lim_{n\to\infty} E_2\big(n,p;\mu_{n,p} + \sigma_{n,p} t\big) = F_2(t),
\]
which is consistent with the limit law, established by El Karoui \cite{karoui2003largest}, when $n, p, p/n \to\infty$.
}
\begin{equation}\label{eq:LUElarge_p}
\lim_{p\to\infty} E_2\big(n,p;\mu_{n,p} + \sigma_{n,p} t\big) = F_2(t) + \sum_{j=1}^m E_{2,j;0}(t) h_{n,\infty}^j + h_{n,\infty}^{m+1}\cdot O(e^{-2t}),
\end{equation}
which is exactly what we obtain from formally passing to the limit $p\to\infty$ in \eqref{eq:LUE} -- an  admissible operation if the conjecture about the uniformity for $\tau_{n,p}>0$ were true.\footnote{A similar observation will be discussed for the 
expansion of the `wave functions' belonging to the Laguerre and Hermite polynomials, in Remark~\ref{rem:Lag2Her2}.}
\end{remark}

As an extreme stress test of the robustness of \eqref{eq:LUE}, for very small $\tau$, 
we compared $p=10^8 n$ (that is, $\tau_{n,p} = 0.00039992\cdots$) for $n=10$ with a statistics of $N=10^9$  draws from the LUE. Except for some different random sampling errors, the result is indistinguishable from Figure~\ref{fig:GUESimulation}, which visualizes the same experiment for the $\GUE$ (which by the preceding remark corresponds to $\tau=0$).

A more modest example with $p=4n$ ($\tau=8/9$) is shown in Figure~\ref{fig:LUESimulation}. It exhibits, at the mode of the distribution, an error of the Tracy--Widom limit law that is noticeably larger than the one in Figure~\ref{fig:GUESimulation}; this can be quantitatively understood
(including the opposite signs of the error) as follows. At the mode of the limit distribution, the first correction term shown in the left panel of Figure~\ref{fig:LUE} is, for $\tau=8/9$, by a factor of $\approx 2.3$ larger in size than for $\tau=0$. Since also $h_{10,40} \approx 1.7 h_{10,\infty}$, the error increases in total by a factor of about $3.9$ in size.

\subsection{Kernel expansion}\label{subsect:kernelLUEexpan}

We start with expanding the induced transformation of the kernel. 

\begin{lemma}\label{lem:LUEkernexpan} With the parameters $\mu_{n,p}, \sigma_{n,p}, h_{n,p}, \tau_{n,p}$ as introduced in \eqref{eq:LUEscaling}, the linear change of variables $s=\mu_{n,p} + \sigma_{n,p}t$ induces the  transformed kernel  
\begin{equation}\label{eq:LUEkernexpan}
\begin{aligned}
\hat K_{n,p}^{\LUE}(x,y) &:= \sigma_{n,p} \, K_{n,p}^{\LUE}(\mu_{n,p} + \sigma_{n,p} x,\mu_{n,p} + \sigma_{n,p} y)\\
 &= K_{\Ai}(x,y) + \sum_{j=1}^m K_j(x,y;\tau_{n,p}) h_{n,p}^j + h_{n,p}^{m+1}\cdot O\big(e^{-(x+y)}\big),
\end{aligned}
\end{equation}
where the expansion is
uniformly valid as $h_{n,p}\to0^+$ when $x,y\geq t_0$ and $\tau_{n,p}\geq \tau_0 >0$,
with~$m$ being any fixed integer in the range $0\leq m\leq m_*$ and  $t_0,\tau_0$  any fixed real numbers. Here, the expansion kernels $K_j$ are certain finite rank kernels of the form
\[
K_j(x,y;\tau) = \sum_{\kappa,\lambda\in\{0,1\}}  p_{j,\kappa\lambda}(x,y;\tau) \Ai^{(\kappa)}(x)\Ai^{(\lambda)}(y)
\]
with polynomial coefficients $p_{j,\kappa\lambda}(x,y;\tau)\in \Q[\tau][x,y]$ -- the first instances are\/\footnote{We refrain from displaying the unwieldy expression for $K_3(x,y;\tau)$ which has polynomial factors of degree $7$ in $x,y$ and degree $3$ in $\tau$.}
\begin{subequations}\label{eq:K2LUE}
\begin{align}
K_1(x,y;\tau) &=  \frac{2\tau-1}{5}\Big((x^2+xy+y^2)\Ai(x)\Ai(y) - (x+y)\Ai'(x)\Ai'(y)\Big)\label{eq:K1LUE}\\*[0mm]
&\qquad  - \frac{\tau-3}{10} \big(\!\Ai(x)\Ai'(y) + \Ai'(x)\Ai(y)\big),\notag\\*[1mm]
K_2(x,y;\tau) &= -\Big(\frac{20 \tau ^2-3 \tau -6}{70}  (x^3+y^3) 
+\frac{114 \tau ^2-64 \tau -9}{350} x y (x+y) \\*[0mm]
&\qquad\quad - \frac{(\tau -3)^2}{100} \Big) \Ai(x)\Ai(y)\notag\\*[0mm]
&\hspace*{-2cm}+ \Big( \frac{(2 \tau -1)^2}{50}  (x^4+x^3 y-x^2 y^2-x y^3-y^4)+\frac{13 \tau ^2-10 \tau -27}{280} x\notag\\*[0mm]
&\qquad\quad    +\frac{51 \tau ^2+34 \tau -261}{1400}y\Big) \Ai(x)\Ai'(y) \notag\\*[0mm]
&\hspace*{-2cm}  - \Big(\frac{(2 \tau -1)^2}{50}  (x^4+x^3 y+x^2 y^2-x y^3-y^4) 
-\frac{51 \tau ^2+34 \tau -261}{1400}x\notag\\*[0mm]
&\qquad\quad -\frac{13 \tau ^2-10 \tau -27}{280} y\Big)\Ai'(x)\Ai(y)\notag\\*[0mm]
 &
 + \Big( \frac{20 \tau ^2-17 \tau +1}{70}  (x^2+y^2)+\frac{43 \tau ^2-18 \tau -8}{175} x y\Big)\Ai'(x)\Ai'(y). \notag
\end{align}
\end{subequations}
Preserving uniformity, the kernel expansion \eqref{eq:LUEkernexpan} can repeatedly be differentiated w.r.t. $x$, $y$.
\end{lemma}

\begin{figure}[tbp]
\includegraphics[width=0.325\textwidth]{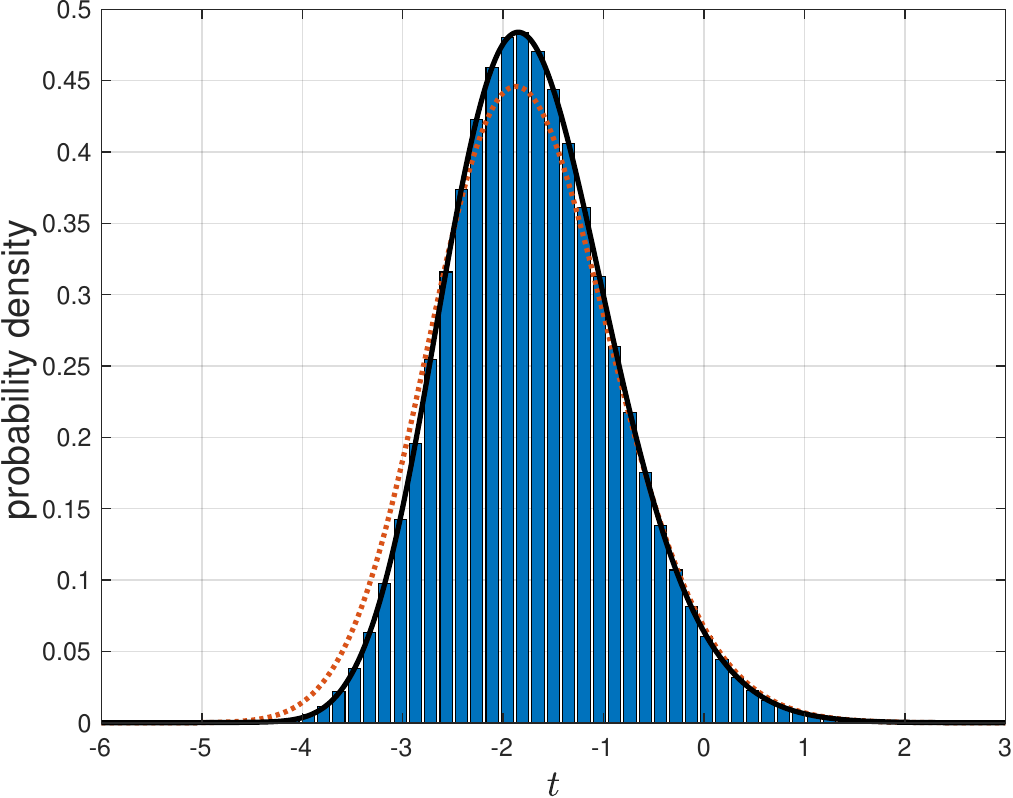}\hfil\,
\includegraphics[width=0.325\textwidth]{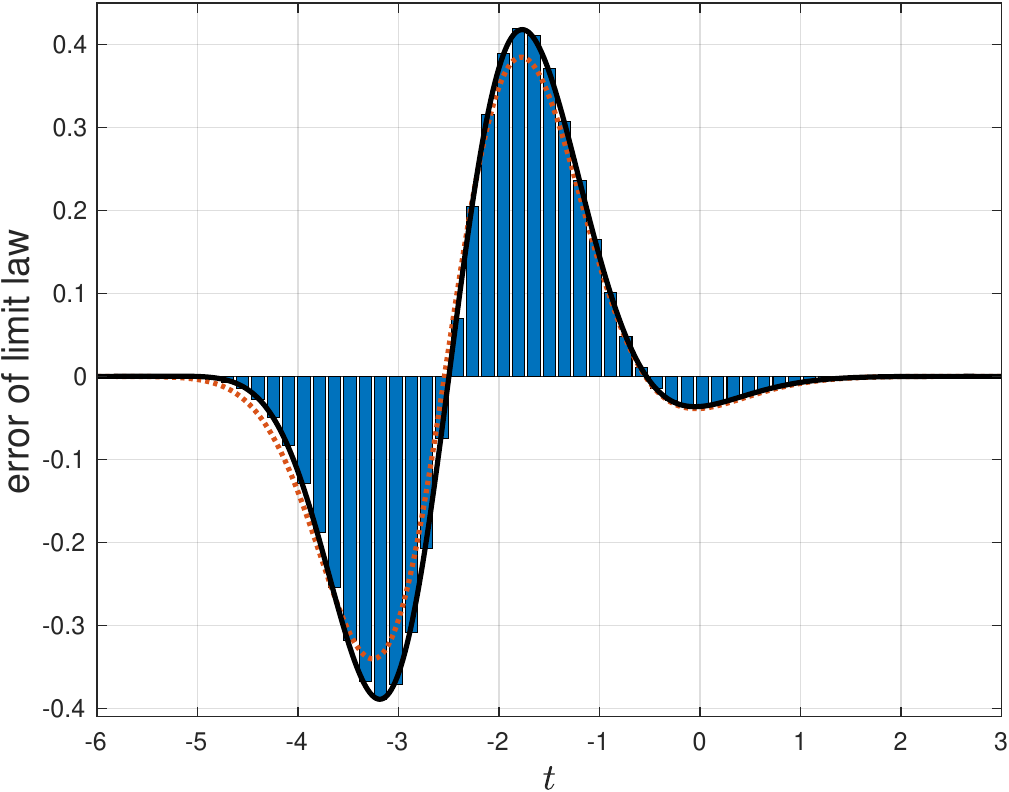}\hfil\,
\includegraphics[width=0.325\textwidth]{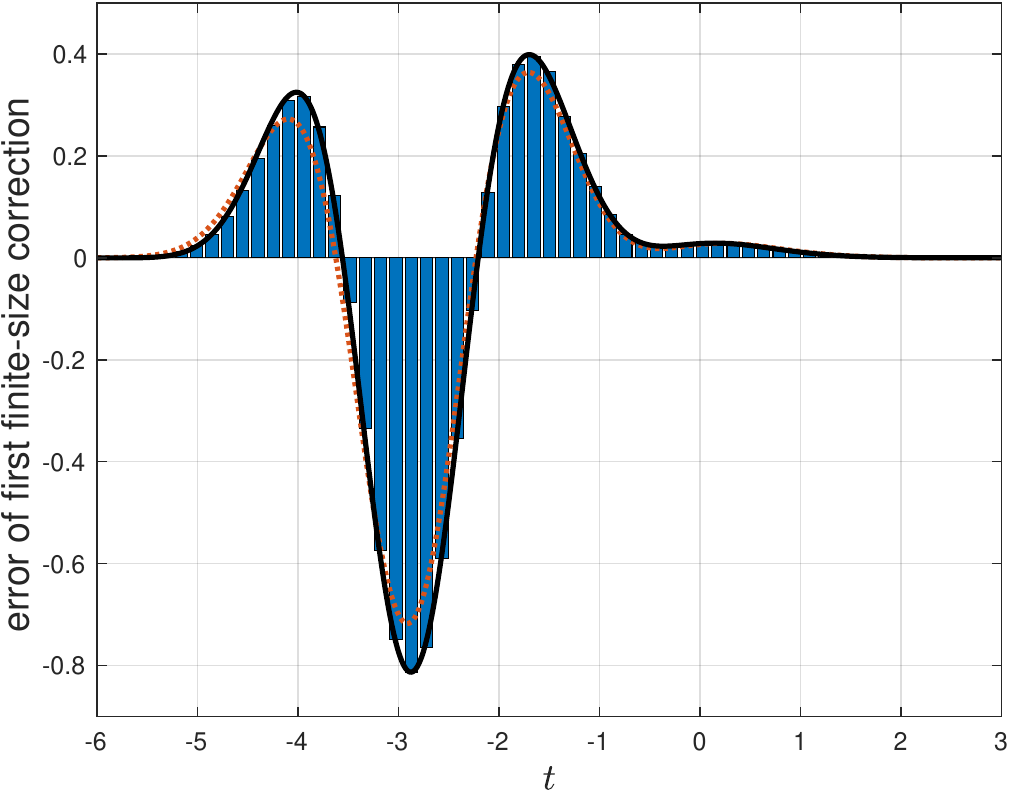}
\caption{{\footnotesize [Expansion terms with a tilde are  `histogram adjusted', see Appendix~\ref{app:histo}.]
Left panel: histogram  of the scaled largest eigenvalues $(\lambda_{\max} - \mu_{n,p})/\sigma_{n,p}$, for $N=10^9$ draws, from the $\LUE$ with  $n=10$, $p=40$, so that $\tau=\tau_{n,p}=8/9$ and $h=h_{n,p}=0.092482\cdots$ (blue bars) vs. the Tracy--Widom limit density $F_2'$ (red dotted line) 
and its second order correction $F_2' + h E_{2,1;\tau}' + h^2 \tilde E_{2,2;\tau}'$ (black solid line). Middle panel: difference, scaled by $h^{-1}$,  between the histogram midpoints and the limit density (blue bars) vs. the correction terms $E_{2,1;\tau}'$ (red dotted line) and $E_{2,1;\tau}' + h \tilde E_{2,2;\tau}'$ (black solid line). Right panel: difference, scaled by $h^{-2}$, between 
the histogram midpoints and the first finite-size correction (blue bars) vs. the correction terms $\tilde E_{2,2;\tau}'$ (red dotted line) and $\tilde E_{2,2;\tau}' + h \tilde E_{2,3;\tau}'$ (black solid line).}}
\label{fig:LUESimulation}
\end{figure}

\begin{proof} Suppressing some explicit dependencies in the notation, we write briefly
\[
h = h_{n,p}, \quad h_\pm = h_{n\pm1/2,p\pm1/2}, 
\]
and in the same fashion $\tau,\tau_\pm,\mu,\mu_\pm,\sigma,\sigma_\pm$.
Upon using $1/\sqrt{np}=2\tau h^{3/2}$  the transformed kernel can be rewritten as
\[
\hat K_{n,p}^{\LUE}(x,y) =  \frac{1}{2\sqrt{h}} \frac{a_{n,p}(x,y)-a_{n,p}(y,x)}{x-y}
\]
(where the singularities at $x=y$ are removable) with 
\[
a_{n,p}(x,y) := \frac{\phi_{n,p}(\mu + \sigma x)}{\tau^{1/2}h^{1/2}}\frac{ \phi_{n-1,p-1}(\mu + \sigma x)}{\tau^{1/2}h^{1/2}}.
\]
According to Corollary~\ref{cor:Laguerre2}, the transformed `wave function' expands as
\[
\frac{\phi_{n,p}(\mu_+ + \sigma_+\theta)}{\tau_+^{1/2}h_+^{1/2}} = \Ai(\theta) \sum_{k=0}^m p_k(\tau_+,\theta) h_+^k + \Ai'(s) \sum_{k=0}^m q_k(\tau_+\theta) h_+^k  + h_+^{m+1} O(e^{-\theta}),
\]
which is uniformly valid for $\theta$ being bounded from below. If we change variables so that we evaluate $\phi_{n,p}(\xi)$ and $\phi_{n-1,p-1}(\xi)$ both at a point of the form
\[
\xi = \mu + \sigma t,
\]
a  re-expansion in the  variable $t$ (arguing as in the proof of Theorem~\ref{thm:localExpan} for the tail),  using~$h$ as the expansion parameter, yields, by the symmetry of $n\pm1/2$, $p\pm 1/2$ w.r.t. $n,p$, that
\begin{align*}
\frac{\phi_{n,p}(\mu + \sigma t)}{\tau^{1/2}h^{1/2}} &= \Ai(t) \sum_{k=0}^{2m+2} \hat p_k(\tau,t) h^{k/2} + \Ai'(t) \sum_{k=0}^{2m+1} \hat q_k(\tau,t) h^{k/2}  + h^{m+3/2} O(e^{-t}),\\*[1mm]
\frac{\phi_{n-1,p-1}(\mu + \sigma t)}{\tau^{1/2}h^{1/2}}  &= \Ai(t) \sum_{k=0}^{2m+2} (-1)^k \hat p_k(\tau,t) h^{k/2} \\*[1mm]
&\qquad\qquad+ \Ai'(t) \sum_{k=0}^{2m+1} (-1)^k \hat q_k(\tau,t) h^{k/2}  + h^{m+3/2} O(e^{-t}),
\end{align*}
with certain polynomials $\hat p_{k}, \hat q_{k} \in \Q[\tau,t]$. These expansions are uniformly valid for $t\geq t_0$ and $\tau\geq \tau_0$ as $h\to 0^+$. The first instances of the polynomial coefficients are
\begin{equation}\label{eq:LUEkernelPolynomials}
\begin{gathered}
\hat p_0(\tau,t) = 1, \quad p_1(\tau,t) = 0, \quad  \hat p_2(\tau,t) =-\frac{\tau-3}{10}t,
\quad \hat p_3(\tau,t) = \frac{2 \tau -1}{5}t^3 - \frac{\tau-3}{10},\\*[1mm]
\hat p_4(\tau,t) = \frac{(2 \tau -1)^2}{50} t^5 +\frac{13 \tau ^2-10 \tau -27}{280}t^2,\quad 
\hat p_5(\tau,t ) = -\frac{20 \tau ^2-3 \tau -6}{70}t^4,\\*[1mm]
\hat q_0(\tau,t) = 0, \quad \hat q_1(\tau,t) = -1, \quad q_2(\tau,t)= -\frac{2 \tau -1}{5}t^2,\quad \hat q_3(\tau,t) =-\frac{\tau-3}{10}t,\\*[1mm]
\hat q_4(\tau,t) = \frac{20 \tau ^2-17 \tau +1}{70} t^3+ \frac{\tau ^2-46 \tau +69}{140},\\*[1mm]
\hat q_5(\tau,t) = -\frac{(2 \tau -1)^2}{50} t^5 -\frac{\tau ^2-46 \tau +69}{280} t^2.
\end{gathered}
\end{equation}
Using this and similarly obtained expressions for larger indices, a routine calculation shows, at least in the range $m\leq m_*$, that  in evaluating
\[
\frac{a_{n,p}(x,y)-a_{n,p}(y,x)}{x-y}
\]
the even powers of $h^{1/2}$ cancel while the positive odd powers generate polynomial factors of the Airy terms $\Ai^{(\kappa)}(x) \Ai^{(\lambda)}(y)$ which belong to $\Q[\tau][x,y]$, being divisible by $x-y$. Finally, dividing by $2h^{1/2}$ proves, first, the claim about the structure of the expansion kernels. Next, the expressions in \eqref{eq:LUEkernelPolynomials} allow us to straightforwardly compute $K_1$, $K_2$. Since all expansions can repeatedly be differentiated under the same conditions while preserving their uniformity, the same holds for the resulting expansion of the kernel.
\end{proof}

\subsection{Simplifying transformation}\label{subsect:simplifyLUE} As in Section~\ref{subsect:simplifyGUE}, the expansion kernels are considerably simplified if a certain {\em nonlinear} transformation is applied.
Proceeding analogously here, we observe from \eqref{eq:ts} and \eqref{eq:sLag} that the scaling and expansion parameters of  Lemma~\ref{lem:LUEkernexpan} are uniquely constructed in such a way that, with  $h=h_{n,p}$, $\tau=\tau_{n,p}$ and $u=n$,
\begin{equation}\label{eq:LUEnaturalScaling}
\mu_{n,p} +\sigma_{n,p} t = u z_* (1+ \tau h t), \qquad u^{2/3} \zeta = t + O(h),
\end{equation}
where the variable $\zeta$ is defined as in Theorem~\ref{thm:genLag}(II) and 
\[
z_*= \left(\sqrt\frac{p}{n}+1\right)^2
\]
denotes the position of the right turning point of the underlying confluent hypergeometric differential equation.\footnote{Note that $u z_* = (\sqrt{n} + \sqrt{p})^2$ is symmetric in $n$ and $p$.} 
Once more, it is precisely that $O(h)$ error term, induced by choosing a {\em linear} scaling transform in the first place, which generates a spray of additional terms in the expansion kernels.\footnote{These terms are, effectively, generated by Taylor expanding the factors $\Ai(u^{2/3}\zeta)$, $\Ai'(u^{2/3}\zeta)$ as in in the proof of Theorem~\ref{thm:localExpan}. On the other hand,  with the variable $s$ in \eqref{eq:psihLUE}, these factors are simply   $\Ai(s)$, $\Ai'(s)$.} 
Hence, we can reverse that effect by introducing the new variable 
\begin{multline}\label{eq:psihLUE}
s = u^{2/3} \zeta = t-\frac{2\tau-1}{5}t^2 h + \frac{43 \tau ^2-18 \tau-8}{175} t^3  h^2 \\*[0mm]
-\frac{ 1384 \tau ^3-551 \tau ^2-212 \tau-148}{7875}t^4 h^3 + \cdots \quad(|h t|<1),
\end{multline}
where the power series is taken from \eqref{eq:sLag}. In doing so, the next lemma significantly reduces the complexity of the resulting expressions.

\begin{lemma}\label{lem:KhLUEexpan} For $h>0$, $0\leq \tau\leq 1$, and the first expansion kernels $K_1$, $K_2$, $K_3$ from Lemma~\ref{lem:LUEkernexpan}, we consider 
\[
K_{(h)}(x,y;\tau) := K_{\Ai}(x,y) + \sum_{j=1}^3 K_j(x,y;\tau) h^j
\]
and the transformation $s = \psi_{h,\tau}^{-1}(t)$ defined in \eqref{eq:psihLUE}, which maps an interval of the form
\[
\rho_{0,\tau} h^{-1} < t < \rho_{1,\tau} h^{-1}, \quad  -1 < \rho_{0,0} \leq \rho_{0,\tau} < 0 < \rho_{1,\tau} \leq \rho_{1,1} < 1
\]
onto $-h^{-1}/2 < s < h^{-1}/2$. The induced transformed kernel
\begin{subequations}\label{eq:auxkernelexpanLUE}
\begin{align}
\tilde K_{(h)}(x,y;\tau) &:= \sqrt{\psi_{h,\tau}'(x)\psi_{h,\tau}'(y)} \, K_{(h)}(\psi_{h,\tau}(x),\psi_{h,\tau}(y);\tau)\\
\intertext{has the expansion}
\tilde K_{(h)}(x,y;\tau) &= K_{\Ai}(x,y) + \sum_{j=1}^3 \tilde K_j(x,y;\tau)h^j + h^4 \cdot O\big(e^{-(x+y)}\big),
\end{align}
which is uniformly valid in $s_0 \leq x,y \leq h^{-1}/2$ and $0\leq \tau\leq 1$ as $h\to 0^+$, $s_0$ being a fixed real number. The three expansion kernels are
\begin{align}
\tilde K_1 &= -\frac{\tau -3}{10} (\Ai\otimes \Ai' + \Ai'\otimes \Ai)\label{eq:K1tildeLUE},\\*[2mm]
\tilde K_2 &= -\frac{3(4 \tau ^2+26 \tau -39)}{175}  \Ai\otimes\Ai + \frac{11 \tau ^2+54 \tau -81}{280}(\Ai\otimes\Ai'''+\Ai'''\otimes\Ai) \\*[1mm]
& \qquad  -\frac{51 \tau ^2-106 \tau +159}{1400}(\Ai'\otimes\Ai''+\Ai''\otimes\Ai'),\notag\\*[2mm]
\intertext{}\notag\\*[-13mm]
\tilde K_3 &= \frac{44 \tau ^3-346 \tau ^2+3713 \tau -3713}{2250} (\Ai\otimes \Ai'' + \Ai''\otimes \Ai) \\*[1mm]
&\qquad -\frac{13 \tau ^3-77 \tau ^2+1171 \tau -1171}{3600}(\Ai\otimes \Ai^{\rm V} + \Ai^{\rm V}\otimes \Ai)\notag\\*[1mm] 
&\qquad\quad -\frac{466 \tau ^3+406 \tau ^2-3743 \tau + 3743}{7875} \Ai'\otimes \Ai' \notag\\*[1mm]
& \qquad\qquad + \frac{583 \tau ^3+553 \tau ^2-3359 \tau +3359}{42000}(\Ai'\otimes \Ai^{\rm IV} + \Ai^{\rm IV}\otimes \Ai')\notag\\*[1mm] 
& \qquad\qquad\quad -\frac{13 \tau ^3-77 \tau ^2+271 \tau -271}{1800}(\Ai''\otimes \Ai''' + \Ai'''\otimes \Ai'').\notag
\end{align}
\end{subequations}
Preserving uniformity, the kernel expansion can repeatedly be differentiated w.r.t. $x$, $y$.
\end{lemma}

\begin{proof} Reversing the power series \eqref{eq:psihLUE} gives
\[
t = s+\frac{2 \tau -1}{5}s^2 h +\frac{13 \tau ^2-38 \tau +22 }{175} s^3 h^2 +\frac{ 34 \tau ^3-776 \tau ^2+1588 \tau -823}{7875}s^4 h^3 +  \cdots,
\]
which is uniformly convergent for $s_0 \leq s\leq h^{-1}/2$ (if $h$ is small enough such that $s_0 > -h^{-1}/2$) since then $|ht| \leq (-\rho_{0,0}) \vee \rho_{1,1} < 1$. Taking the expressions for $K_1$, $K_2$ given in \eqref{eq:K2LUE}, and a similar one for the kernel $K_3$, which is not shown here, a routine calculation with truncated power series gives formula \eqref{eq:K1tildeLUE} for $\tilde K_1$ and the  expansion kernels 
\begin{align*}
\tilde K_2(x,y) &= \frac{(\tau -3)^2}{100}  \Ai(x)\Ai(y) - \Big(\frac{51 \tau ^2-106 \tau +159}{1400}x-\frac{11 \tau ^2+54 \tau -81}{280}y \Big)\Ai(x)\Ai'(y)\\*[1mm]
&\qquad+ \Big(\frac{11 \tau ^2+54 \tau -81}{280}x-\frac{51 \tau ^2-106 \tau +159}{1400}y \Big)\Ai'(x)\Ai(y),\\*[2mm]
\intertext{}\notag\\*[-13mm]
\tilde K_3(x,y) &= -\frac{19 \tau ^3+229 \tau ^2-1787 \tau +1787}{9000} (x+y)\Ai(x)\Ai(y)\\*[1mm]
&\qquad +\Big(\frac{583 \tau ^3+553 \tau ^2-3359 \tau +3359}{42000} x^2-\frac{13 \tau ^3-77 \tau ^2+271 \tau -271}{1800} x y\\*[1mm]
&\qquad\qquad-\frac{13 \tau ^3-77 \tau ^2+1171 \tau -1171 }{3600}y^2\Big)\Ai(x)\Ai'(y)\\*[1mm]
&\qquad - \Big(\frac{13 \tau ^3-77 \tau ^2+1171 \tau -1171}{3600} x^2+\frac{13 \tau ^3-77 \tau ^2+271 \tau -271}{1800} x y\\*[1mm] 
&\qquad\qquad-\frac{583 \tau ^3+553 \tau ^2-3359 \tau +3359 }{42000}y^2\Big)\Ai'(x)\Ai(y)\\*[1mm]
&\qquad - \frac{23 \tau ^3-7 \tau ^2-979 \tau +979}{6300}\Ai'(x)\Ai'(y)
\end{align*}
-- much simpler expressions than those for $K_1, K_2, K_3$ (e.g., the degree of the polynomial coefficients in $K_3$ is reduced from seven to two).

A repeated application of the replacement rule \eqref{eq:airyrule}
allows us to absorb any powers of $x$ and~$y$ into higher order derivatives of $\Ai$. This process yields the asserted form of $\tilde K_2$ and~$\tilde K_3$, which will be the preferred form in the course of the calculations in Section~\ref{sect:functionalformLUE}.

Since we stay within the range of uniformity of the power series expansions, and calculations with truncated powers series are amenable to repeated differentiation, the result now follows from the bounds in \cite[§2.1]{arxiv:2301.02022}.
\end{proof}

\subsection{Proof of the general form of the expansion}\label{sect:generalformLUE}

By Lemma~\ref{lem:LUEkernexpan} and \cite[Theorem~2.1]{arxiv:2301.02022},  we get for the change of variable $x=\mu_{n,p} + \sigma_{n,p} t$
\begin{align*}
E_2(n,p;x) &= \det(I - K_{n,p}^{\LUE})|_{L^2(x,\infty)}  \\*[1mm] 
&= \det(I - \hat K_{n,p}^{\LUE})|_{L^2(t,\infty)}
= F_2(t) + \sum_{j=1}^m E_{2,j;\tau_{n,p}}(t) h_{n,p}^j + h_{n,p}^{m+1} O(e^{-2t}),
\end{align*}
uniformly for $t\geq t_0$ and $0 < \tau_0\leq \tau_{n,p}\leq 1$ as $h_{n,p}\to0^+$; preserving uniformity, this expansion can be repeatedly differentiated w.r.t. the variable $t$. By \cite[Theorem~2.1]{arxiv:2301.02022}, the $E_{2,j;\tau}(t)$ are certain smooth functions that can be expressed in terms of traces of integral operators. This finishes the proof of \eqref{eq:GUE}.

\subsection{Functional form of the expansion terms}\label{sect:functionalformLUE}

We follow the method of \ref{sect:functionalformGUE} to avoid computing $E_{2,j;\tau}$ directly from the trace formulas in \cite[Theorem~2.1]{arxiv:2301.02022} applied to $K_1, K_2, \ldots$ as in (\ref{eq:K2LUE}). Instead, we will first compute the corresponding functions $\tilde E_{2,j;\tau}$ induced by the much simpler kernels $\tilde K_1$, $\tilde K_2$, $\tilde K_3$ in (\ref{eq:auxkernelexpanLUE}). 

Once again, inserting the coefficients in (\ref{eq:auxkernelexpanLUE}) into the formulas \eqref{eq:K1K3} and \eqref{eq:G1G3nice}   yields
\begin{align*}
\tilde E_{2,1;\tau}(t) &= \frac{\tau -3}{10} F''(t), \\*[2mm]
\tilde E_{2,2;\tau}(t) &= \frac{\tau ^2+94 \tau -141}{350} F'(t) - \frac{4 \tau ^2+26 \tau -39}{175} t   F''(t) +\frac{(\tau -3)^2}{200}  F^{(4)}(t),\\*[3mm]
\tilde E_{2,3;\tau}(t) &= -\frac{2(4 \tau ^3-161 \tau ^2+1108 \tau -1108) }{7875}tF'(t) + \frac{9 \tau ^3-56 \tau ^2+568 \tau -568}{2625}t^2F''(t)\\*[1mm]
&\qquad\qquad +\frac{61 \tau ^3+1351 \tau ^2-10403 \tau +10403 }{31500}F^{(3)}(t)\\*[1mm]
&\qquad\qquad\qquad -\frac{(\tau -3) (4 \tau ^2+26 \tau -39)}{1750}tF^{(4)}(t) +\frac{(\tau -3)^3 }{6000}F^{(6)}(t).
\end{align*}
The relation between $E_{2,1;\tau}(t)$, $E_{2,2;\tau}(t)$, $E_{2,3;\tau}(t)$ and their counterparts with a tilde is established in Lemma~\ref{lem:KhGUEexpan}: with $t$ being fixed, $h$ sufficiently small and $s=\psi_h^{-1}(t)$, we get by applying \cite[Theorem~2.1]{arxiv:2301.02022} 
\begin{multline*}
F_2(t) + E_{2,1;\tau}(t) h + E_{2,2}(t;\tau)h^2 + E_{3,2}(t;\tau)h^3 + O(h^4) \\*[1mm]
= \det\big(I-K_{(h)}\big)|_{L^2(t, \rho_{1,\tau} h^{-1})}  = \det\big(I-\tilde K_{(h)}\big)|_{L^2(s, h^{-1}/2)} \\*[1mm]
= F_2(s) + \tilde E_{2,1;\tau}(s) h + \tilde E_{2,2;\tau}(s)h^2 + \tilde E_{2,3;\tau}(s)h^3 + O(h^4),
\end{multline*}
after absorbing the exponentially small contributions $e^{-\rho_{1,\tau} h^{-1}} O(e^{-t})$ and $e^{-h^{-1}/2} O(e^{-s})$
into the $O(h^4)$ error term.
Upon inserting the power series \eqref{eq:psihLUE},
followed by Taylor expansions and comparing coefficients, we obtain the expressions  for $E_{2,1;\tau}$, $E_{2,2;\tau}$, $E_{2,3;\tau}$ in \eqref{eq:F22LUE}. This finishes the proof of Theorem~\ref{thm:softLUE}.

\subsection{The LUE with bounded parameter}\label{subsect:boundedAlpha}

To compare with  prior work in the literature, we consider the case of fixed, or bounded from above, parameter $\alpha>-1$ in the LUE. Then, with $p=n+\alpha$, as $n\to\infty$, the expansion parameters of Theorem~\ref{thm:softLUE} simplify to leading order as follows: 
\begin{gather*}
\mu_{n,p} = 4n+2\alpha + O(n^{-1}),\quad
\sigma_{n,p}= 2(2n)^{1/3} + O(n^{-2/3}),\\*[1mm]
h_{n,p} = (2n)^{-2/3} + O(n^{-5/3}),\quad
\tau_{n,p} = 1 + O(n^{-2}).
\end{gather*}
Now, a simple re-expansion of the case $m=1$ in Theorem~\ref{thm:softLUE} yields  the following corollary -- a result, which we had given without the tail bound in \cite[§3.4]{arxiv:2301.02022}.

\begin{corollary}\label{cor:softLUE} There holds the large matrix expansion of the LUE at the soft edge
\smallskip
\begin{equation}\label{eq:softLUEalpha}
E_{\LUE,\alpha}\big(n;4n+2\alpha + 2(2n)^{1/3} t\big) = F_2(t) + G(t)\cdot (2n)^{-2/3}  + n^{-1} O(e^{-2t}),
\end{equation}

\smallskip
\noindent
which is uniformly valid as $n\to\infty$ when $t\geq t_0$ and $-1 < \alpha \leq \alpha_0$, with $t_0,\alpha_0$ being any fixed real numbers. Here, the finite-size correction coefficient $G$ is given by the expression\footnote{We note that $G(t)$
 is {\em independent} of the Laguerre parameter $\alpha$. Therefore, the expansion \eqref{eq:softLUEalpha} is not robust for large $\alpha$: there must be a break-down of uniformity when 
$\alpha\to\infty$, since otherwise the LUE-to-GUE transition (as discussed in Remark~\ref{rem:LUE2GUE}) could not be matched.}
\begin{equation}\label{eq:E1LUEalpha}
G(t) := E_{2,1;\tau}(t)\big|_{\tau=1} = -\frac{1}{5}\big (t^2 F'(t) + F''(t)\big).
\end{equation}
Preserving uniformity, the expansion can be repeatedly differentiated w.r.t. the variable $t$. 
\end{corollary}

\begin{remark} 
The first  correction term $L(x,y)n^{-2/3}$ in the kernel expansion, obtained by accordingly re-expanding \eqref{eq:LUEkernexpan}, is simply given by the kernel
\begin{multline}\label{eq:kernelLalpha}
L(x,y) 
:= K_1(x,y;\tau)\big|_{\tau=1} \\*[1mm]
= \frac{1}{5}\big((x^2+xy+y^2) \Ai(x)\Ai(y) 
- (x+y) \Ai'(x)\Ai'(y) + \Ai(x)\Ai'(y) + \Ai'(x)\Ai(y) \big).
\end{multline}
\end{remark}

\subsection{Comments on prior work}\label{subsect:previousLUE}

We discuss the relevant prior  in chronological order.%

\subsubsection{The case $m=1$, bounded parameter $\alpha$}\label{subsect:ChoupLUE}

In 2005, Garoni, Forrester and Frankel \cite[Eq.~(73)]{MR2178598} had basically\footnote{After adjusting for the different scaling in their paper, by inserting $\xi=x + \alpha/(2n)^{1/3}$ into \cite[Eq.~(73)]{MR2178598} and re-expanding.} found the diagonal term $L(x,x)$ of \eqref{eq:kernelLalpha} as the finite-size correction to the one-point function  at the soft edge of the LUE with bounded $\alpha$.

In 2006, upon establishing  formula \eqref{eq:kernelLalpha} for $L(x,y)$, Choup \cite[Eq.~(3.5)]{MR2233711} gave an incomplete\footnote{See the discussion in Footnote~\ref{fn:Choup}.}  proof that 
\begin{multline}\label{eq:LUEChoup}
E_{\LUE,\alpha}\big(n;4n+2\alpha + 2(2n)^{1/3} t\big) \\*[1mm]
= F_2(t) + F_2(t) \det\big(I-(2n)^{-2/3}(I - K_{\Ai})^{-1} L \big)\big|_{L^2(t,\infty)} + O(n^{-1}),
\end{multline}
 uniformly  for bounded $t$. By  expressing the Fredholm determinant of a finite-rank perturbation of the identity  in terms of a matrix determinant, Choup 
inferred from this the expansion
\[
E_{\LUE,\alpha}\big(n;4n+2\alpha + 2(2n)^{1/3} t\big) = F_2(t) + G(t) \cdot (2n)^{-2/3} + O(n^{-1}),
\]
see \cite[Theorem~1.4]{MR2233711},  representing 
\begin{equation}\label{eq:GChoup}
G/F_2 = \frac{1}{5}\big(2w_1 - 3u_2 -2 v_0 + u_1 v_0 - u_0 v_1 + u_0v_0^2 -u_0^2 w_0\big)
\end{equation}
in terms of the  quantities introduced in Section~\ref{subsect:ChoupGUE} for the GUE.

As in  Section~\ref{subsect:ChoupGUE}, there are basically two ways to simplify Choup's findings to our expression~\eqref{eq:E1LUEalpha} of $G(t)$. First, when recursively applying the replacement rule \eqref{eq:airyrule} to trade the multiplication with powers of $\xi$ for higher order derivatives of the Airy function,  we can rewrite \eqref{eq:GChoup} in the form
\[
G/F_2 = \frac{1}{5} \big(2 u_{01}-3 u_{04}+2 u_{13}+u_{00}^2-u_{00}u_{03} +u_{01} u_{02} - u_{00}^2u_{11}+u_{00} u_{01}^2  \big)
\]
-- an expression which the algorithm developed in \cite[Appendix~B]{arxiv:2301.02022} then readily reduces 
to~\eqref{eq:E1LUEalpha}.
Second, as we have already sketched in \cite[§3.4]{arxiv:2301.02022}, the expansion \eqref{eq:LUEChoup} allows us to deduce the expression \eqref{eq:E1LUEalpha} rather directly -- without harnessing the power of the Tracy--Widom theory at all; see also \eqref{eq:detexpand2} and Remark~\ref{rem:u01}.

\subsubsection{The case $m=0$, parameter $\alpha$ proportional to $n$}
In 2006, El Karoui \cite{MR2294977} studied the convergence rates of the Tracy--Widom limit law for complex Wishart matrices when 
\[
p/n \to \gamma \in (0,\infty)
\]
as $n\to \infty$ (so that $\alpha = p-n \sim (\gamma-1)n$ if the symmetry between $p$ and $n$ is broken with choosing $p\geq n$).
By fully addressing the technical difficulties in lifting kernel bounds to trace class bounds, emphasizing differentiability, El Karoui \cite[Theorem~2]{MR2294977} proved
\[
E_2\big(n,p;\mu_{n,p} + \sigma_{n,p} t\big) = F_2(t) + (n\wedge p)^{-2/3} O(e^{-t})
\]
to be uniformly valid  as $n\to\infty$ when $t$ is bounded from below. Except for  a slightly weaker tail bound, this result is the case $m=0$ of our Theorem~\ref{thm:softLUE}, and El Karoui's proof is using a  technique which is similar to ours, see Remark~\ref{rem:Lag_m01}.

\subsubsection{The case $m=1$, parameter $\alpha$ proportional to $n$}
 In 2018, Forrester and Trinh \cite[§IV.A]{MR3802426} studied the case $m=1$ of Theorem~\ref{thm:softLUE} as $n\to\infty$ while
\[
\tilde\alpha := \alpha/n \geq 0
\]
is kept fixed. Expressed in terms of $n$ and $\tilde\alpha$, but briefly writing $a:=\sqrt{\tilde\alpha+1}\geq 1$,
the parameters of Theorem~\ref{thm:softLUE} become (cf.~\eqref{eq:LagScaling} and \eqref{eq:htau}), with $p = a^2n$,
\[
\mu_{n,p} = \left(a+1\right)^2 n,\quad \sigma_{n,p}= \frac{(a+1)^{4/3}}{a^{1/3}}n^{1/3},\quad
\tau_{n,p} =  \frac{4 a}{(a+1)^2}, \quad h_{n,p}= \frac{(a+1)^{4/3}}{4a^{4/3}} n^{-2/3}.
\]
Forrester and Trinh succeeded in representing the first finite-size correction of Theorem~\ref{thm:softGUE} in an operator theoretic form \cite[Eq.~(2.14)]{MR3802426} (cf. also~\cite[Theorem~2.1]{arxiv:2301.02022}), 
\begin{equation}\label{eq:FTfinitesizecor}
E_{2,1;\tau_{n,p} }(t) h_{n,p}= - F_2(t) \tr((I-K_{\Ai})^{-1} L_{\tilde\alpha}) n^{-2/3},
\end{equation}
where \cite[Prop.~5]{MR3802426} 
\begin{equation}\label{eq:FTexpansion}
\sigma_{n,p} \, K_{n,p}^{\LUE}(\mu_{n,p} + \sigma_{n,p} x,\mu_{n,p} + \sigma_{n,p} y) = K_{\Ai}(x,y) + L_{\tilde\alpha}(x,y) n^{-2/3} + O(n^{-1}),
\end{equation}
with the kernel $L_{\tilde\alpha}$ expressed in a form that contains the Airy kernel,\footnote{The original expression \cite[Eq.~(4.3)]{MR3802426} is written using $\sqrt{\tilde\alpha+1}$ instead of $a$. As such, even  a computer algebra system such as Mathematica has  difficulties in getting from this stack of algebraic functions to the simplification \eqref{eq:simpleLalpha}.}
\begin{multline*}
L_{\tilde\alpha}(x,y) = -\frac{(a-1)^2}{32 a^{4/3} (a+1)^{2/3}} (x^2+y^2)^2 K_{\Ai}(x,y)\\*[0mm]
+ \frac{1}{160 a^{4/3} (a+1)^{2/3} (x-y)} \bigg(-8 (a^2-6 a+1)(x^3-y^3)\Ai(x)\Ai(y)\\*[0mm]
+\Big(4 (3 a^2+2 a+3) (x-y) -5 (a-1)^2 (x^2+y^2)^2\Big)\Ai(y)\Ai'(x) \\*[0mm]
+\Big(4 (3 a^2+2 a+3) (x-y) +5 (a-1)^2 (x^2+y^2)^2\Big)\Ai(x)\Ai'(y) \\*[0mm]
+8 (a^2 - 6 a + 1) (x^2-y^2) \Ai'(x)\Ai'(y)\bigg)
\end{multline*}
-- a form that does not immediately reveal its finite-rank structure. 

However, collecting  equal Airy factors, followed by a polynomial division by $x-y$, yields the much simplified 
\begin{multline*}
L_{\tilde\alpha}(x,y)= \frac{(a+1)^{-2/3}}{40a^{4/3}} \Big(2(a^2-6 a+1)\big((x+y)\Ai'(x)\Ai'(y) - (x^2+xy+y^2)\Ai(x)\Ai(y)\big)\\*[0mm]
+ (3 a^2+2 a+3) \big(\Ai(x)\Ai'(y)+\Ai'(x)\Ai(y)\big) \Big) = \frac{(a+1)^{4/3}}{4a^{4/3}} K_1(x,y;\tau_{n,p}),
\end{multline*}
a kernel that now clearly exhibits its finite-rank structure. In particular, we get
\begin{equation}\label{eq:simpleLalpha}
L_{\tilde\alpha}(x,y) n^{-2/3} = K_1(x,y;\tau_{n,p}) h_{n,p},
\end{equation}
so that the expansion \eqref{eq:FTexpansion} corresponds to the case $m=1$ of Lemma~\ref{lem:LUEkernexpan} with a somewhat weaker error estimate. 

Forrester and Trinh evaluated the finite-size correction \eqref{eq:FTfinitesizecor}, using their expression of the kernel $L_{\tilde\alpha}$, by the numerical methods developed in \cite{MR2895091, MR2600548,MR3647807}. Numerical results, which are similar to the one shown in the middle panel of Figure~\ref{fig:LUESimulation}, are then displayed in \cite[Figure~5]{MR3802426} for the specific values $\tilde\alpha=1/2$ ($\tau_{n,p}=0.98979\cdots$) and $\tilde\alpha=5$ ($\tau_{n,p}=0.82342\cdots$).

\part*{Part II. Orthogonal and Symplectic Ensembles}

\section{The general approach}\label{sect:genAppr}

\subsection{Interrelations of the three symmetry classes}\label{subsect:inter}

The  expansion terms of the orthogonal and symplectic ensembles will be obtained by observing a tight algebraic relation to the expansion terms of the unitary ensemble, which we discussed in Part I. The starting point is the interrelations, through decimation and superposition, between the three symmetry classes as established by Forrester and Rains \cite{MR1842786}.

Specifically, in generalizing results of Dyson, Gunson and Mehta \cite{MR0177643,MR0148401,MR0151232} for the circular ensembles, Forrester and Rains proved, for the particular choices of the variances of the random matrices in Section~\ref{sect:intro}, that \cite[Theorems~4.3/4.7/5.1/5.2]{MR1842786}
\begin{equation}\label{eq:FR2001}
\UE_{n} = \even(\OOE_{n} \cup \OOE_{n+1}), \qquad \SE_n = \even(\OOE_{2n+1});\footnote{Forrester \cite[Theorem 2]{MR2461989} generalized the $\SE$/$\OOE$ interrelation
to one between  $\beta=2d$ and $\beta'=2/d$, 
\[
\text{$\beta$E}_n = \text{Dec}_{d}(\text{$\beta'$E}_{d(n + 1) - 1}),
\]
where $\text{Dec}_{d}$ denotes taking every $d$-th ordered level ($d=1,2,3,\ldots$). Those intriguingly simple interrelations require carefully adjusting the scaling and parametrization of the corresponding classical (Gaussian, Laguerre, Jacobi, or circular) matrix ensembles.}
\end{equation}
where, for any given $\alpha>-1$, the Laguerre parameters $\alpha^\dagger$ of the three symmetry classes are connected according to 
\begin{equation}\label{eq:alphaprime}
\alpha^\dagger:= 
\begin{cases}
(\alpha-1)/2, &\quad \beta = 1,\\*[1mm]
\alpha,           	&\quad \beta = 2,\\*[1mm]
\alpha+1, 			&\quad \beta = 4.
\end{cases}
\end{equation}
In view of \eqref{eq:alphabetagamma}, this means simply that the Wishart parameter $p$ (`degrees of freedom') undergoes the same arithmetic operations as $n$ -- namely,
\begin{equation}\label{eq:FR2001LbE}
\LUE_{n,p} = \even(\LOE_{n,p} \cup \LOE_{n+1,p+1}), \qquad \LSE_{n,p} = \even(\LOE_{2n+1,2p+1}),
\end{equation}
once more emphasizing the symmetry between $n$ and $p$. 

From now on, in the general discussion of this Section~\ref{sect:genAppr}, we will suppress from the notation the dependence on $p$, or  on the parameter $\tau = \tau_{n,p}$ (assuming $\tau_{n,p}\geq \tau_0>0$) for that matter. It will tacitly be understood that whenever the dimension $n$ gets modified, the Wishart parameter $p$ is modified in exactly the same way.

The meaning of the interrelations \eqref{eq:FR2001} is as follows. The first equation states that the statistics of the ordered eigenvalues of the 
$n$-dimensional unitary ensemble are the same as that of the even-numbered (`decimation') ordered levels obtained from joining ('superposition') the eigenvalues of an $n$-dimensional orthogonal ensemble with those of an independent $n+1$-dimensional one. The second equation states that the statistics of the ordered eigenvalues of the 
$n$-dimensional symplectic ensemble are the same as that of the even-numbered ordered eigenvalues of a $2n+1$-dimensional 
orthogonal one. If we briefly write 
\[
E_{\beta,q}(n;x) := \prob\big(\text{exactly $q$ eigenvalue of the $\beta$-ensemble are larger than $x$}\big),
\]
so that $E_{\beta,0}(n;x) = E_\beta(n;x)$,
the interrelations \eqref{eq:FR2001} translate, by elementary probability, into \cite[Eq.~(5-4)]{MR1842786} 
\begin{subequations}\label{eq:FR2001G}
\begin{align}
E_{2,0}(n;x) &= E_{1,0}(n;x)E_{1,0}(n+1;x) \\*[1mm]
&\quad +  E_{1,0}(n;x) E_{1,1}(n+1;x) +  E_{1,1}(n;x)  E_{1,0}(n+1;x),\notag\\*[2mm]
E_{4,0}(n;x) &= E_{1,0}(2n+1;x) + E_{1,1}(2n+1;x).
\end{align}
\end{subequations}

On the other hand, for the limit distributions in \eqref{eq:TWlimit}, Tracy and Widom \cite[Eqs.~(52--54)]{MR1385083} established the limit interrelations
\begin{subequations}\label{eq:TWtheory}
\begin{gather}
F_\pm(t) = e^{\mp \frac12\int_t^\infty q(x)\,dx} \sqrt{F_2(t)},\label{eq:Fpm}\\*[1mm]
F_1(t) = F_+(t), \quad F_2(t)  = F_+(t) F_-(t),\quad F_4(t) = \frac{1}{2}\big(F_+(t) + F_-(t)\big),\label{eq:FerrariSpohn}
\end{gather}
\end{subequations}
where $q(x)$ denotes the Hastings--McLeod solution of Painlevé II. 
Now, upon defining
\begin{subequations}
\begin{equation}\label{eq:Epm}
E_+(n;x) := E_{1,0}(n;x), \qquad E_-(n;x) := E_{1,0}(n;x) + 2E_{1,1}(n;x),
\end{equation}
we can put the interrelations \eqref{eq:FR2001G} in a form that is analogous to \eqref{eq:FerrariSpohn} -- namely,%
\begin{equation}\label{eq:FR2001PM}
\begin{aligned}
E_{1}(n;x) &= E_+(n;x) \\*[2mm]
E_{2}(n;x) &= \frac{1}{2} \big( E_+(n;x) E_-(n+1;x) + E_+(n+1;x) E_-(n;x)\big),\\*[2mm]
E_{4}(n;x) &= \frac{1}{2} \big( E_+(2n+1;x) + E_-(2n+1;x)\big).
\end{aligned}
\end{equation}
\end{subequations}

\subsection{Limit `laws' and the self-consistent expansion hypothesis}\label{subsect:selfconsistent} From \cite{MR1385083,MR1863961} it is known that in the orthogonal and unitary cases $\beta=1,2$, the limit laws
\[
\lim_{n\to\infty} E_{\beta}(n;\mu_n + \sigma_n t)  = F_\beta(t)
\]
hold uniformly in $t$ for the scaling parameters $\mu_n$, $\sigma_n$, chosen as in Theorem~\ref{thm:softGUE} in the Gaussian cases, and as in Theorem~\ref{thm:softLUE} in the Laguerre cases. Since 
\[
\mu_n + \sigma_n t = \mu_{n+1} + \sigma_{n+1} t_\sharp \quad\Rightarrow\quad t_\sharp = t + O(n^{-1/3}),
\]
it follows from  \eqref{eq:FerrariSpohn} and \eqref{eq:FR2001PM} that
\begin{equation}
\lim_{n\to\infty} E_\pm(n;\mu_n + \sigma_n t) = F_\pm(t),
\end{equation}
locally uniform in $t$. On the other hand, from Theorems \ref{thm:softGUE} and  \ref{thm:softLUE} we know that, in the case of the unitary ensembles, there is an expansion of the form
\begin{equation}\label{eq:E2Expansion}
E_2(n;\mu_n + \sigma_n t) = F_2(t) + \sum_{j=1}^m E_{2,j}(t) h_{n}^j + h_n^{m+1} O(e^{-2t}),
\end{equation}
which, as $h_n\to0^+$, holds uniformly for $t$ bounded from below. Now, in analogy to the hard-to-soft edge transition limit, where we established (with a proof based on determinantal representations) a corresponding expansion \cite[Eq.~(40)]{2306.03798}, we assume the following hypothesis to be true (recall $n_-=n-1/2$).

\medskip
\paragraph{\bf Self-consistent expansion hypothesis} As $h_{n_-}\to 0^+$, there holds the expansion
\begin{equation}\label{eq:EpmSelfConsistent}
E_\pm(n;\mu_{n_-} + \sigma_{n_-} t) = F_\pm(t) + \sum_{j=1}^m E_{\pm,j}(t) h_{n_-}^j + h_{n_-}^{m+1} O(e^{-t}),
\end{equation}
uniformly for $t$ bounded from below. Inductively, this implies tail bounds for the expansion terms of the form $E_{\pm,j}(t) = O(e^{-t}).$ Preserving uniformity, the expansion can be repeatedly differentiated w.r.t. the variable $t$.

\medskip

We are pretty confident that, with persistence and  proper bookkeeping, based on the expansions of the Hermite and Laguerre polynomials of Part III, one can prove this hypothesis by extending the calculations in the work of Johnstone and Ma \cite{MR3025686,MR2888709}, who established the case $m=0$ for the GOE and LOE using Pfaffian representations.

We call the ansatz \eqref{eq:EpmSelfConsistent} `self-consistent' because of the following reasoning. Let us define the variables $t_\nu$ ($\nu=0,1$) by setting
\[
x = \mu_n + \sigma_n t =: \mu_{n_-+\nu} + \sigma_{n_-+\nu} \, t_\nu
\]
which gives, in terms of $t$ and $h=h_n$, a local power series expansion
\begin{subequations}\label{eq:alpha2delta}
\begin{equation}
t_{0,1} = t \left(1 + \sum_{j=1}^\infty (\pm1)^j \alpha_j h^{3j/2}\right) \pm h^{1/2}\left(1 + \sum_{j=1}^\infty (\pm1)^j \beta_j h^{3j/2}\right),
\end{equation}
where the plus signs apply to $t_0$ and the minus signs to $t_1$. Likewise there are, for the  quantities $\tilde h_\nu := h_{n_-+\nu}$
and, in the Laguerre case, also for $\tau_\nu = \tau_{n_-+\nu,p_-+\nu}$, local power series expansions
\begin{align}
\tilde h_{0,1} &:= h \left(1+ \sum_{j=1}^\infty (\pm1)^j \gamma_j h^{3j/2}\right),\\*[2mm]
 \tilde \tau_{0,1} &:= \tau \left(1+ \sum_{j=1}^\infty (\pm1)^j \delta_j h^{3j/2}\right).
\end{align}
\end{subequations}
Here, $\alpha_j,\beta_j,\gamma_j,\delta_j \in \Q[\tau]$, with $\tau=0$ in the Gaussian case (see Remark~\ref{rem:LUE2GUE}).
Now, when plugging the expansions \eqref{eq:EpmSelfConsistent} into \eqref{eq:FR2001PM} and re-expanding in terms of $t$, $\tau$ and $h$, we obtain, by symmetry,
that the odd powers of $h^{1/2}$  cancel, 
\begin{multline}\label{eq:expansionE2Epm}
E_2(n;\mu_n + \sigma_n t)\\*[1mm] = \frac{1}{2} \sum_{\nu=0,1} E_+(n+\nu,\mu_{n_-+\nu} + \sigma_{n_-+\nu} \, t_\nu) E_-(n+1-\nu,\mu_{n_-+1-\nu} + \sigma_{n_-+1-\nu} \, t_{1-\nu}) \\*[0mm]
= \frac{1}{2} \sum_{\nu=0,1} \left(F_+(t_\nu) + \sum_{j=1}^m E_{+,j}(t_\nu)\tilde h_\nu^j + \tilde h_\nu^{m+1} O(e^{-t})\right)\\*[2mm]
\times \left(F_-(t_{1-\nu}) + \sum_{j=1}^m E_{+,j}(t_{1-\nu})\tilde h_{1-\nu}^j + \tilde h_{1-\nu}^{m+1} O(e^{-t})\right)\\*[1mm]
= F_+(t) F_-(t) + \sum_{j=1}^m \tilde E_{2,j}(t) h_n^j +  h_n^{m+1} O(e^{-2t}),
\end{multline}
which is  consistent with $F_2 = F_+ F_-$ and \eqref{eq:E2Expansion}. By uniqueness of the expansions coefficients, we get $\tilde E_{2,j}(t) = E_{2,j}(t)$ -- for example, going through all the details, we have
\begin{equation}\label{eq:E21algebra}
E_{2,1} =F_+ E_{-,1} +E_{+,1}F_- + \frac{1}{2}\big(F_+ F_-'' - 2F_+'F_-' + F_+'' F_-\big) ,
\end{equation}
and similar expressions, though much more involved, for $E_{2,2}(t)$, $E_{2,3}(t)$, \ldots, which, of course, will then depend  on $\alpha_j,\beta_j,\gamma_j,\delta_j$.

Clearly, by  \eqref{eq:FR2001PM}, the ansatz \eqref{eq:EpmSelfConsistent}  induces expansions for $\beta=1,4$. With
\begin{equation}\label{eq:nprime}
n' = 
\begin{cases}
n_- = n - \frac12, &\quad \beta = 1,\\*[1mm]
n,           &\quad \beta = 2,\\*[1mm]
(2n+1)_- = 2n+ \frac12, &\quad \beta = 4,
\end{cases}
\end{equation}
we obtain, in basically one single notation for all the three symmetry classes,
\begin{equation}\label{eq:Ebetaexpansion}
E_\beta\big(n;\mu_{n'} + \sigma_{n'} t\big) = F_\beta(t) + \sum_{j=1}^m E_{\beta,j}(t) h_{n'}^j + h_{n'}^{m+1}\cdot
\begin{cases}
O(e^{-t}), &\quad \beta=1,4,\\*[2mm]
O(e^{-2t}), &\quad \beta=2,
\end{cases}  
\end{equation}
which is uniformly valid for $t$ bounded from below (if the hypothesis for $\beta=1,4$ is true). Preserving uniformity, the expansion can be repeatedly differentiated w.r.t. the variable $t$. In terms of $E_{\pm,j}$ the coefficients $E_{\beta,j}$ 
are given, for $\beta=1,4$, by
\begin{equation}\label{eq:EbetaJ}
E_{1,j}(t) = E_{+,j}(t),\qquad  E_{4,j}(t) = \frac12\big(E_{+,j}(t)+E_{-,j}(t)\big),
\end{equation}
and, for $\beta=2$, by the relations such as \eqref{eq:E21algebra}, which were derived from \eqref{eq:expansionE2Epm}.

\subsection{Multivariate polynomial algebra and the linear form hypothesis}\label{subsect:multilinear}

Theorems~\ref{thm:softGUE} and \ref{thm:softLUE}
give (at least, for $m\leq m_*$) that
\begin{equation}\label{eq:E2LinForm}
E_{2,j} = \sum_{k=1}^{2j} p_{2,jk} F_2^{(k)}, \quad p_{2,jk} \in \Q[t,\tau],
\end{equation}
with $\tau=0$ in the Gaussian case. In analogy to the hard-to-soft edge transition limit, as discussed in \cite[§3.3]{2306.03798} (where we provided proof for the case $j=1$ and strong numerical evidence for the cases $j\leq 3$), we are led to the following hypothesis.

\medskip
\paragraph{\bf Linear form hypothesis} The coefficient functions $E_{\pm,j}$ can be represented in the form
\begin{equation}\label{eq:EpmHypo}
E_{\pm,j} =\sum_{k=1}^{2j} p_{\pm,jk} F_\pm^{(k)}, \quad p_{\pm,jk} \in \Q[t,\tau],
\end{equation}
with $\tau=0$ in the Gaussian case.

\medskip

Actually, based on this hypothesis, it is possible to calculate the polynomials $p_{\pm,jk}$ (at least, for $j\leq m_*$).
The thus obtained expansions for $E_\beta(n;x)$ can then be validated against simulation data, which exhibits perfect agreement -- see Figures~\ref{fig:GOEGSE}, \ref{fig:GOEGSESimulation} and \ref{fig:LOELSESimulation}.

The calculation of the polynomials $p_{\pm,jk}$ is based on the observation that
\[
\frac{F_2^{(k)}}{F_2},\;\frac{F_\pm^{(k)}}{F_\pm} \in \Q[t][q,q'] \qquad (k=1,2,\ldots),
\]
which easily follows, by induction, from the basic formulas of the Tracy--Widom theory \cite{MR1385083} -- that is, the interrelations \eqref{eq:TWtheory},
the Painlevé II equation
\[
q''(t) = t q(t)  + 2 q(t)^3,
\]
so that $\Q[t][q,q']$ is closed under differentiation, and the logarithmic derivative
\begin{subequations}\label{eq:FpmQ2}
\begin{equation}
\frac{F_2'}{F_2} = q'^2 - t q^2 - q^4.
\end{equation}
For instance, we thus get 
\begin{equation}
\begin{aligned}
\frac{F_\pm'}{F_\pm} &= \tfrac{1}{2} q'^2-\tfrac{1}{2} q^4-\tfrac{t}{2}  q^2\pm\tfrac{1}{2}q,\\*[1mm]
\frac{F_2''}{F_2} &= -2 q^4 q'^2-2 t q^2 q'^2+q'^4+t^2 q^4+q^8+2 t q^6-q^2,\\*[1mm]
\frac{F_\pm''}{F_\pm} &= -\tfrac{1}{2} q^4 q'^2-\tfrac{t}{2}  q^2 q'^2\pm\tfrac{1}{2} q q'^2+\tfrac{1}{4} q'^4\pm\tfrac{1}{2}q'+\tfrac{t^2}{4}  q^4 +\tfrac{1}{4}q^8\\
&\qquad +\tfrac{t}{2}  q^6\mp\tfrac{1}{2}q^5\mp\tfrac{t}{2}  q^3-\tfrac{1}{4}q^2.
\end{aligned}
\end{equation}
\end{subequations}
Since  any nonzero solution $q$ of the Painlevé II equation (such as the Hastings--McLeod solution) and its first derivative $q'$  are algebraically independent over the ring $\Q[t]$ (see \cite[Theorems~20.2/21.1]{MR1960811}), the functions $q$ and $q'$ serve as  indeterminates over $\Q[t,\tau]$. That is, a multivariate polynomial in $\Q[t,\tau][q,q']$ vanishes if and only if all the coefficients in $\Q[t,\tau]$ vanish: equality can be tested coefficient-wise.

Now, we plug the linear forms \eqref{eq:E2LinForm} and \eqref{eq:EpmHypo} into those formulas which relate the $E_{2,j}$ to the $E_{\pm,j}$, such as \eqref{eq:E21algebra},  and divide by $F_2=F_+ F_-$. Assuming that we have already
found the $p_{\pm,jk}$ for all $j<m$, we  get inductively a linear equation in $\Q[t,\tau][q,q']$ for the still unknown polynomials $p_{\pm,m\mu}$. As in \cite[§3.3.4]{2306.03798}, comparing coefficients in $\Q[t,\tau]$ shows that the $p_{\pm,m\mu}$ must satisfy an overdetermined linear system of size $(16m^2-8m+5)\times 4m$ over $\Q[t,\tau]$. For instance, in the case $m=1$ displayed in \eqref{eq:E21algebra}, we get\footnote{We observe that \eqref{eq:FpmQ2} implies $F_+ F_-'' - 2F_+'F_-' + F_+'' F_- = 0$.}
\[
p_{2,11} \frac{F_2'}{F_2} + p_{2,12} \frac{F_2''}{F_2} =  \sum_{\sigma = \pm} \left(p_{\sigma,11} \frac{F_\sigma'}{F_\sigma} + p_{\sigma,12} \frac{F_\sigma''}{F_\sigma}\right).
\]
Now, by using \eqref{eq:FpmQ2} and comparing the coefficients of the 13 different $(q,q')$-monomials \[
q, \, q^2,\, q^3,\, q^4,\, q^5,\, q^6,\, q^8,\, q',\, q'^2,\, q'^4,\, q q'^2,\, q^2q'^2,\, q^4q'^2,
\]
which enter the expressions in \eqref{eq:FpmQ2}, we get the $13\times 4$ system (canceling,  right away, common factors in each line)
{\begin{equation}\label{eq:exampleSystem}
\begin{pmatrix}
 1 & -1 & 0 & 0 \\
 2t & 2t & 1 & 1\\
 0 & 0 & 1 & -1 \\
 2 & 2 & -t^2 & -t^2\\
 0 & 0 & 1 & -1 \\
 0 & 0 & 1 & 1\\
 0 & 0 & 1 & 1 \\
 0 & 0 & 1 & -1 \\
 1 & 1 & 0 & 0  \\
 0 & 0 & 1 & 1  \\
 0 & 0 & 1 & -1 \\
 0 & 0 & 1 & 1 \\
 0 & 0 & 1 & 1 
\end{pmatrix}
\begin{pmatrix}
p_{+,11}\\*[1mm]
p_{-,11}\\*[1mm]
p_{+,12}\\*[1mm]
p_{-,12}
\end{pmatrix} =
\begin{pmatrix}
 0 \\
 4tp_{2,11} +  4p_{2,12}\\
0 \\*[1mm]
4p_{2,11}- 4t^2 p_{2,12} \\
 0 \\
 4  p_{2,12} \\
 4p_{2,12} \\
 0 \\
 2p_{2,11} \\
4 p_{2,12} \\
 0 \\*[1mm]
 4  p_{2,12} \\
 4p_{2,12} \\
\end{pmatrix}.
\end{equation}}%
Here, the unique solution is clearly
\[
p_{\pm,11} = p_{2,11},\qquad p_{\pm,12} = 2p_{2,12}.
\]
Continuing in the same fashion, the $53\times 8$ system for $m=2$ is uniquely solved by
\[
\begin{gathered}
p_{\pm,21} = p_{2,21}+2 p_{2,24} -\tfrac{1}{2} \partial_t^2 p_{2,11}+p_{2,12} + \tfrac{1}{12}-\beta_1,\quad
p_{\pm,22} = 2 p_{2,22} -\tfrac{1}{2} p_{2,11}^2-\partial_t^2 p_{2,12},\\*[1mm]
p_{\pm,23} = 4 p_{2,23}-2 p_{2,11} p_{2,12},\quad 
p_{\pm,24} = 8 p_{2,24} -2 p_{2,12}^2,
\end{gathered}
\] 
and the $125\times 12$ system for $m=3$ uniquely by
\begin{align*}
p_{\pm,31} &= 8 t p_{2,36} + p_{2,12} p_{2,11}^2 - 2 p_{2,23} p_{2,11} + p_{2,31} +2 p_{2,34}\\
&\qquad -\tfrac{1}{2} p_{2,11}^2 +\tfrac{1}{3} t p_{2,12}+p_{2,22}+4 t p_{2,24}\\
&\qquad\quad + 2 \alpha _1 t p_{2,12} -\alpha _1 t \partial_t^2 p_{2,11}-(\beta_1 +\gamma_1-\tfrac13) \partial_t p_{2,11}-\delta _1 \tau  \partial_\tau\partial_t p_{2,11} \\
&\qquad\qquad  -2 \partial_t p_{2,12} p_{2,11}+2 \partial_t p_{2,23}-\tfrac{1}{2} \partial_t^2 p_{2,12}-\tfrac{1}{2} \partial_t^2 p_{2,21}\\
&\qquad\qquad\quad  -\partial_t^2 p_{2,24}+\tfrac{5}{24} \partial_t^4 p_{2,11} +\tfrac{1}{3} \alpha_1 t-\alpha_2 t+\tfrac{1}{90}t,\\*[1mm]
p_{\pm,32} &= 8 p_{2,11} p_{2,12}^2 -8 p_{2,23} p_{2,12} -p_{2,11} p_{2,21} -18 p_{2,11} p_{2,24}+2 p_{2,32}+20 p_{2,35}\\
&\qquad -5 p_{2,11} p_{2,12} +6 p_{2,23} +(\tfrac{1}{12}-\beta_1) p_{2,11} -2 (\beta_1 + \gamma_1-\tfrac23) \partial_t p_{2,12} \\
&\qquad\quad -8 \partial_t p_{2,12} p_{2,12}+(\partial_t p_{2,11})^2+16 \partial_t p_{2,24}+\tfrac{1}{2} p_{2,11} \partial_t^2p_{2,11}\\
&\qquad\qquad -\partial_t^2 p_{2,22} +\tfrac{5}{12} \partial_t^4 p_{2,12} -2 \alpha _1 t \partial_t^2 p_{2,12} -2 \delta_1 \tau  \partial_\tau\partial_t p_{2,12},\\*[0mm]
p_{\pm,33} &= \tfrac{1}{2} p_{2,11}^3 -2 p_{2,22} p_{2,11}+16 p_{2,12}^3 -2 p_{2,12} p_{2,21} -68 p_{2,12}p_{2,24}
+4 p_{2,33} \\
&\qquad  +120 p_{2,36} -10 p_{2,12}^2+24 p_{2,24} +(\tfrac{1}{6} -2 \beta_1) p_{2,12}\\
&\qquad\quad +\partial_t^2 p_{2,12} p_{2,11}
+4 \partial_t p_{2,11} \partial_t p_{2,12}+p_{2,12} \partial_t^2 p_{2,11}-2\partial_t^2 p_{2,23},\\*[1mm]
p_{\pm,34} &= 3 p_{2,12} p_{2,11}^2-4 p_{2,23} p_{2,11}-4 p_{2,12} p_{2,22}+8 p_{2,34} \\
&\qquad +4 (\partial_t p_{2,12})^2+2 p_{2,12}\partial_t^2 p_{2,12}-4 \partial_t^2 p_{2,24},\\*[1mm]
p_{\pm,35} &= 6 p_{2,11} p_{2,12}^2-8 p_{2,23} p_{2,12}-8 p_{2,11}p_{2,24}+16 p_{2,35},\\*[1mm]
p_{\pm,36} &= 4 p_{2,12}^3-16 p_{2,24} p_{2,12}+32 p_{2,36}.
\end{align*}
Clearly, in the Gaussian case, all terms starting with the operator $\delta_1\tau \partial_\tau$ vanish.

With the help of a computer algebra system, the calculations were successfully (with unique solutions in each step) continued up to $m=m_*$ -- where we decided to stop. Just by inspection, we obtain that, in all these instances, the $E_{\pm,j}$ share the same coefficient polynomials -- namely,
\begin{equation}\label{eq:plusMinusSame}
p_{+,jk} = p_{-,jk} \quad (k=1,\ldots,2j, \, j=1,\ldots,m_*);
\end{equation}
we conjecture this to be generally true and the systems always to be uniquely solvable. Hence, by the linearity of \eqref{eq:EbetaJ}, we get from \eqref{eq:EpmHypo} and \eqref{eq:plusMinusSame}
that
\begin{equation}\label{eq:FbetaLinForm}
E_{\beta,j} = \sum_{k=1}^{2j} p_{+,jk} F_\beta^{(k)} \quad (\beta=1,4).
\end{equation}
So, if the expansion hypothesis and the linear form hypothesis are true, also the cases $\beta=1,4$ have expansion terms (at least, for $j\leq m_*$) that are multilinear in the derivatives of the limit law with rational polynomial coefficients -- and the two symmetry classes of orthogonal and symplectic ensembles share the same set of such polynomial coefficients.\footnote{Such an agreement of the polynomial coefficients is reminiscent of the $\beta$ to $4/\beta$ duality studied in \cite{MR2522665}.}

\section{Expansions of the GOE and GSE}\label{sect:GOEGSE}
In specifying the general approach of Section~\ref{sect:genAppr} (assuming the hypotheses made there), using the notation and  result of Theorem~\ref{thm:softGUE}, while  introducing $n'$ as in \eqref{eq:nprime}, we get
\begin{equation}\label{eq:GbetaE}
E_\beta\big(n;\mu_{n'} + \sigma_{n'} t\big) = F_\beta(t) + \sum_{j=1}^m E_{\beta,j}(t) h_{n'}^j + h_{n'}^{m+1}\cdot
O(e^{-t}) \qquad (\beta=1,4),
\end{equation}
as $h_{n'}\to 0^+$, uniformly for $t$ bounded from below; preserving uniformity, the expansion can repeatedly be differentiated w.r.t. the variable $t$. For the quantities $t_\nu$ and $\tilde h_\nu$, defined by
\[
x = \mu_n + \sigma_n t =: \mu_{n_-+\nu} + \sigma_{n_-+\nu} \, t_\nu, \quad \tilde h_\nu := h_{n_- + \nu} \qquad (\nu=0,1),
\]
a routine calculation gives expansions in $h=h_n$ (the brackets are power series  in $h^{3/2}$):
\begin{gather*}
t_\nu = t\left(1 \mp\frac{2}{3} h^{3/2}-\frac{10}{9} h^3 \mp \cdots \right)\pm h^{1/2}\left(1 \pm \frac{1}{3}h^{3/2}+\frac{2}{9} h^3 \pm \cdots \right), \\
\tilde h_\nu = h\left(1 \pm \frac83 h^{3/2} + \frac{80}{9} h^3 \pm \cdots \right),
\end{gather*}
where the upper signs apply to $\nu=0$ and the lower signs to $\nu=1$.
In view of \eqref{eq:alpha2delta}, this specifies the parameters
\[
\alpha_1 = -\frac{2}{3},\quad \alpha_2 = -\frac{10}{9}, \quad \beta_1 = \frac{1}{3}, \quad \gamma_1 = \frac{8}{3}.
\]
\begin{figure}[tbp]
\includegraphics[width=0.325\textwidth]{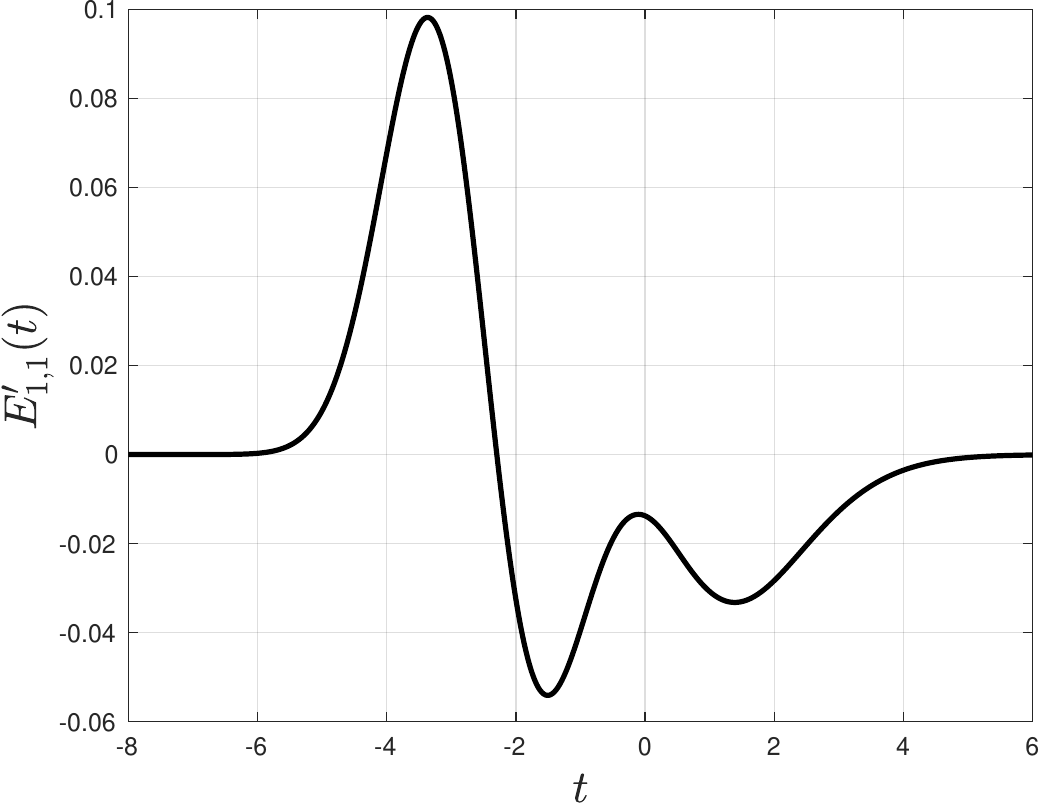}\hfil\,
\includegraphics[width=0.325\textwidth]{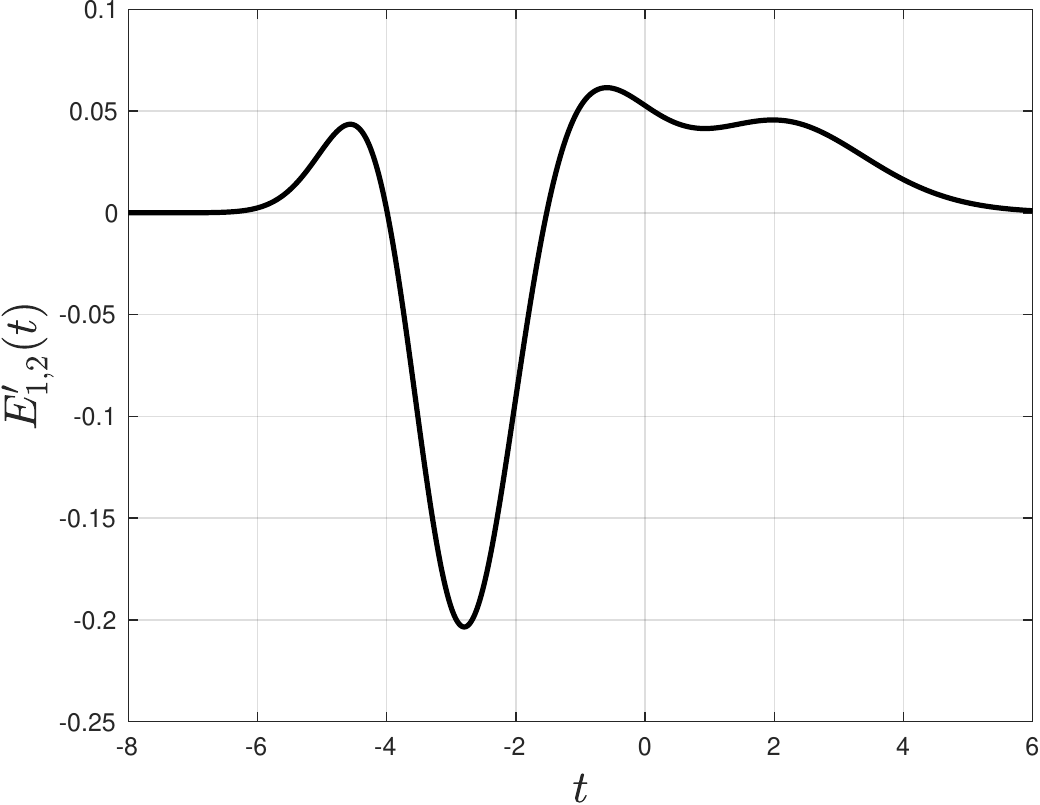}\hfil\,
\includegraphics[width=0.325\textwidth]{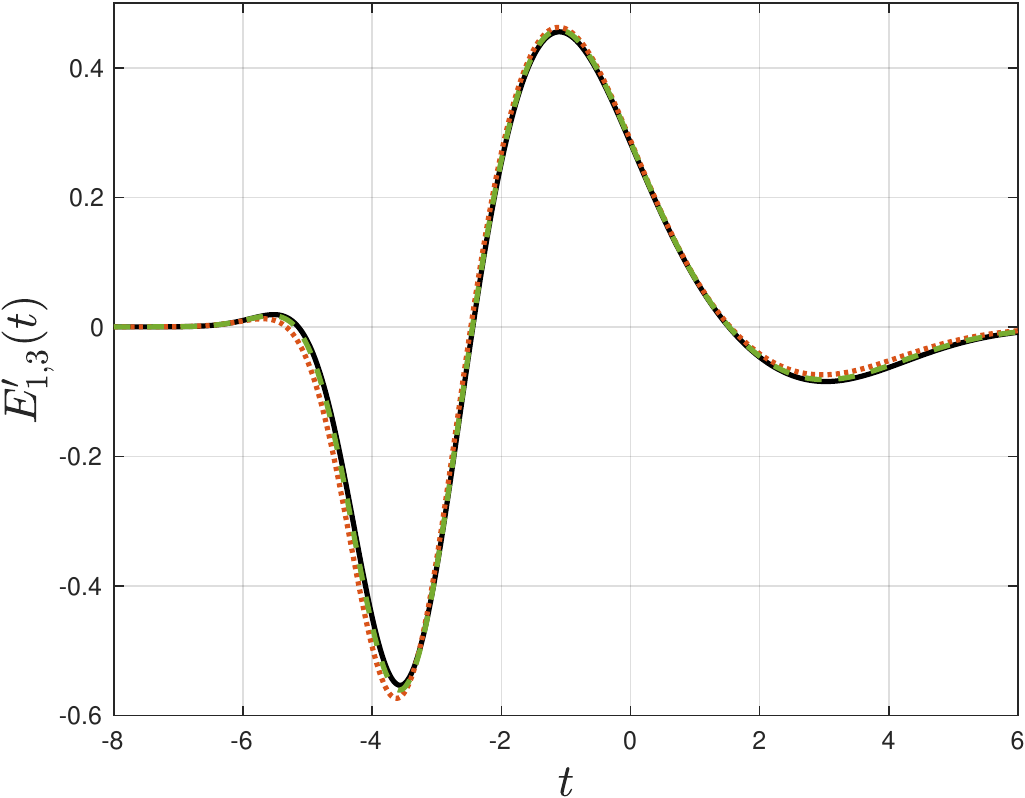}\\*[2mm]
\includegraphics[width=0.325\textwidth]{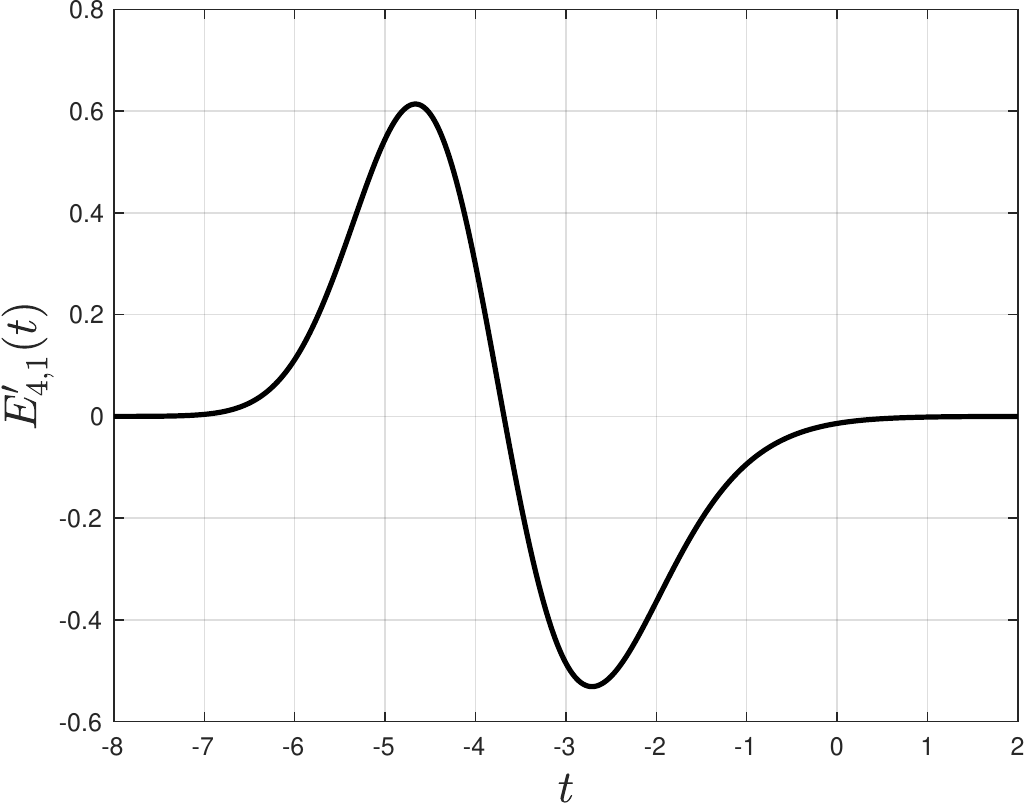}\hfil\,
\includegraphics[width=0.325\textwidth]{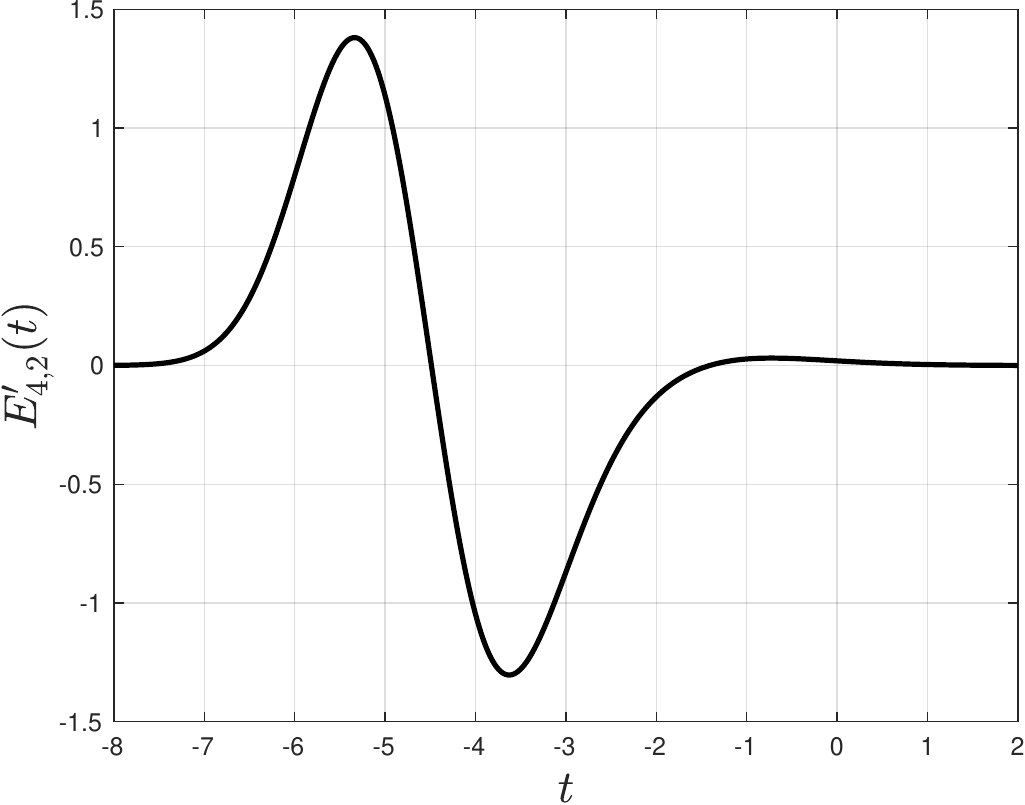}\hfil\,
\includegraphics[width=0.325\textwidth]{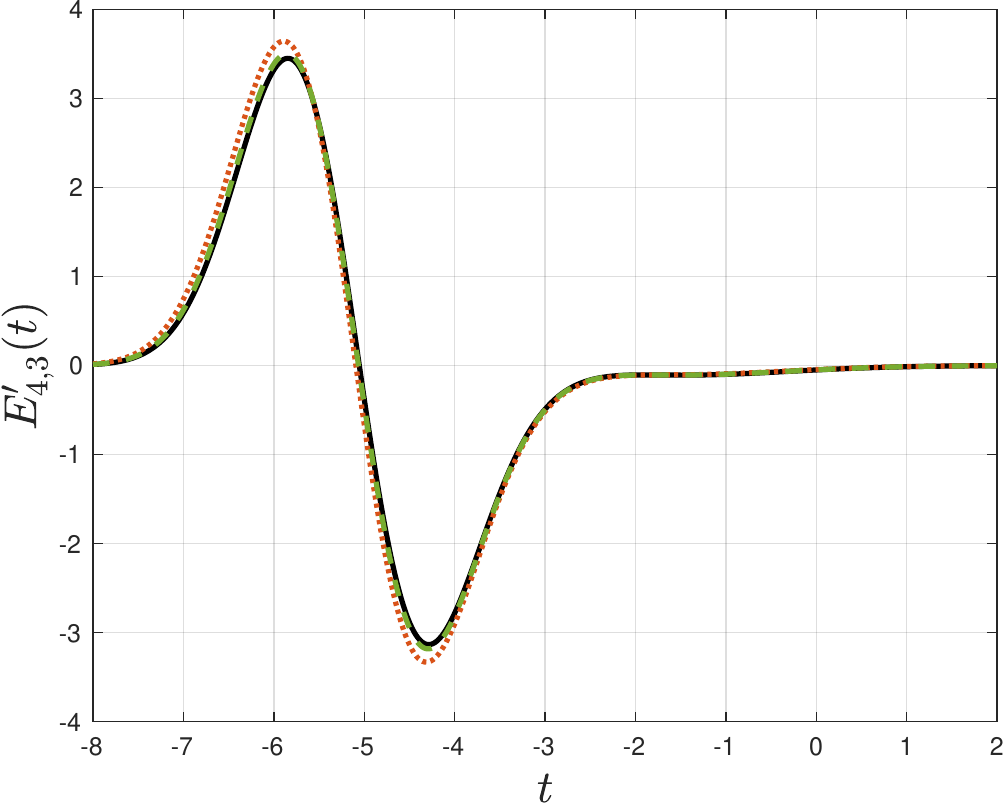}
\caption{{\footnotesize Top row: GOE ($\beta=1$). Bottom row: GSE ($\beta=4$). Plots of $E_{\beta,1}'(t)$ (left panel) and $E_{\beta,2}'(t)$ (middle panel).
The right panel shows $E_{\beta,3}'(t)$ (black solid line) with the approximations \eqref{eq:Fbeta3} for $n=10$ (red dotted line) and $n=80$ (green dashed line): the close match validates the functional forms displayed in~\eqref{eq:Fbeta2} and the differentiability of  the expansion \eqref{eq:GbetaE}.}}
\label{fig:GOEGSE}
\end{figure}%
So,  the formulas for the $p_{+,jk}$ in Section~\ref{subsect:multilinear}, with the polynomials $p_{2,jk}$ taken from  Theorem~\ref{thm:softGUE}, yield the first three instances of \eqref{eq:FbetaLinForm} -- namely, for $\beta=1,4$,
\begin{subequations}\label{eq:Fbeta2}
\begin{align}
E_{\beta,1}(t) &= \frac{t^2}{5} F_\beta'(t) -\frac{3}{5} F_\beta''(t),\label{eq:Fbeta1}\\*[1mm]
E_{\beta,2}(t) &=-\bigg(\frac{186}{175} + \frac{8t^3}{175}\bigg) F_\beta'(t) + \bigg(\frac{78t}{175} + \frac{t^4}{50}\bigg) F_\beta''(t) - \frac{3t^2}{25} F_\beta'''(t) + \frac{9}{50} F_\beta^{(4)}(t),\\*[1mm]
E_{\beta,3}(t) &= 
\bigg(\frac{6392 t}{7875} + \frac{148 t^4}{7875}\bigg)  F_\beta'(t)
-\bigg( \frac{292t^2}{525}  + \frac{8t^5}{875} \bigg) F_\beta''(t) \\*[0.25mm]
&\qquad +\bigg(\frac{11618}{7875} +\frac{102 t^3}{875} + \frac{t^6}{750}\bigg) F_\beta'''(t)
-\bigg(\frac{234t}{875} +\frac{3 t^4}{250} \bigg)F_\beta^{(4)}(t)\notag\\*[0.25mm]
&\qquad +\frac{9 t^2 }{250}F_\beta^{(5)}(t)
-\frac{9}{250} F_\beta^{(6)}(t).\notag
\end{align}
\end{subequations}

Using numerical methods we can validate the formulas in (\ref{eq:Fbeta2}), and the differentiability of the expansion \eqref{eq:GbetaE}: Figure~\ref{fig:GOEGSE} plots, for $\beta=1,4$,  the functions $E_{\beta,1}'(t)$, $E_{\beta,2}'(t)$, $E_{\beta,3}'(t)$ next to the approximation
\begin{equation}\label{eq:Fbeta3}
E_{\beta,3}'(t) \approx h_{n'}^{-3} \cdot \big( \tfrac{d}{dt}E_\beta\big(n;\mu_{n'} + \sigma_{n'} t\big) - F_\beta'(t) - E_{\beta,1}'(t)h_{n'} - E_{\beta,2}'(t) h_{n'}^2\big).
\end{equation}
For matrix dimensions $n=10$ and $n=80$, we observe a close match. Note that, for the GOE, the absolute scale of the first expansion terms is of order $1/10$, which explains that in actual statistical applications the accuracy of the limit law is more than sufficient. 

\begin{remark} Choup \cite[Theorem~1.2]{MR2492622} claimed (with no validation whatsoever) an expansion of the form 
\[
E_1(n;\mu_{n'} + \sigma_n t) = F_1(t) + H(t) n^{-1/3} + J(t) n^{-2/3} + O(n^{-1}),
\]
expressing $H(t), J(t)$ in terms of quantities from the work of Tracy and Widom \cite{MR1385083}.  Since his very involved expression for $H(t)$ is not identically zero (as a plot makes clear), the result would contradict the $O(n^{-2/3})$ convergence rate as  established by Johnstone and Ma -- who state \cite[p.~1963]{MR3025686} to have learned of an error in Choup's argument by a private communication. 
\end{remark}

\begin{figure}[tbp]
\includegraphics[width=0.325\textwidth]{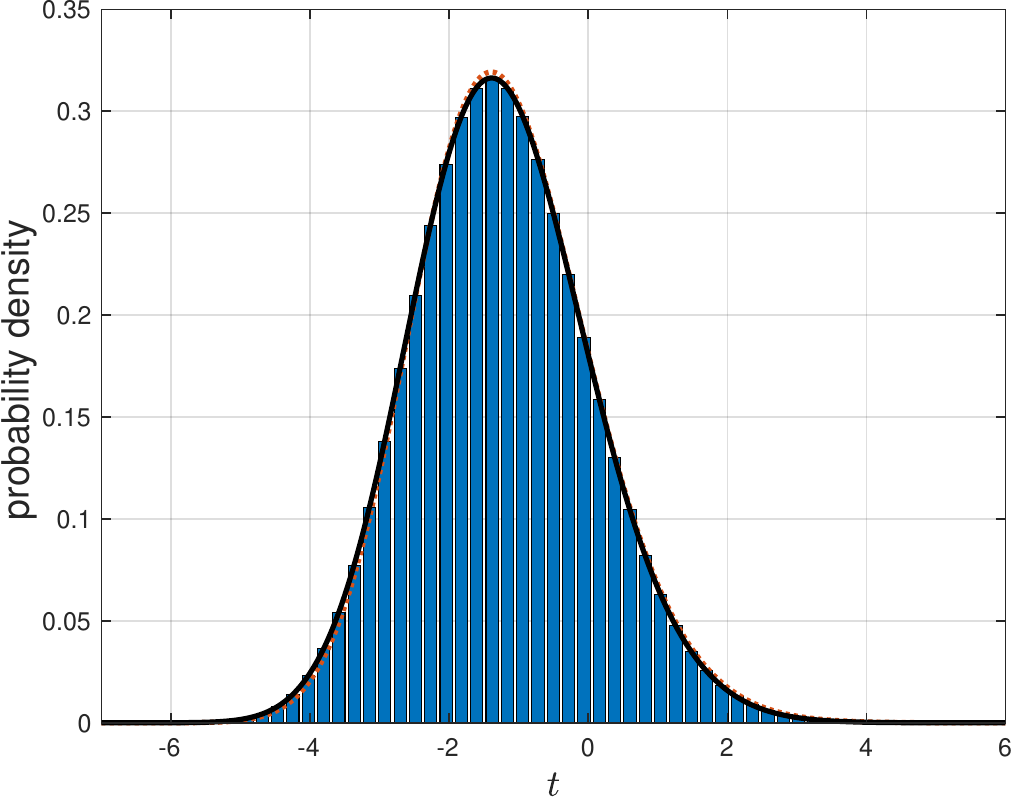}\hfil\,
\includegraphics[width=0.325\textwidth]{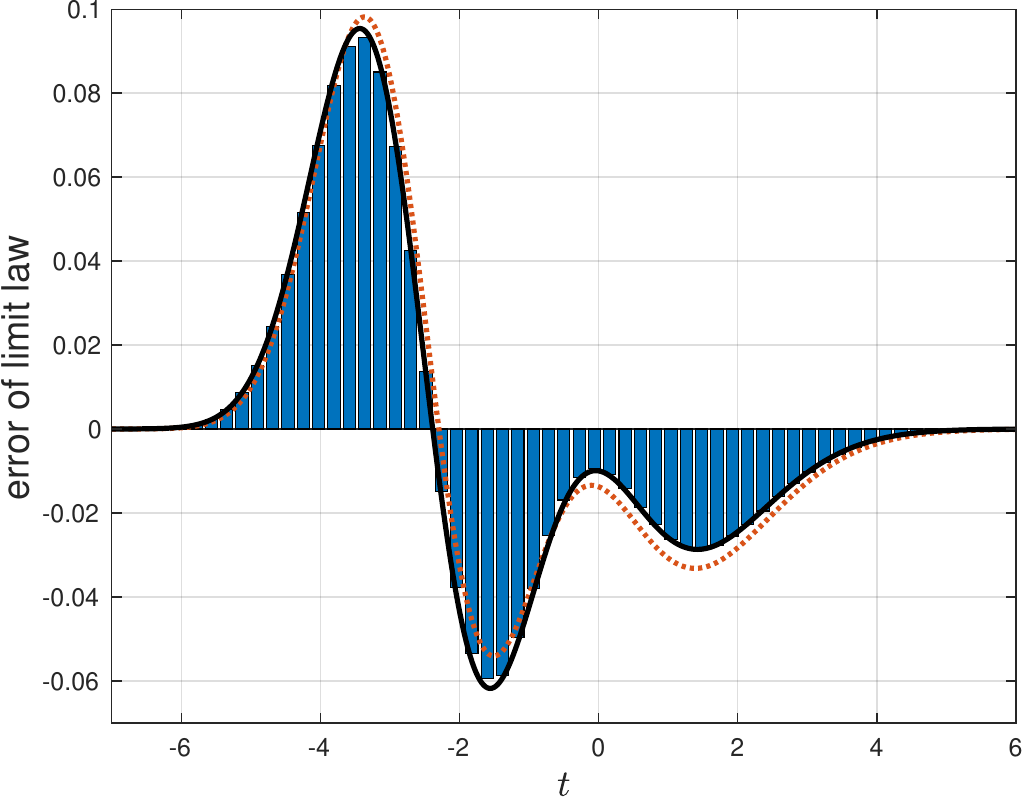}\hfil\,
\includegraphics[width=0.325\textwidth]{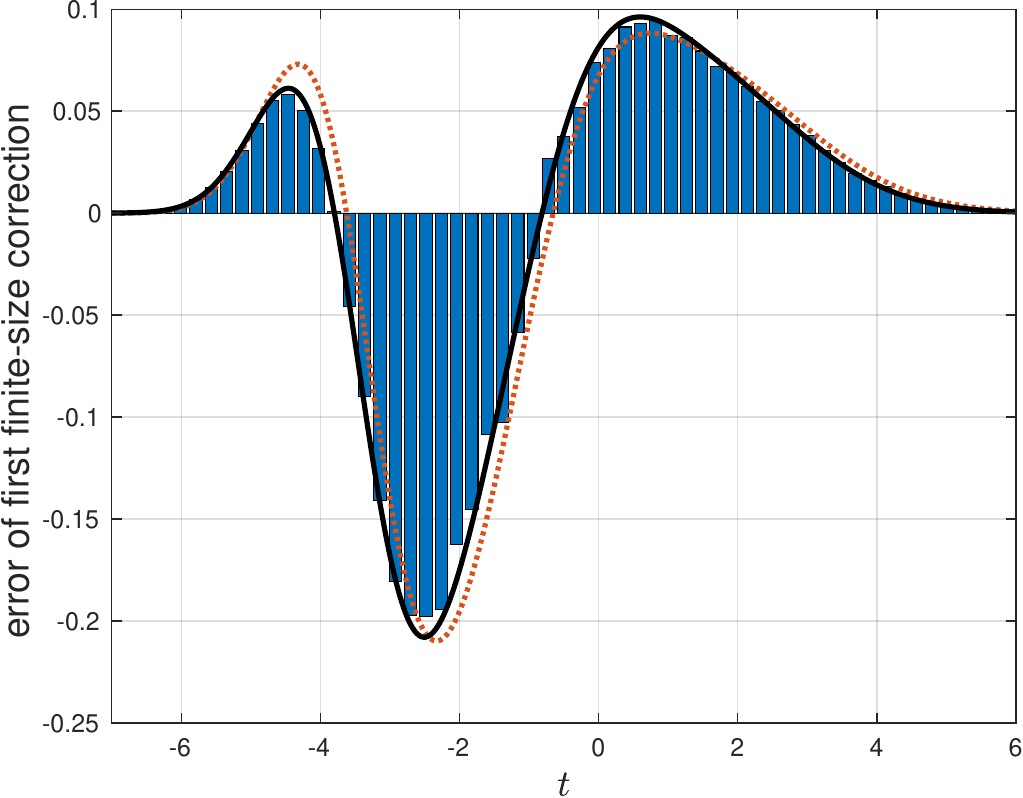}\\*[2mm]
\includegraphics[width=0.325\textwidth]{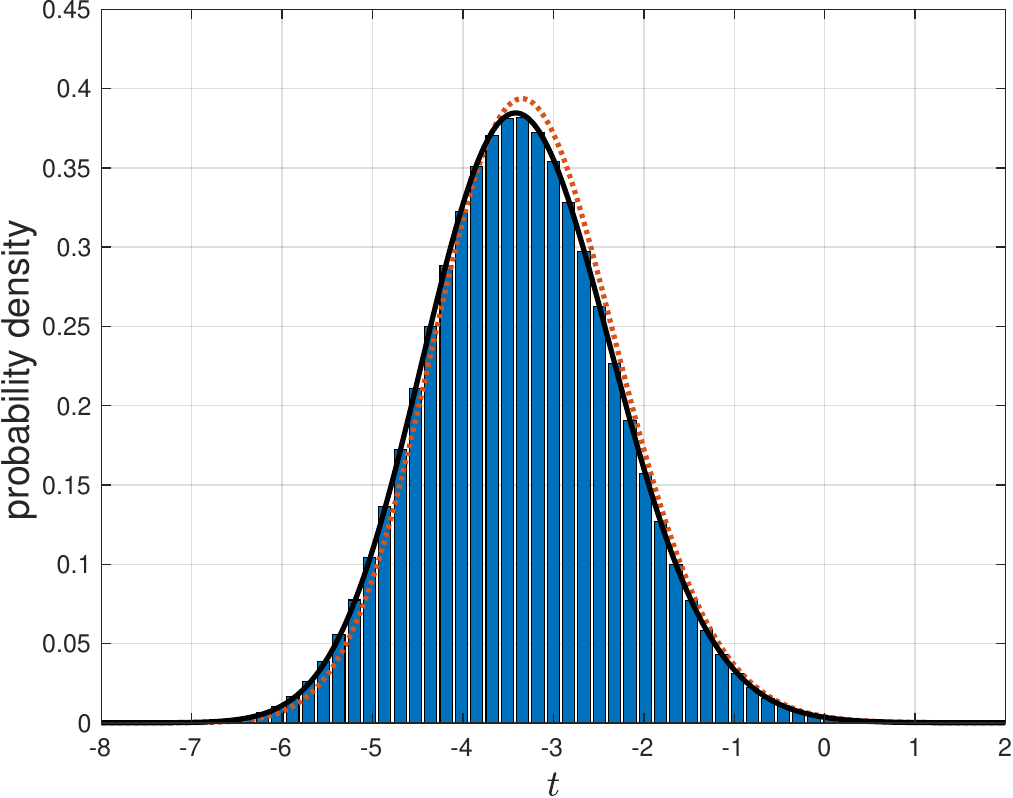}\hfil\,
\includegraphics[width=0.325\textwidth]{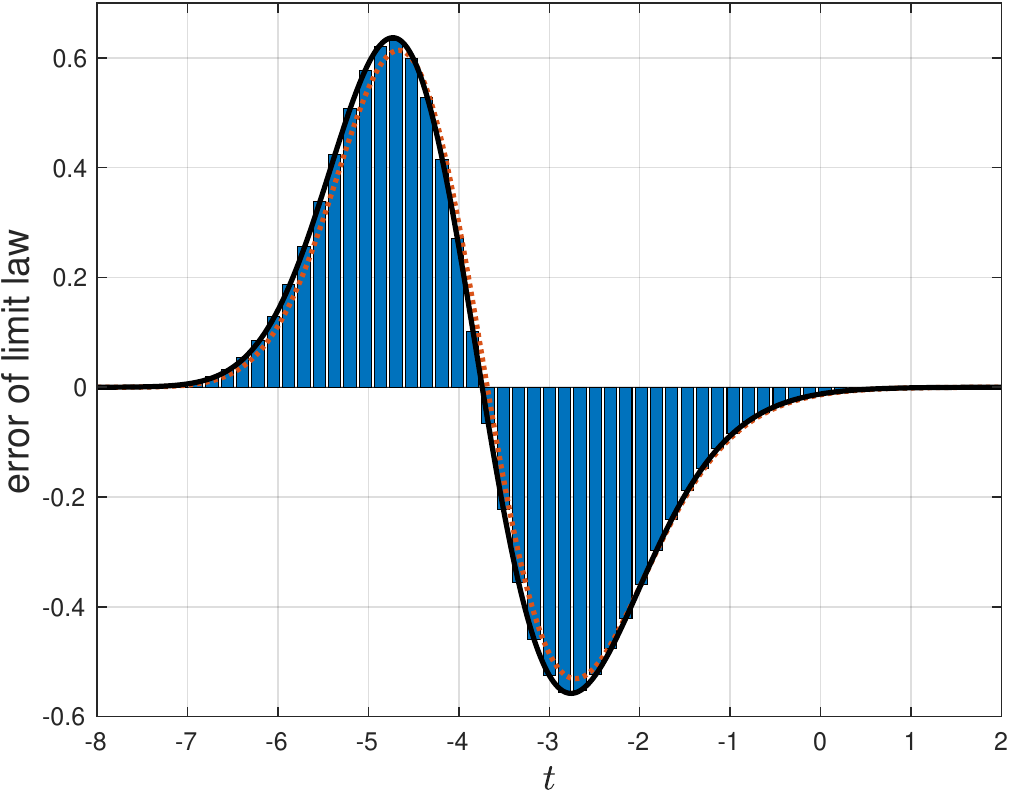}\hfil\,
\includegraphics[width=0.325\textwidth]{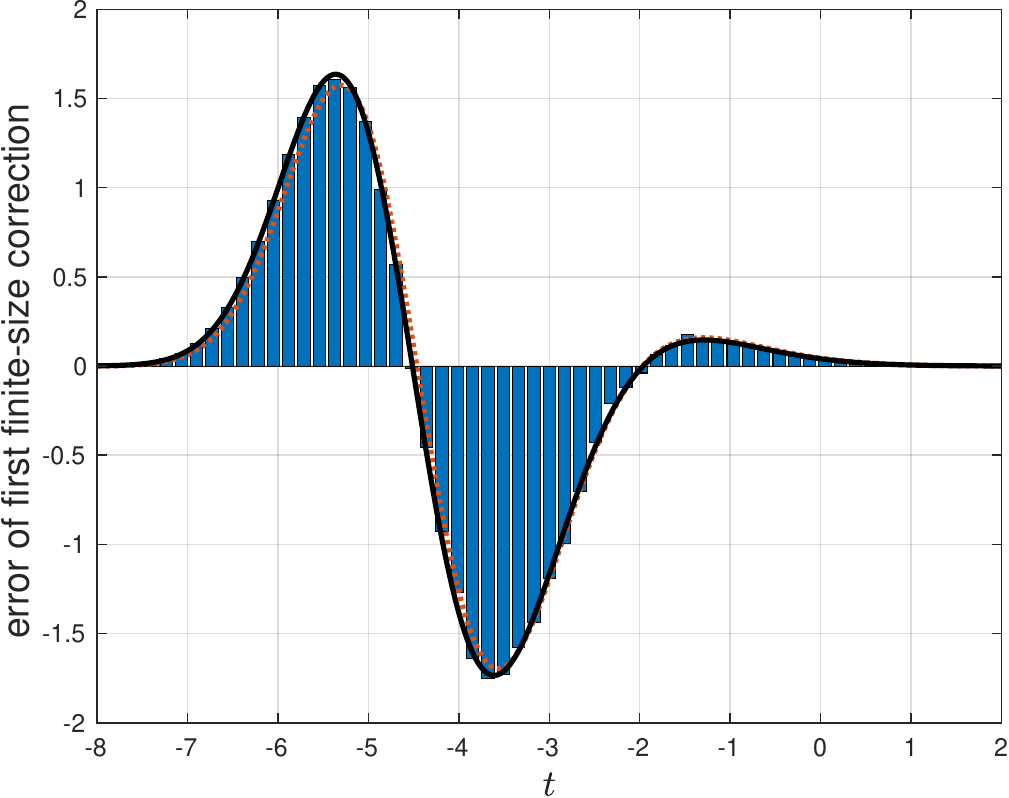}
\caption{{\footnotesize [Expansion terms with a tilde are  `histogram adjusted', see Appendix~\ref{app:histo}.] Top row: GOE ($\beta=1$). Bottom row: GSE ($\beta=4$). Left panels: histogram  of the scaled largest eigenvalues $(\lambda_{\max} - \mu_{n'})/\sigma_{n'}$, for $N=10^9$ draws, from the $\GbE$ for matrix dimension $n=10$ (blue bars) vs. the Tracy--Widom limit density $F_\beta'$ (red dotted line) 
and its first finite-size correction $F_\beta' + h_{n'} E_{\beta,1}'$ (black solid line). Middle panels: difference, scaled by $h_{n'}^{-1}$,  between the histogram midpoints and the limit density (blue bars) vs. the correction terms $E_{\beta,1}'$ (red dotted line) and $E_{\beta,1}' + h_{n'} \tilde E_{\beta,2}'$ (black solid line). Right panels: difference, scaled by $h_{n'}^{-2}$, between 
the histogram midpoints and the first finite-size correction (blue bars) vs. the correction terms $\tilde E_{\beta,2}'$ (red dotted line) and $\tilde E_{\beta,2}' + h_{n'} \tilde E_{\beta,3}'$ (black solid line).}}
\label{fig:GOEGSESimulation}
\end{figure}

\section{Expansions of the LOE (real Wishart) and LSE (quaternionic Wishart)}\label{sect:LOELSE}

In specifying the general approach of Section~\ref{sect:genAppr} (assuming the hypotheses made there), using the notation and  result of Theorem~\ref{thm:softLUE}, while  introducing $n',p'$ as in \eqref{eq:nprime}, we get
\begin{equation}\label{eq:LbetaE}
E_\beta\big(n,p;\mu_{n',p'} + \sigma_{n',p'} t\big) = F_\beta(t) + \sum_{j=1}^m E_{\beta,j;\tau_{n',p'}}(t) h_{n',p'}^j + h_{n',p'}^{m+1}\cdot O(e^{-t}) \quad (\beta=1,4)
\end{equation}
as $h_{n',p'}\to 0^+$, uniformly for $t$ bounded from below and  $\tau_{n',p'}$  bounded away from zero; preserving uniformity, the expansion can repeatedly be differentiated w.r.t. the variable $t$. For the quantities $t_\nu$, $\tilde h_\nu$ and $\tau_\nu$, defined by ($\nu=0,1$)
\[
x = \mu_{n,p} + \sigma_{n,p} t =: \mu_{n_-+\nu,p_-+\nu} + \sigma_{n_-+\nu,p_-+\nu} \, t_\nu, \quad \tilde h_\nu := h_{n_- + \nu,p_-+\nu}, \quad \tau_\nu  := \tau_{n_- + \nu,p_-+\nu},
\]
a routine calculation gives the following expansions in $h=h_{n,p}$ (the brackets are power series  in $h^{3/2}$), briefly writing $\tau=\tau_{n,p}$\,:
\begin{gather*}
t_\nu = t\left(1 \pm\big(\tau - \tfrac{2}{3}\big) h^{3/2}+\big(\tfrac{4}{3}\tau -\tfrac{10}{9}\big) h^3 \pm \cdots \right)\pm h^{1/2}\left(1 \pm \tfrac{1}{3}h^{3/2}+\tfrac{2}{9} h^3 \pm \cdots \right), \\
\tilde h_\nu = h\left(1 \pm \big(\tfrac{8}{3}-2 \tau\big) h^{3/2} + \big(3 \tau ^2-\tfrac{34}{3}\tau +\tfrac{80}{9}\big) h^3 \pm \cdots \right),\\
\tau_\nu = \tau\left(1 \pm 2 (\tau -1) h^{3/2} +(\tau ^2+\tau -2) h^3 \pm \cdots\right),
\end{gather*}
where the upper signs apply to $\nu=0$ and the lower signs to $\nu=1$.
In view of \eqref{eq:alpha2delta}, this specifies the parameters
\[
\alpha_1 = \tau -\frac{2}{3},\quad \alpha_2 = \frac{4}{3}\tau -\frac{10}{9}, \quad \beta_1 = \frac{1}{3}, \quad \gamma_1 = \frac{8}{3} - 2\tau, \quad
\delta_1 = 2(\tau-1).
\]
\begin{figure}[tbp]
\includegraphics[width=0.325\textwidth]{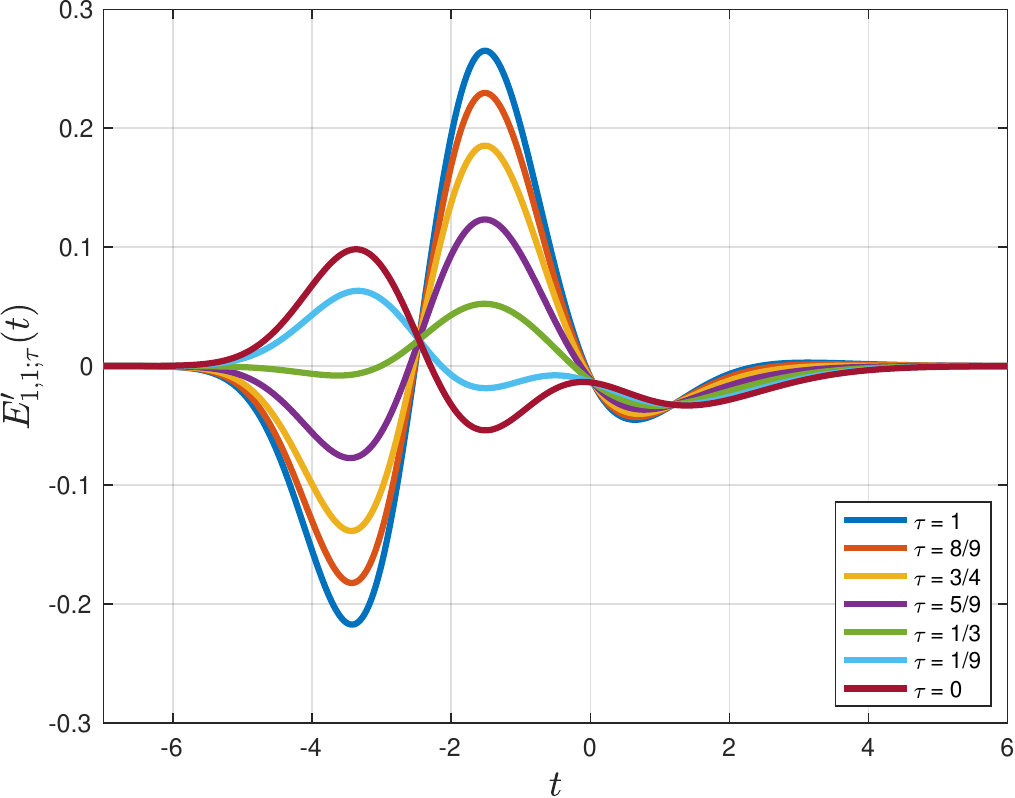}\hfil\,
\includegraphics[width=0.325\textwidth]{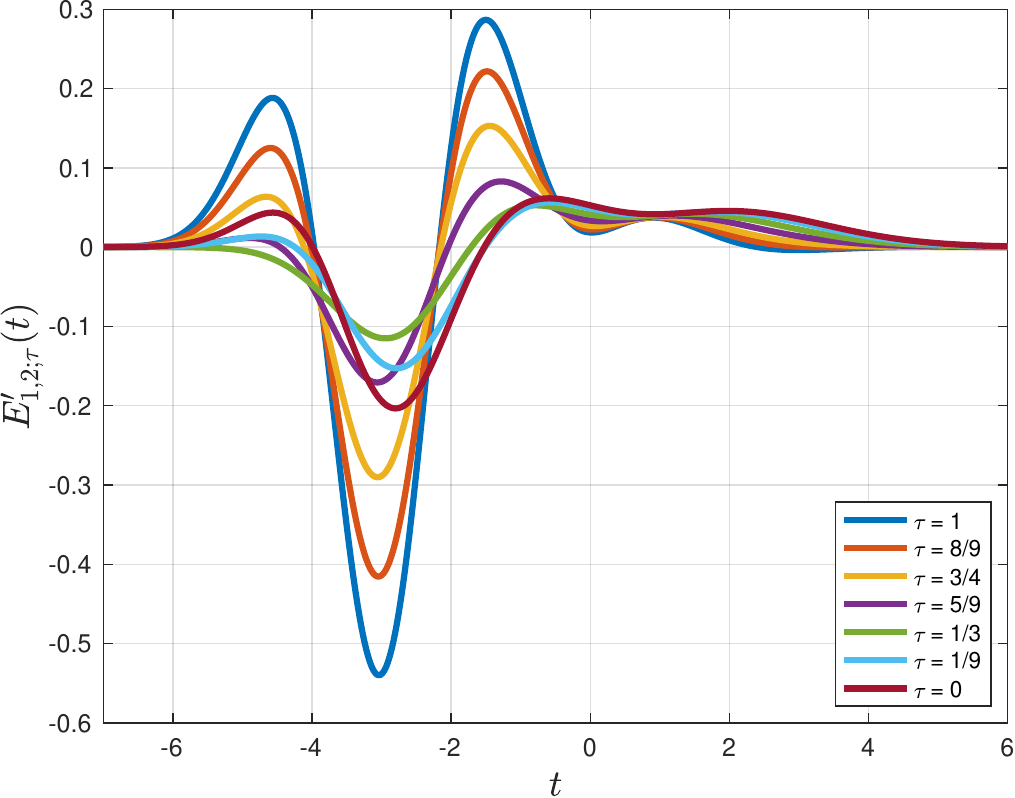}\hfil\,
\includegraphics[width=0.325\textwidth]{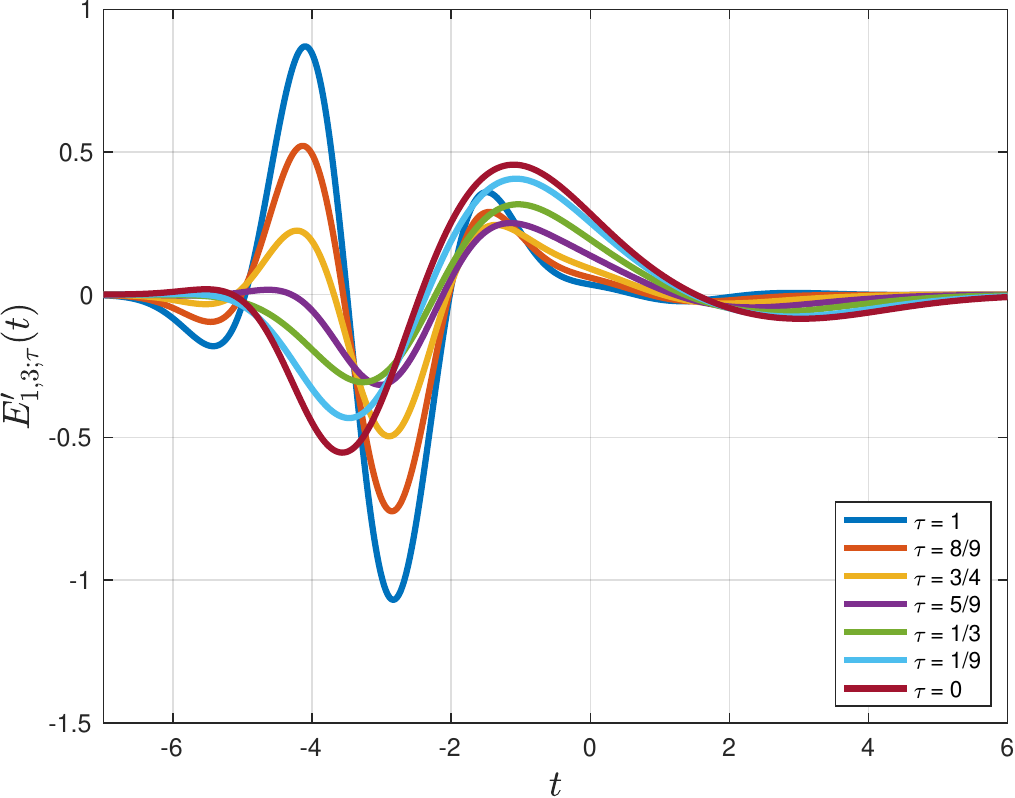}\\*[2mm]
\includegraphics[width=0.325\textwidth]{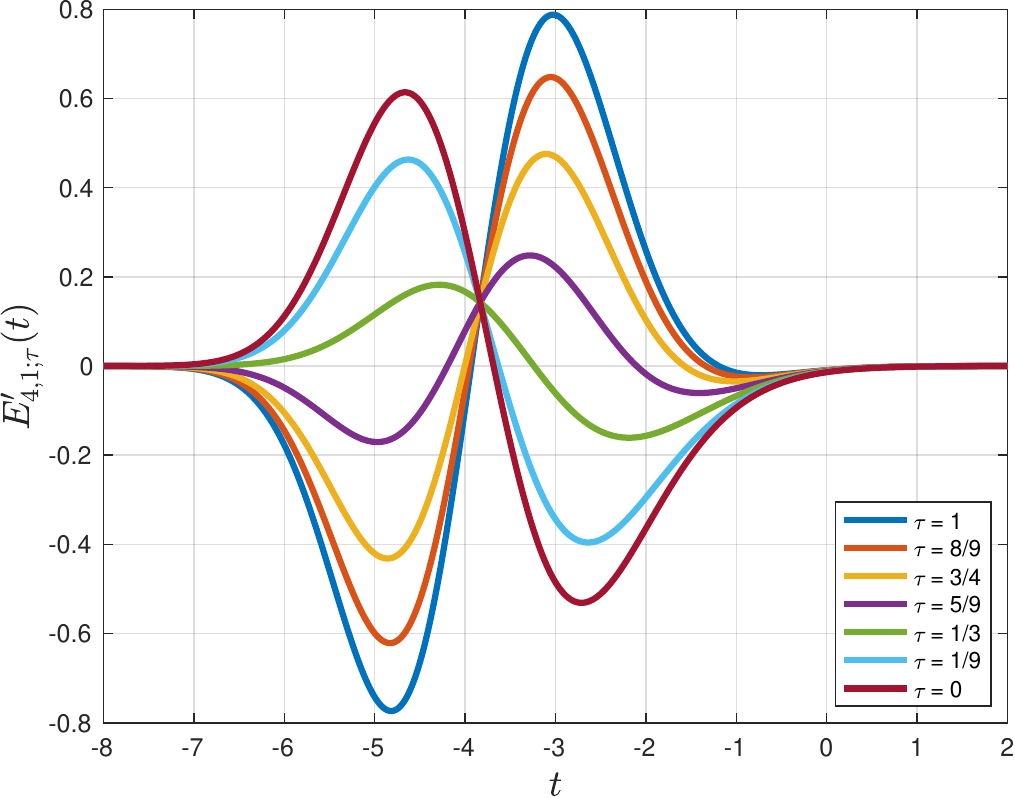}\hfil\,
\includegraphics[width=0.325\textwidth]{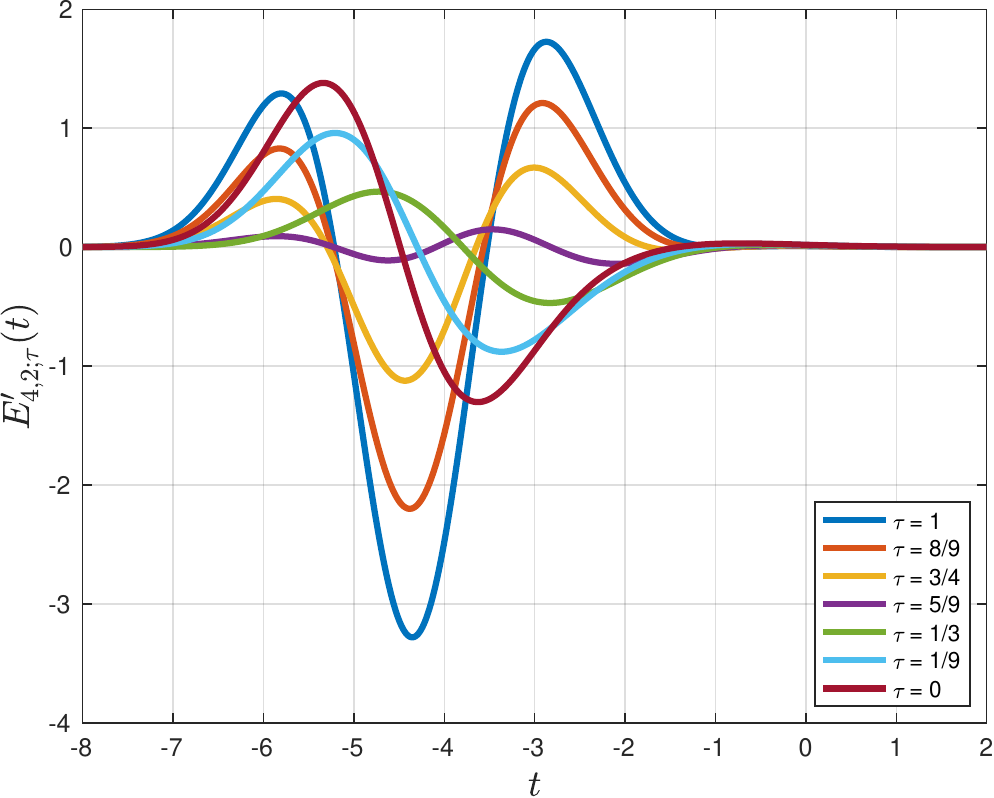}\hfil\,
\includegraphics[width=0.325\textwidth]{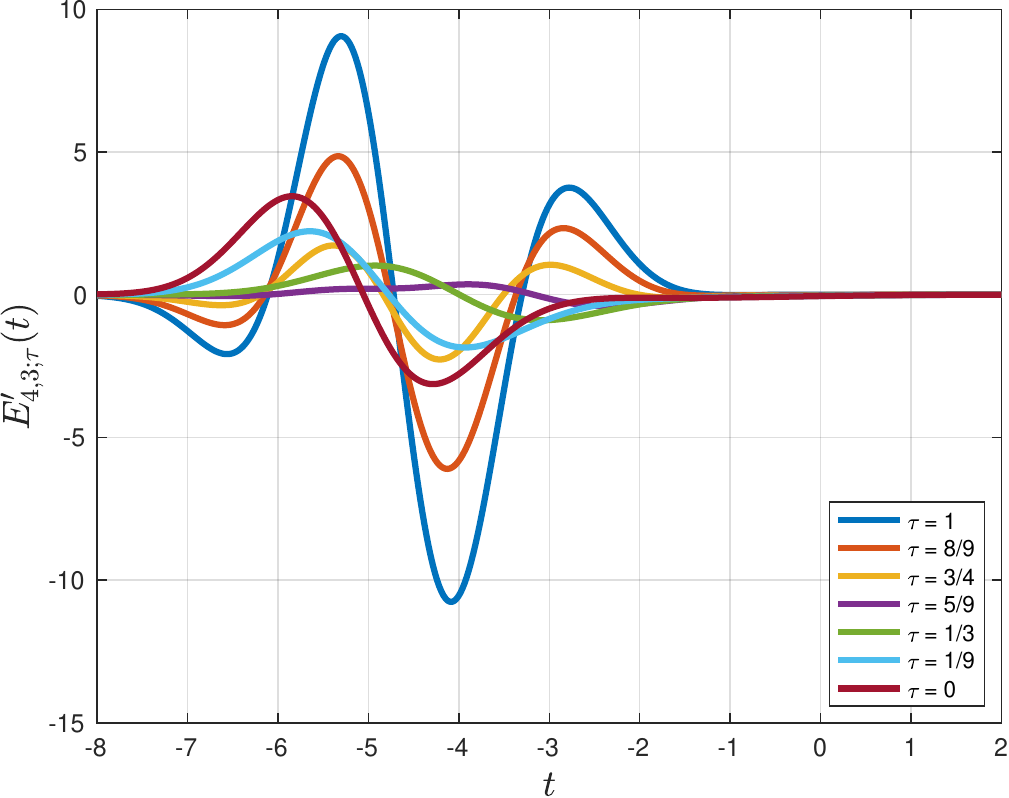}
\caption{{\footnotesize Top row: LOE ($\beta=1$). Bottom row: LSE ($\beta=4$). Plots of $E_{\beta,1;\tau}'(t)$ (left panel), $E_{\beta,2;\tau}'(t)$ (middle panel) 
and $E_{\beta,3;\tau}'(t)$ (right panel) for $\tau\in\{1, 8/9, 3/4, 5/9,1/3,1/9,0\}$. In the case $p\geq n$, these values of $\tau$ correspond to the ratios $p=n$, $p\approx 4n$, $p \approx 9n$,  $p\approx 25 n$, $p\approx 100n$, $p\approx 1000 n$, $p=\infty$. See Figure~\ref{fig:LOELSESimulation} for a comparison with sampling data for $p=4n$.}}
\label{fig:LOELSE}
\end{figure}%
So,  the formulas for the $p_{+,jk}$ in Section~\ref{subsect:multilinear}, with the polynomials $p_{2,jk}$ taken from Theorem~\ref{thm:softLUE}, yield the first three instances of \eqref{eq:FbetaLinForm} -- namely, for $\beta=1,4$,
\begin{subequations}\label{eq:F22LbetaE}
\begin{align}
E_{\beta,1;\tau}(t) &= -\frac{2 \tau -1}{5} t^2 F_\beta'(t) + \frac{\tau -3}{5} F_\beta''(t),\label{eq:F21LbetaE}\\*[1mm]
E_{\beta,2;\tau}(t) &=\bigg(\frac{9\tau ^2+496 \tau -744}{700} +\frac{43 \tau ^2-18 \tau -8 }{175}t^3\bigg) F_\beta'(t)\\*[0.5mm]
&\qquad\qquad  -\bigg(\frac{2(4 \tau ^2+26 \tau -39)}{175}t  - \frac{(2 \tau -1)^2}{50} t^4\bigg)F_\beta''(t)\notag\\*[0.5mm]
&\qquad\qquad\qquad -\frac{(\tau -3)(2 \tau -1) }{25}t^2 F_\beta'''(t) + \frac{(\tau -3)^2}{50} F_\beta^{(4)}(t),\notag\\*[1mm]
\intertext{}\notag\\*[-13mm]
E_{\beta,3;\tau}(t) &= -\bigg(\frac{ 67 \tau ^3-1778 \tau ^2+12784 \tau -12784)}{15750}t\\*[0.5mm]
&\qquad\quad +\frac{1384 \tau ^3-551 \tau ^2-212 \tau -148}{7875}t^4\bigg)F'_\beta(t)\notag\\*[0.5mm]
&\hspace*{-2.5cm} + \bigg(\frac{ 42 \tau ^3-449 \tau ^2+1600 \tau -1168}{2100}t^2
-\frac{(2 \tau -1)(43 \tau ^2-18 \tau -8)}{875}t^5\bigg) F_\beta''(t)\notag\\*[0.5mm]
&\hspace*{-1cm} +\bigg(\frac{289 \tau ^3+6349 \tau ^2-46472 \tau +46472 }{31500}\notag\\*[0.5mm]
&\;  +\frac{59 \tau ^3-51 \tau ^2-162 \tau +102}{875} t^3-\frac{(2 \tau -1)^3 }{750} t^6\bigg)F_\beta'''(t)\notag\\*[0.5mm]
&\qquad -\bigg( \frac{2(\tau -3) \left(4 \tau ^2+26 \tau -39\right)}{875}t
-\frac{(\tau -3) (2 \tau -1)^2 }{250} t^4\bigg)F_\beta^{(4)}(t)\notag\\*[0.5mm]
&\qquad\qquad\qquad\qquad\qquad -\frac{ (\tau -3)^2 (2 \tau -1)}{250}t^2 F_\beta^{(5)}(t) + \frac{(\tau -3)^3 }{750}F_\beta^{(6)}(t).\notag
\end{align}
\end{subequations}
\begin{figure}[tbp]
\includegraphics[width=0.325\textwidth]{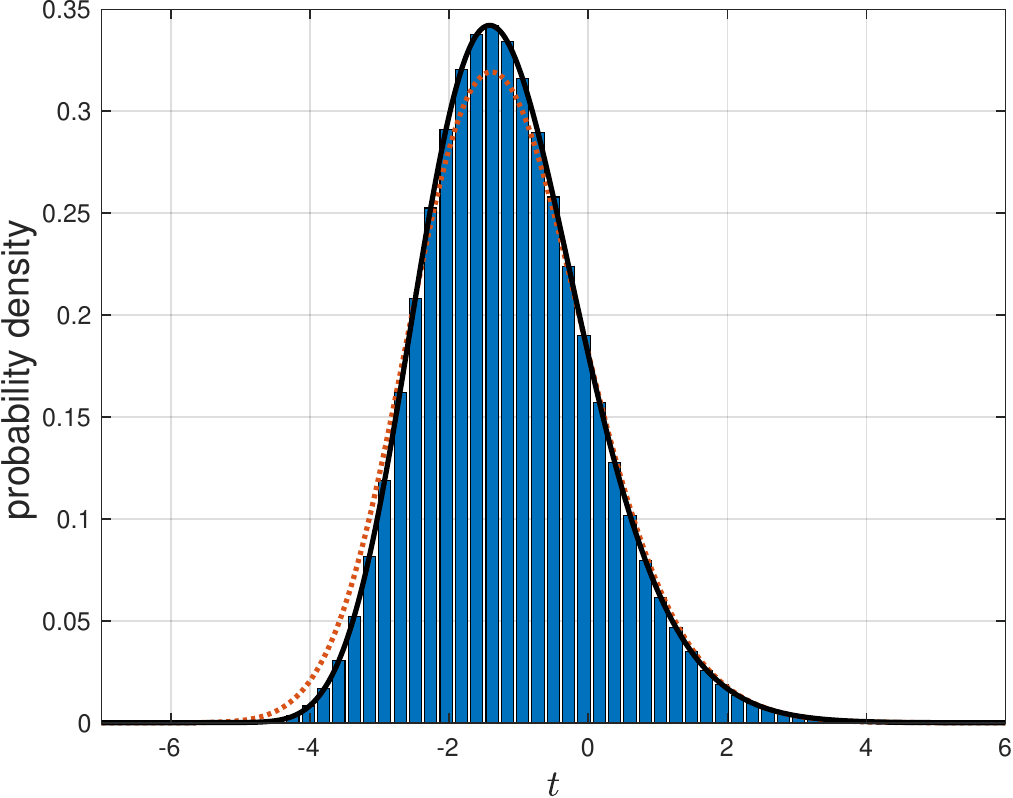}\hfil\,
\includegraphics[width=0.325\textwidth]{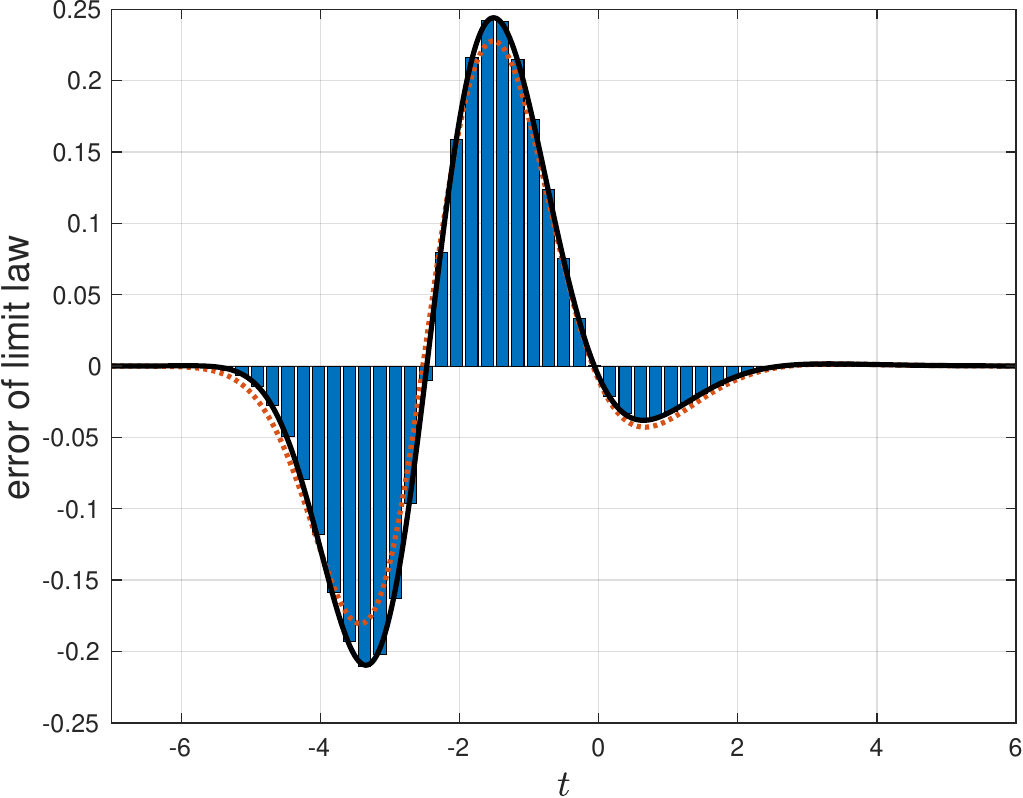}\hfil\,
\includegraphics[width=0.325\textwidth]{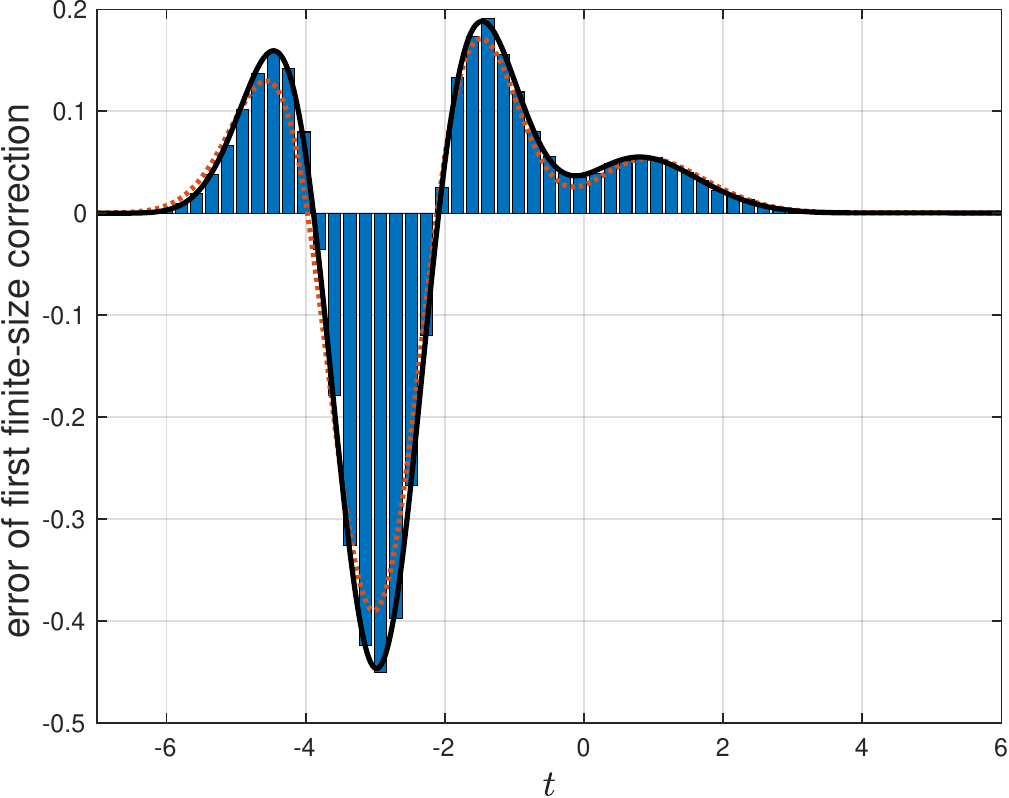}\\*[2mm]
\includegraphics[width=0.325\textwidth]{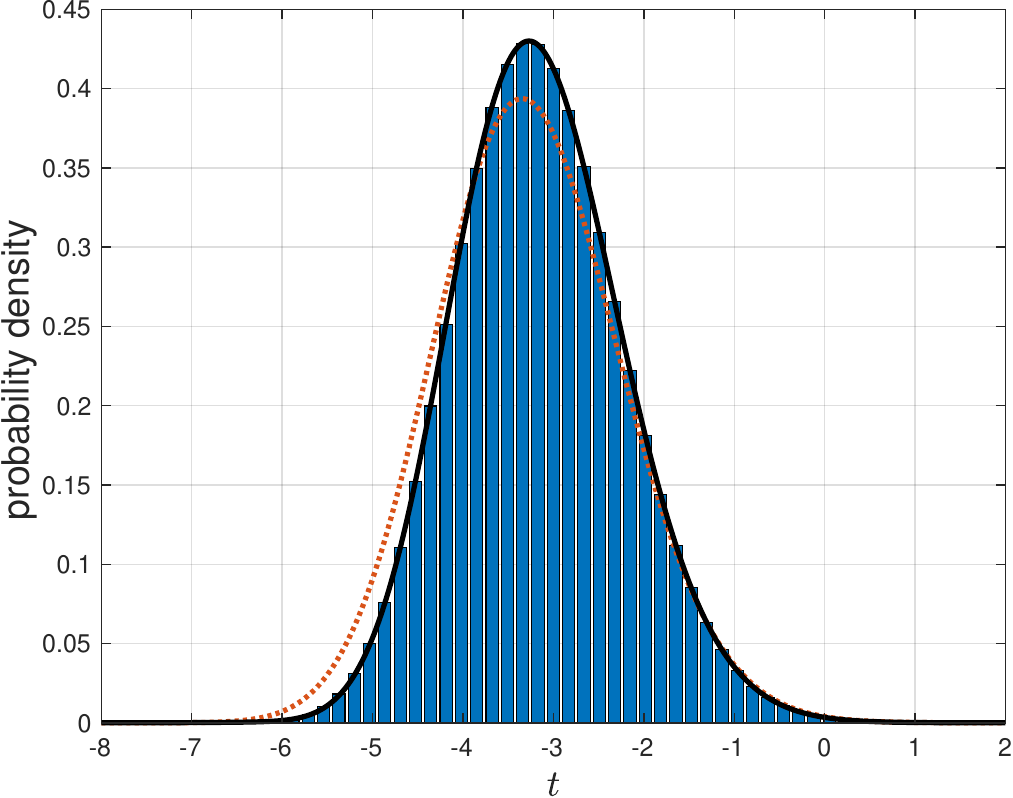}\hfil\,
\includegraphics[width=0.325\textwidth]{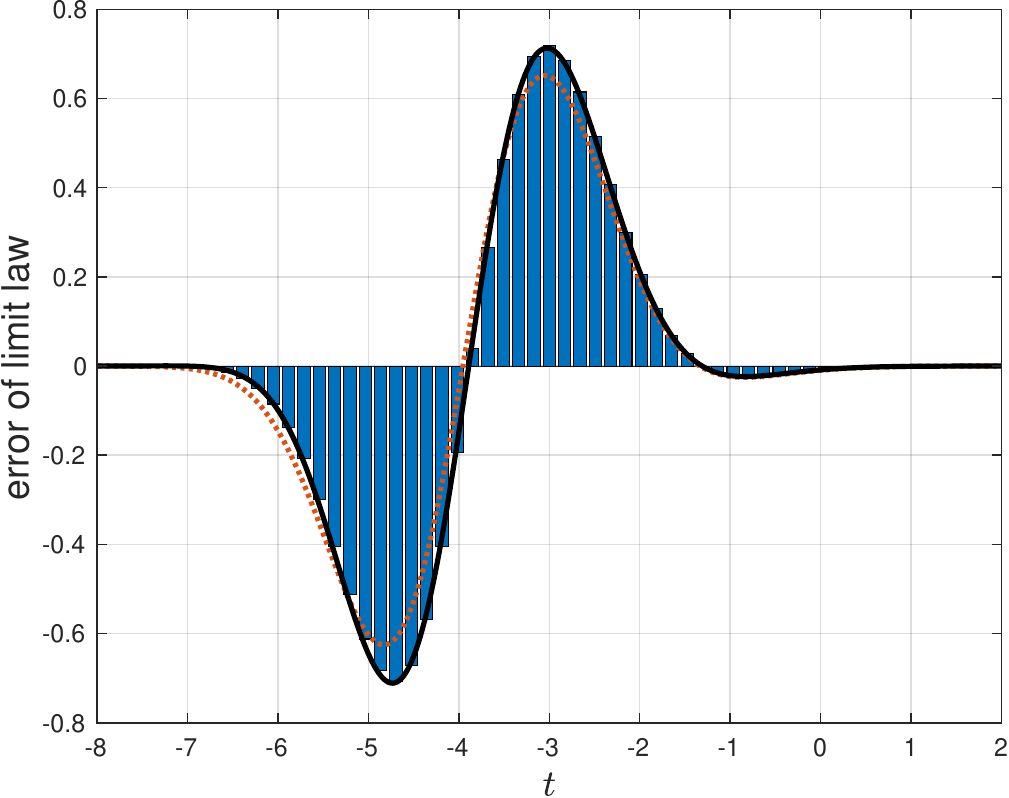}\hfil\,
\includegraphics[width=0.325\textwidth]{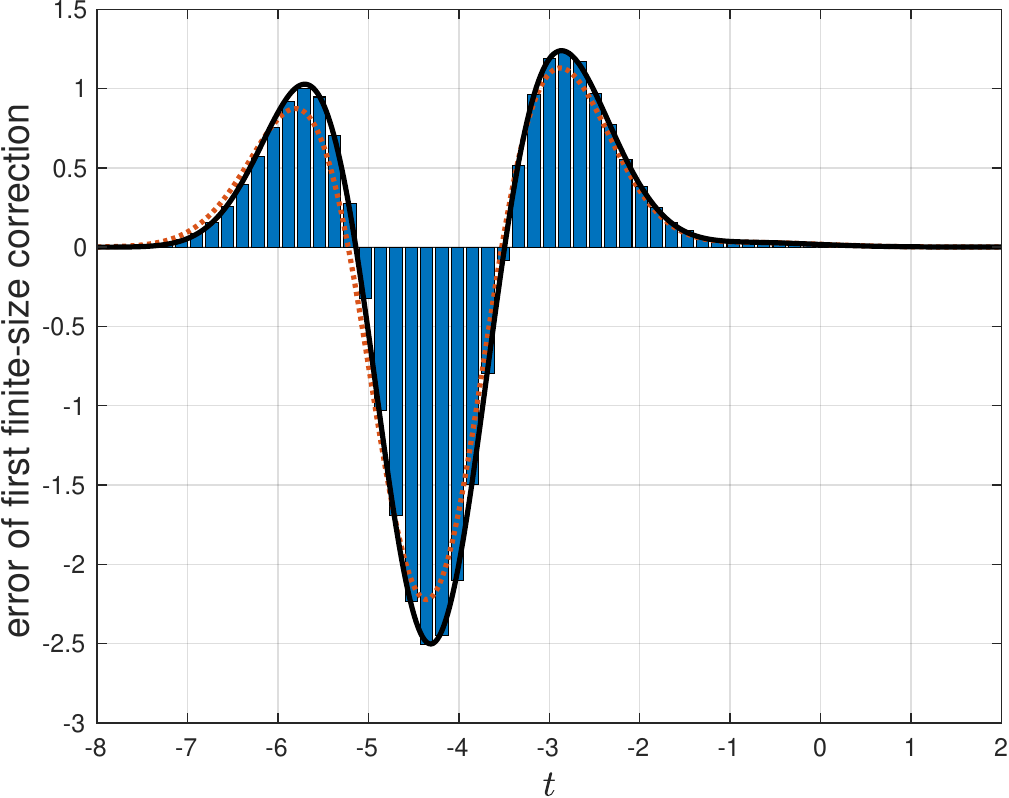}
\caption{{\footnotesize [Expansion terms with a tilde are  `histogram adjusted', see Appendix~\ref{app:histo}.] Comparison
of the expansion \eqref{eq:LbetaE} with $N=10^9$ draws from the $\LbE$ with $n=10$, $p=40$. 
Top row: LOE ($\beta=1$, $\tau=\tau_{n',p'}=0.88309\cdots$, $h=h_{n',p'}=0.094885\cdots$). Bottom row: LSE ($\beta=4$, $\tau=\tau_{n',p'}=0.89161\cdots$, $h=h_{n',p'}=0.057545\cdots$). 
Left panel: histogram of the scaled largest eigenvalues $(\lambda_{\max} - \mu_{n',p'})/\sigma_{n',p'}$ from the $\LbE$ (blue bars) vs. the Tracy--Widom limit density $F_\beta'$ (red dotted line) 
and its second order correction $F_\beta' + h E_{\beta,1;\tau}' + h^2 \tilde E_{\beta,2;\tau}'$ (black solid line). Middle panel: difference, scaled by $h^{-1}$,  between the histogram midpoints and the limit density (blue bars) vs. the correction terms $E_{\beta,1;\tau}'$ (red dotted line) and $E_{\beta,1;\tau}' + h \tilde E_{\beta,2;\tau}'$ (black solid line). Right panel: difference, scaled by $h^{-2}$, between 
the histogram midpoints and the first finite-size correction (blue bars) vs. the correction terms $\tilde E_{\beta,2;\tau}'$ (red dotted line) and $\tilde E_{\beta,2;\tau}' + h \tilde E_{\beta,3;\tau}'$ (black solid line).}}
\label{fig:LOELSESimulation}
\end{figure}%
\begin{remark}
By the very construction of the coefficient functions $E_{\beta,j}$ and $E_{\beta,j;\tau}$, we get from Remark~\ref{rem:LUE2GUE} that, locally uniform in $t$,
\[
h_{n',\infty} := \lim_{p\to\infty} h_{n',p'} = h_{n'},\quad 
 \lim_{p \to\infty}E_{\beta,j;\tau_{n',p'}}(t) = E_{\beta,j;0}(t)=  E_{\beta,j}(t).
\]
If the expansion \eqref{eq:LbetaE} is uniformly valid for all $0<\tau = \tau_{n',p'} \leq 1$ (which we conjecture to be the case), 
this limit relation is consistent with the  $\LbE$-to-$\GbE$ transition law
\[
\lim_{p\to\infty} E_\beta\big(n,p;\mu_{n',p'} + \sigma_{n',p'} t\big)  = E_\beta\big(n;\mu_{n'} + \sigma_{n'} t\big),
\]
which is established in Appendix~\ref{sect:Laguerre2Gauss}, Corollary~\ref{cor:LUE2GUE}. 
\end{remark}

\part*{Part III. Hermite and Laguerre Polynomials in the Transition Region}

We establish explicit asymptotic expansions of the Hermite and Laguerre polynomials $H_n(x)$ and $L^{(\alpha)}_n(x)$ -- or,
for that matter, their associated wave functions $\phi_n(x)$ and $\phi_n^{(\alpha)}$ -- when zooming, by means of  a linear transformation 
\[
x = \mu + \sigma s,
\]
into the transition region between oscillatory behavior and exponential decay near their largest zero (which is a.s. an approximation of the largest eigenvalue of the GUE, resp. LUE). 
The centering $\mu=\mu_n$ and scaling $\sigma=\sigma_n$ are chosen so that the leading order term of the expansion is (up to a constant) given by the prime model of such a transition between oscillatory behavior and exponential decay, namely the Airy function $\Ai(s)$. Specifically, the expansions take the form, including a  tail bound,\footnote{Such an expansion for the Bessel functions $J_\nu$ of large orders $\nu$ in the transition region was used in \cite{arxiv:2301.02022} to study the hard-to-soft edge transition limit of LUE.}%
\begin{equation}\label{eq:genForm}
c h^\gamma \phi_n(\mu + \sigma s) = \Ai(s) \sum_{k=0}^{m} p_k(s) h^{k} + \Ai'(s) \sum_{k=0}^{m} q_k(s) h^{k} + h^{m+1} O(e^{-s}) 
\end{equation}
as $n\to\infty$, uniformly for $s$ bounded from below; likewise for $\phi_n^{(\alpha)}(x)$, where we can arrange for the expansion to be uniformly valid for parameters growing as $\alpha = O(n)$. Here
\[
h=h_n \asymp n^{-2/3}
\] is a carefully chosen expansion parameter which prevents the expansion from being polluted\footnote{See Remarks \ref{rem:Her_m01} and \ref{rem:Lag_m01} for examples of such pollutions if one chooses to expand in powers of $n^{-1/3}$.} by  terms of the order of odd positive powers of $n^{-1/3}$, and the $p_k, q_k$ turn out to be certain polynomials with $p_0(s)=1$ and $q_0(s)=0$, $c>0$ and $\gamma$ are certain constants. The catch is that we are not just interested in the existence and validity of such an expansion but that we need to provide a method of actually computing the coefficients $p_k, q_k$.

Since such expansions (with a uniform tail bound) are not readily available in the existing literature, we have to infer them from known asymptotic studies of the Hermite and Laguerre polynomials. We observe that by \cite[Eqs.~(18.7.19/20)]{MR2723248}, 
\begin{equation}\label{eq:Lag2Her}
H_{2n}(x) =  2^{2n} n! \, (-1)^n L_n^{(-1/2)}(x^2), \quad H_{2n+1}(x) =  x\,2^{2n+1} n! \, (-1)^n L_n^{(1/2)}(x^2),
\end{equation}
expansions of the Hermite polynomials can, basically, be obtained as special cases from those of the Laguerre polynomials (see also Remarks~\ref{rem:Lag2Her} and \ref{rem:Lag2Her2}). Several methods relating to  expansions, which are uniformly valid for sufficiently large regions of $x$ as to capture the required tail bound, have been put forward in the literature:\footnote{We only include references to work where a full expansion, and not just the leading order term, is studied.}
\begin{itemize}\itemsep=4pt
\item the steepest descent method for integrals with the Bleistein's technique of repeated integration by parts: see \cite{MR2462927} for Hermite, \cite{MR0957682} for Laguerre with {\em bounded} parameter~$\alpha$;
\item the Deift--Zhou nonlinear steepest descent method for Riemann--Hilbert problems: see \cite{MR2291866} for Hermite, \cite{MR2457107} for Laguerre with {\em bounded} parameter $\alpha$;
\item turning point analysis of linear differential equations: see \cite{MR0109898} for Weber's parabolic cylinder functions, which includes as special case the Hermite polynomials, and \cite{MR0990876, MR3874026} for confluent hypergeometric functions  in a parameter regime which includes as special case the Laguerre polynomials with parameters growing as $\alpha = O(n)$.
\end{itemize}

We will build on the last method since it addresses, already in the existing literature, the case of a large parameter $\alpha = O(n)$ which is bounded proportionally to $n$, and it  provides also formulas that are sufficiently explicit for establishing an expansion of the form \eqref{eq:genForm}, as well as computing its coefficients.

\section{Turning point analysis}\label{subsect:gen}  The uniform asymptotic description of  solutions\footnote{To avoid typographical clutter, we will mostly suppress the dependence on $u$ in the notation.} $w(t)$ of a linear second order differential equation of the form
\begin{equation}\label{eq:dw}
d^2 w/dt^2 = \big(u^2 f(t) + g(t) \big)w
\end{equation}
in terms of a large positive parameter $u\to\infty$ is an intensively studied problem. One of the basic general ideas is to transform 
the differential equation to a perturbation of a model equation that has suitable special function solutions (e.g., the exponential, trigonometric, Airy, Bessel functions). The induced perturbation of the solution itself is then obtained recursively by successive approximation or, equivalently, by matching coefficients of some ansatz while controlling the error. The definitive treatment of this method can be found in the classic monograph \cite[Chaps.~6, 10--13]{MR0435697} of Olver.\footnote{In the case of oscillatory behavior or exponential growth/decay, obtaining leading order approximations by transforming to an equation of the form
\[
d^2 V/d\xi^2 = \pm \big(u^2 + \psi(\xi)\big) V,
\]
is known as the LG (Liouville--Green, sometimes also called Liouville--Steklov) or WKB (Wentzel--Kramers--Brillouin) approximation. The important contribution, however, of the latter authors was the determination of connection formulas for linking exponential and oscillatory LG approximations across a turning point on the real axis. For a history of the asymptotics of turning point problems, see \cite{MR1554147} and \cite[Chap.~I]{MR0771669}.}

In applying the method to a concrete problem, three tasks have to be addressed:
\begin{enumerate}\itemsep=4pt
\item identify the domain of uniform validity of the expansion (in the complex plane);\footnote{Even if one is interested in real  arguments only, extending the uniform validity to the complex plane allows one to justify the repeated (term-wise) differentiation of the expansion by referring to Ritt's theorem \cite[Theorem~1.4.2]{MR0435697} (which is proved  by representing derivatives in terms of Cauchy's integral formula).}
\item compute the coefficients of a normalized expansion;
\item match, by finding the constant of proportionality, $w$ to the normalized expansion.
\end{enumerate}
The literature focuses predominantly on the first, complex analytic task. It requires a delicate discussion of the underlying variable transformation in the complex plane. Even though there are some general theorems (cf. \cite[Theorems~11.9.1/2]{MR0435697}) that can serve as a guideline, the details have to be worked out attentively in each case anew. 

The second and third computational tasks are often surprisingly hard to work out explicitly enough to be of any use, and the literature seldom provides any explicit coefficients beyond the leading order approximation. It is therefore not entirely surprising that, while all three tasks were tackled in the case of the Hermite polynomials of large orders (or more general, Weber's parabolic cylinder functions) by Olver \cite{MR0109898} in 1959, Task (1) was studied for the large order and parameter case of the Laguerre polynomials (or more general, Whittaker's confluent hypergeometric functions) by Dunster \cite{MR0990876} in 1989, whereas Task (2) and (3) were spelt out by Dunster, Gil and Segura \cite{MR3874026} just recently in 2018. 

However, since we want to have an algorithmic description of the expansion terms (Task~2) which would be suitable for a computer algebra system, and since we can provide a much more specific matching (Task~3) than put forward in \cite{MR3874026}, we explain the general method and its application to the Hermite and Laguerre polynomials. 
We focus thereby on the computational Tasks (2)/(3), whereas the complex analytic Task (1) can serendipitously be settled by referring to the existing literature.

\subsection{Classical uniform Airy expansions}
To be specific, following \cite[§§11.3/7/9]{MR0435697} we study solutions $w(t)$  of the differential equation \eqref{eq:dw}
which are recessive as $t\to+\infty$ and thus unique up to a $u$-dependent constant. Here, $u$ denotes a large positive parameter and the functions $f, g$ are assumed to be real on the real interval $(t_0,\infty)$, extending analytically to an unbounded domain $D$ in the complex plane. We assume that there is a unique zero $t_*>t_0$ of $f$ in $D$, which is simple and satisfies $f'(t_*) > 0$. Then, $t_*$ is a turning point separating oscillatory behavior in the interval $t_0<t<t_*$ from exponential decay for $t>t_*$.

An asymptotic description as $u\to\infty$ along the real axis is then facilitated by the Liouville transform $\zeta^{1/2}\,d\zeta = f(t)^{1/2}dt$ with the constant of integration fixed by $\lim_{t\to+\infty}\zeta(t)=\infty$. That is, assuming a sufficient growth of $f$,
\begin{equation}\label{eq:liouville}
\frac{2}{3}\zeta^{3/2} = \xi, \quad \xi = \int_{t_*}^t f^{1/2}(\tau)\,d \tau,\quad W = \left(\frac{f(t)}{\zeta}\right)^{1/4} w,
\end{equation}
which transforms the turning point $t_*$ to $\zeta=0$ and the differential equation \eqref{eq:dw} to the perturbed Airy equation
\begin{equation}\label{eq:dW}
d^2 W/d\zeta^2 = \big(u^2 \zeta + \psi(\zeta) \big) W,
\end{equation}
where \cite[Eq.~(11.3.06)]{MR0435697}\footnote{Local power series expansions show that the singularities at $\zeta=0$ in  quantities such as $\psi(\zeta)$ and $(\zeta/f(t))^{1/4}$ are removable, see \eqref{eq:powerZetaGen}. The same applies for the coefficients $A_k, B_k$ as defined in \eqref{eq:AB}.}

\begin{equation}\label{eq:LGpsi}
\psi(\zeta) := \frac{5}{16\zeta^2} + \zeta \phi(t),\quad \phi(t) := \frac{4 f(t) f''(t) - 5f'(t)^2}{16f(t)^3} + \frac{g(t)}{f(t)}.
\end{equation}
Since we have $\lim_{t\to+\infty} \xi(t) = \infty$, the 
change of variables $t \mapsto \zeta$ extends from $t\geq t_*$ to a biholomorphic function on some domain of the complex plane which contains $(t_0,\infty)$. We note that, for $t_0 < t \leq t_*$, the analytic continuation of $\zeta$ is given by
\[
\frac{2}{3}(-\zeta)^{3/2} = \int_{t}^{t_*} (-f)^{1/2}(\tau)\,d\tau
\]
and $\zeta$ maps the interval $t_0 < t <\infty$ strictly increasing to 
\begin{equation}\label{eq:zeta0}
 -\left(\frac{3}{2}\int_{t_0}^{t_*} (-f(t))^{1/2}\,dt\right)^{2/3} < \zeta < \infty.
\end{equation}
Under suitable assumptions on $\psi(\zeta)$ (see \cite[Theorem~11.9.2]{MR0435697}), each recessive solution $w(t)$ of the differential equation \eqref{eq:dw} has a compound asymptotic expansion of the form
\begin{equation}\label{eq:wAiry}
w(t) \sim c \left(\frac{\zeta}{f(t)}\right)^{1/4}\left( \Ai(u^{2/3} \zeta) \sum_{k=0}^\infty \frac{A_k(\zeta)}{u^{2k}} +  \frac{\Ai'(u^{2/3} \zeta)}{u^{4/3}} \sum_{k=0}^\infty \frac{B_k(\zeta)}{u^{2k}} \right)
\end{equation}
which is uniformly valid for sufficiently large domains in the complex $t$-plane, typically at least when $\Re t\geq t_0 + \delta$ as $u\to\infty$, with $\delta>0$ being any fixed positive number. 
Here, $c$ is a constant (which may, of course, depend on $u$), which reflects the fact that the space of recessive solutions is one-dimensional,\footnote{Note that $c$ is unique up to an additive term that is smaller than any positive power of $u^{-1}$ as $u\to\infty$. We use the symbol `$\simeq$' to denote equality up to such terms.\label{fn:simeq}} and $A_k(\zeta)$, $B_k(\zeta)$ are analytic coefficients, recursively defined by (see \cite[Eqs.~(11.7.08/09)]{MR0435697})%
\begin{subequations}\label{eq:AB}
\begin{align}
A_0(\zeta) = 1, \quad A_{k+1}(\zeta) &= -\frac{1}{2}B_k'(\zeta) + \frac{1}{2} \int \psi(\zeta)B_k(\zeta)\,d\zeta,\\*[2mm]
B_k(\zeta) &= \frac{1}{2\zeta^{1/2}} \int_0^\zeta \big(\psi(v) A_k(v) - A_k''(v)\big) \frac{dv}{v^{1/2}} \quad (\zeta>0).
\end{align}
\end{subequations}
Whereas, for $B_k(\zeta)$ to be analytic at $\zeta=0$, the lower integration limit has to be chosen as $0$, there is no such restriction in the choice of the integration constant for $A_{k+1}(\zeta)$.

Using the well-known leading order term of the large $\zeta$ asymptotics of the Airy function \cite[Eq.~(11.1.07)]{MR0435697}, as $\zeta\to+\infty$ for fixed $u>0$,
\begin{equation}\label{eq:airy1}
\Ai(u^{2/3} \zeta) \sim \frac{e^{-u\xi}}{2\sqrt{\pi}\, \zeta^{1/4}} u^{-1/6}, \quad 
\Ai'(u^{2/3} \zeta) \sim -\frac{e^{-u\xi}\zeta^{1/4}}{2\sqrt{\pi} } u^{1/6},
\end{equation}
we get an asymptotic expansion of the constant $c$ by letting $t\to+\infty$ in the expansion \eqref{eq:wAiry}. In fact, computing term-wise because of the uniformity in $t$ (the limits are known to exist under the assumptions mentioned for $\psi$), we obtain the asymptotic series
\begin{subequations}\label{eq:cuOlver}
\begin{align}
c &\sim 2\sqrt{\pi}\, u^{1/6} \left(1 + \sum_{k=1}^\infty \frac{\alpha_{k+1}}{ u^{2k+2}} - \sum_{k=0}^\infty \frac{\beta_k}{u^{2k+1}}\right)^{-1} \lim_{t\to+\infty} e^{u\xi} f(t)^{1/4} w(t),\\*[2mm]
\alpha_{k+1} &:=\lim_{\zeta\to +\infty} A_{k+1}(\zeta),\quad \beta_k:=\lim_{\zeta\to+\infty} \zeta^{1/2} B_{k}(\zeta) \qquad (k=0,1,2,\ldots).\label{eq:limAB}
\end{align}
\end{subequations}
Note that it is common in the literature to fix the arbitrary constants of integration in the recurrence relations \eqref{eq:AB} by specifying the conditions $\alpha_{k+1}=0$. However, we will use a different specification in \eqref{eq:limABexp} which considerably simplifies the computations and resulting expressions. 

The repeated integration in the recursion \eqref{eq:AB} makes it very hard to explicitly compute the functions $A_1,A_2, \ldots$ and $B_0,B_1,\ldots$\,. Often, one does not even find in the literature any explicit formulas for the first correction to the leading order approximation -- that is, for the coefficient $B_0$,  let alone its limit $\beta_0$.\footnote{For instance, \cite[Example~11.7.2]{MR0435697} leaves the study of Hermite polynomials of large orders computationally unresolved in such manner. The same holds with \cite[Example~11.7.3]{MR0435697} and \cite{MR0990876} for the study of the Whittaker functions (related to the Laguerre polynomials).}

 Though, for the purposes of this paper, it would suffice to use the recursion \eqref{eq:AB}
with some power series expansions near $\zeta=0$ (which are rather easy to compute with), and discuss just some rough decay estimate of the Airy expansion \eqref{eq:wAiry} to get the global picture as $t\to+\infty$, we would still need to know the first terms of the expansion for $c$, and thus the values of the first limits $\alpha_1,\alpha_2,\ldots$ and $\beta_0,\beta_1,\ldots$
 -- which are inherently inaccessible to local power series expansions.\footnote{See Section~\ref{subsect:expansions} and Theorem~\ref{thm:localExpan} for more details on this issue.}

The computation of the coefficients $A_k(\zeta)$, $B_k(\zeta)$ simplifies by relating the Airy expansion~\eqref{eq:wAiry} to the Liouville--Green expansion as $u\to\infty$, which describes the exponential decay of $w(t)$  when keeping $t>t_*$ at a safe distance to the transition point $t_*$. Formulas for the coefficients $A_k(\zeta)$, $B_k(\zeta)$, which are obtained there, will then extend to the transition region by analytic continuation. On the level of Poincaré-type expansions, such a procedure was suggested 
by Olver \cite[Footnote p.~410]{MR0435697} as being `useful in applications' and, in fact, he had used it (attributing it to a remark of Cherry) in his early studies 
of Bessel functions of large orders \cite[§6]{MR67250} and of the Weber parabolic cylinder functions  \cite[§8]{MR0109898} (of which the Hermite polynomials are a special case). 

More recently, Dunster, Gil, and Segura \cite{MR3735704} suggested using this procedure with exponential-form expansions (that is, with Poincaré-type expansions appearing inside exponentials), which often allow for even simpler computations. In \cite{MR3735704}, they illustrated their version of the procedure with yet another study of the Bessel functions in the transition region and, in the subsequent work \cite{MR3874026}, they applied it to the Laguerre polynomials of large orders and parameter.

\subsection{Examples of exponential-form expansions}

We give two examples that turn out to be useful in our study; here, $\delta>0$ is any fixed positive number.

First, though there is no simple explicit formula for the coefficients $\gamma_k$ of the Stirling-type formula (see, e.g., \cite[Eqs. (12.10.16/17)]{MR2723248}),
\begin{subequations}\label{eq:stirling}
\begin{equation}
\Gamma(z+1/2) \sim \sqrt{2\pi} \,(z/e)^z \sum_{k=0}^\infty \frac{\gamma_k}{z^k} \qquad (z \to \infty,\, |\!\arg z| \leq \pi - \delta),
\end{equation}
exponentiating the Stirling-series \cite[Eqs. (5.11.8)]{MR2723248} yields the fairly explicit exponential-form expansion (with $B_j(x)$ denoting the Bernoulli polynomials \cite[Eq.~(24.2.3)]{MR2723248})
\begin{equation}\label{eq:gammaOlver}
\sum_{k=0}^\infty \frac{\gamma_k}{z^k} \sim \exp\left( \sum_{k=1}^\infty \frac{B_{2k}(1/2)}{2k(2k-1) z^{2k-1}}\right).
\end{equation}
\end{subequations}
Expanding the right-hand side by computing with truncated series gives the numerical values of any finite number of the $\gamma_k$ (see, e.g., \cite[Eq.~(2.25)]{MR0109898}).

Second, Dunster, Gil and Segura \cite[Eqs.~(A.3/4)]{MR3735704} show that (recall $\xi = \frac{2}{3}\zeta^{3/2}$)
\begin{subequations}\label{eq:airy2}
\begin{align}
\Ai(u^{2/3}\zeta) &\sim \frac{e^{-u\xi} }{2\sqrt{\pi} \,\zeta^{1/4}} u^{-1/6} \exp\left(\sum_{k=1}^\infty (-1)^k \frac{a_k}{k u^k \xi^k}\right),\\*[2mm]
\intertext{}\notag\\*[-13mm]
\Ai'(u^{2/3} \zeta) &\sim -\frac{e^{-u\xi}\zeta^{1/4}}{2\sqrt{\pi} } u^{1/6}\exp\left(\sum_{k=1}^\infty (-1)^k \frac{a_k'}{k u^k \xi^k}\right)
\end{align}
as $\zeta\to \infty$, uniformly for $|\!\arg \zeta| \leq \pi - \delta$ and $u\geq \delta$, where the coefficients $a_k, a_k'$ are solutions of the quadratic recurrence relation
\begin{equation}
a_{k+1} = \frac{1}{2}(k+1) a_k + \frac{1}{2}\sum_{j=1}^{k-1} a_j a_{k-j} \quad (k=1,2,\ldots)
\end{equation}
with initial values $a_1=5/72$ and $a_1' = -7/72$.
\end{subequations}
By expanding the exponential form, we get the better known Poincaré-type expansions \cite[Eqs.~(9.7.5/6)]{MR2723248} of the Airy function for large arguments, which exhibit more unwieldy recurrence formulas for their coefficients, though.

\subsection{Exponential form of the uniform Airy expansions}

With the notation as in \eqref{eq:liouville}, 
Dunster, Gil and Segura \cite[Eqs.~(2.20/21)]{MR3735704} show that the asymptotic series of the classical Airy expansion \eqref{eq:wAiry} can be represented in the exponential form
\begin{subequations}\label{eq:dunster}
\begin{align}
\sum_{k=0}^\infty \frac{A_k(\zeta)}{u^{2k}} &\sim \exp\left(\sum_{k=1}^\infty\frac{E_{2k}(t) + a_{2k}'\xi^{-2k}/(2k)}{u^{2k}}\right) \\*[1mm] &\qquad\times\cosh\left(\sum_{k=1}^\infty\frac{E_{2k-1}(t) - a_{2k-1}'\xi^{-(2k-1)}/(2k-1)}{u^{{2k-1}}}\right),\notag\\*[2mm]
\sum_{k=0}^\infty \frac{B_k(\zeta)}{u^{2k}}&\sim \frac{u}{\zeta^{1/2}}\exp\left(\sum_{k=1}^\infty\frac{E_{2k}(t) + a_{2k}\xi^{-2k}/(2k)}{u^{2k}}\right) \\*[1mm] &\qquad\times \sinh\left(\sum_{k=1}^\infty\frac{E_{2k-1}(t) - a_{2k-1}\xi^{-(2k-1)}/(2k-1)}{u^{{2k-1}}}\right),\notag
\end{align}
\end{subequations}
where the coefficients $a_k$, $a_k'$ are from the exponential-form expansion \eqref{eq:airy2} of the Airy function and its derivative, whereas the functions $E_k(t)$ are constructed as follows.

Rewriting the recursion \cite[Eqs.~(2.6)/(2.7)]{MR3735704} in terms of the original variable $t$, we first introduce the auxiliary functions 
\begin{subequations}\label{eq:FE}
\begin{equation}
F_1(t) = \frac{1}{2}\phi(t),\quad F_{k+1}(t) = -\frac{F_k'(t)}{2 f(t)^{1/2}}  - \frac{1}{2} \sum_{j=1}^{k-1} F_j(t)F_{k-j}(t)\quad (k=1,2,\ldots).
\end{equation}
They are clearly well-defined for $t>t_*$, but we see inductively that these functions extend because of the Liouville relation
\[
f(t)^{1/2}\,dt = \zeta^{1/2}\,d\zeta,
\]
such that $F_{2k-1}(t)$ and $\zeta^{1/2} F_{2k}(t)$  are meromorphic functions of $\zeta$ (or, equivalently, of $t$) in all of the domain. Upon introducing the integrals (note that there is no repeated integration)
\begin{equation}\label{eq:Ek}
E_k(t) = \int F_k(t) f(t)^{1/2}\,dt = \int F_k(t(\zeta)) \zeta^{1/2}\,d\zeta,
\end{equation}
\end{subequations}
we see that the $E_{2k}(t)$ are functions of $\zeta$, which are meromorphic for all choices of the integration constants, and that there are unique integration constants  such that the $\zeta^{1/2} E_{2k-1}(t)$ are also meromorphic: in fact, by plugging the local Laurent series 
\[
F_{2k-1}(t(\zeta)) = \sum_{j=-{n_k}}^\infty a_{k,j} \zeta^j 
\]
into the integral, term-wise integration with an arbitrary integration constant $c_{2k-1}$ gives
\[
\zeta^{1/2} E_{2k-1}(t) = \zeta^{1/2} \int \zeta^{1/2} F_{2k-1}(t(\zeta))\,d\zeta = c_{2k-1} \zeta^{1/2} +  \sum_{j=-{n_k}}^\infty \frac{a_{k,j}}{j+3/2} \zeta^{j+2},
\]
which is meromorphic at $\zeta=0$ iff $c_{2k-1}=0$. By expanding \eqref{eq:dunster}, we reclaim $A_0(\zeta)=1$ and get the first expansion terms as
\begin{subequations}\label{eq:A1B0gen}
\begin{align}
B_0(\zeta) &= -\frac{5}{48\zeta^2} + \frac{\zeta^{1/2} E_1(t)}{\zeta}, \\*[1mm]
A_1(\zeta) &= -\frac{455}{4608\zeta^3} + \frac{7 \zeta^{1/2} E_1(t)}{48\zeta^2} + \frac{\big(\zeta^{1/2} E_1(t)\big)^2}{2 \zeta} + E_2(t),
\end{align}
\end{subequations}
writing them in a form such that the meromorphic nature at $\zeta=0$ as a function of $\zeta$ (and thus of $t$) becomes clearly visible. By a uniqueness argument, the pole at $\zeta=0$ is removable and the functions are actually analytic.

Under the same assumption on $\psi$ which provides the limits \eqref{eq:limAB}, also the limits%
\begin{subequations}\label{eq:lambda}
\begin{equation}\label{eq:LimE}
\lambda_k = \lim_{t\to+\infty} E_k(t)
\end{equation}
exist, and we now fix the arbitrary integration constants in the definition of the $E_{2k}(t)$ by specifying the conditions
\begin{equation}\label{eq:lambdaEven}
\lambda_{2k}=0 \quad (k=1,2,\ldots).
\end{equation}
\end{subequations}
This choice  specifies the limits $\alpha_{k+1}$ and $\beta_k$ in \eqref{eq:cuOlver} by expanding the following asymptotic series as $u\to\infty$:
\begin{equation}\label{eq:limABexp}
\begin{aligned}
1 + \sum_{k=0}^\infty \frac{\alpha_{k+1}}{u^{2k+2}} &\sim \cosh\left(\sum_{k=1}^\infty \frac{\lambda_{2k-1}}{u^{2k-1}}\right), \\*[3mm]
\sum_{k=0}^\infty \frac{\beta_{k}}{u^{2k+1}} &\sim \sinh\left(\sum_{k=1}^\infty \frac{\lambda_{2k-1}}{u^{2k-1}}\right).
\end{aligned}
\end{equation}
Finally, using the series on the right transforms the expansion \eqref{eq:cuOlver} of the constant $c$ to 
\begin{equation}\label{eq:cu}
c \sim 2\sqrt{\pi}\, u^{1/6} \exp\left(\sum_{k=1}^\infty \frac{\lambda_{2k-1}}{u^{2k-1}}\right) \lim_{t\to+\infty} e^{u\xi} f(t)^{1/4} w(t)\quad (u\to\infty).
\end{equation}
Yet another convenient formula for $c$, using connection coefficients, is given in Theorem~\ref{thm:connect}.

\section{Application to Hermite polynomials of large orders}\label{subsect:HerGen}

By expressing the wave functions $\phi_n(x)$  associated with the Hermite polynomials, as given in \eqref{eq:HermitePhiN}, in terms of the Weber parabolic cylinder function $U(a,z)$ \cite[Eq.~(18.11.3)]{MR2723248}, 
\begin{equation}\label{eq:Weber}
\phi_n(x) = \frac{1}{\pi^{1/4} \sqrt{n!}} U(-n-1/2,2^{1/2} x),
\end{equation}
we could, at least in principle, apply the formulas and results of Olver's extensive study~\cite{MR0109898}. However, based on the exponential form of the general method in Section~\ref{subsect:gen}, we are led to simpler computations, which also serve as a model for the subsequent discussion of the Laguerre polynomials. By writing
\begin{subequations}\label{eq:gfHer}
\begin{equation}\label{eq:wHer}
w(t) = \phi_n(u^{1/2} t),\quad u = 2n+1,
\end{equation}
we infer from the differential equation \cite[Table~18.8.1]{MR2723248} satisfied by the Hermite polynomials that
$w(t)$ is a recessive (as $t\to+\infty$) solution of an equation of the form \eqref{eq:dw} with 
\begin{equation}\label{eq:fgHer}
f(t) = t^2-1,\quad g(t)=0,
\end{equation}
\end{subequations}
possessing the turning points $\pm 1$. The larger one $t_1=1$ takes the role of $t_*$ in Section~\ref{subsect:gen}, whereas $t_0=-1$ limits the domain of validity of the Airy expansion. In the following theorem, we restrict ourselves,  for convenience, to domains in the complex plane that are certain half-spaces. Maximal domains of uniformity are displayed in \cite[Figure~7.1(a)]{MR0109898}.

\begin{theorem}\label{thm:genHer} {\em (I)} With $u=2n+1$, there holds the compound asymptotic series\footnote{\label{fn:olverPre}This expansion has a considerably simpler pre-factor than Olver's \cite[Eqs.~(6.2)/(8.11)]{MR0109898} formula applied to 
\[
\phi_n(u^{1/2} t) = \frac{1}{\pi^{1/4} \sqrt{n!}} U(-u/2,(2u)^{1/2} t).
\] The simplification is caused by a different fix of the integration constants that specify the coefficients $A_k$ in the recurrence relations \eqref{eq:AB}: Olver \cite[Eq.~(8.3)]{MR0109898} takes $\alpha_{k+1}=0$ while we choose \eqref{eq:limABexp}.}
\begin{equation}\label{eq:expanHer}
\phi_n(u^{1/2}t) \sim u^{-1/12} \left(\frac{4\zeta}{t^2-1}\right)^{1/4} \left(\Ai(u^{2/3}\zeta) \sum_{k=0}^\infty \frac{A_k(\zeta)}{u^{2k}} + \frac{\Ai'(u^{2/3}\zeta)}{u^{4/3}} \sum_{k=0}^\infty \frac{B_k(\zeta)}{u^{2k}}\right),
\end{equation}
which is uniformly valid as $n\to\infty$ when $t$ belongs to the complex half-plane $\Re t \geq -1 + \delta$, with $\delta$ being any fixed positive number. Preserving uniformity, the expansion can be repeatedly differentiated w.r.t. the variable $t$. 

\smallskip
{\em (II)} Writing\/\footnote{Here, $S(t)$ is positive for $t>1$, meromorphic at $\infty$ and has a branch cut along the interval $[-1,1]$.}  $S(t)=\sqrt{t^2-1}$, the transformation $\zeta=\zeta(t)$ is given by
\begin{equation}\label{eq:xiHer}
\frac43\zeta^{3/2} =  t  S(t) - \arcosh t \quad (t>1),
\end{equation}
extending analytically to $\Re t > -1$ with $\zeta(1)=0$ and mapping the real interval $-1< t <\infty$ strictly increasing to 
\begin{equation}\label{eq:zeta0Her}
- \left(\frac{3\pi}{4}\right)^{2/3} < \zeta < \infty.
\end{equation}

{\em (III)} The coefficients $A_k(\zeta)$, $B_k(\zeta)$ are analytic for $\Re t > -1$. They are defined by expanding the asymptotic series \eqref{eq:dunster}, with the functions $E_k(t)$ given by
\begin{equation}\label{eq:EkHer}
E_k(t) = \frac{P_k(t)}{S(t)^{3k}},
\end{equation}
where $P_k \in \Q[t]$ are polynomials such that $P_{2k}(t)$ is even of degree $2k$ and $P_{2k-1}(t)$ is odd of degree $6k-3$. Subject to the conditions on their degree, the polynomials 
$P_{k}$ are uniquely determined by the differential recurrence relation
\begin{subequations}
\begin{gather}
(t^2-1) P_{k}'(t) - 3k t P_{k}(t) = Q_{k}(t),\label{eq:pkHer}\\*[1mm]
Q_{k+1}(t) =\frac{3}{2} (k+1)t Q_k(t) -  \frac{1}{2} (t^2-1) Q_k'(t) - \frac{1}{2}\sum_{j=1}^{k-1}Q_j(t)Q_{k-j}(t), \label{eq:qkHer}\\*[1mm]
Q_1(t) = -(3t^2+2)/8.\label{eq:q1Her}
\end{gather}
\end{subequations}
The limits $\lambda_k := E_k(\infty)$ are $\lambda_{2k}=0$ and
\begin{equation}\label{eq:lambdaOddHer}
\lambda_{2k-1} = \frac{2^{2k-1}B_{2k}(1/2)}{4k(2k-1)}.
\end{equation}

{\em (IV)} 
The first instances evaluate to
\begin{subequations}\label{eq:ABHer}
\begin{equation}
P_1(t)  = t(6-t^2)/24, \qquad P_2(t) = (3t^2+2)/16,
\end{equation}
and \eqref{eq:A1B0gen} can thus be written in the form\footnote{Note that $S(t)\zeta^{1/2}$ is analytic at $t=1$.}
\begin{align}
A_0(\zeta)=1, \quad A_1(\zeta) &= -\frac{455}{4608\zeta^3} + \frac{7}{48S(t)\zeta^{3/2}}\frac{P_1(t)}{S(t)^2} + \frac{P_1(t)^2+2P_2(t)}{2S(t)^6},\\*[2mm]
B_0(\zeta) &= -\frac{5}{48\zeta^2} + \frac{1}{S(t)\zeta^{1/2}} \frac{P_1(t)}{S(t)^2}.
\end{align}
\end{subequations}
\end{theorem}

\begin{proof} We organize the proof along the Tasks (1)--(3) of the general method in  Section~\ref{subsect:gen}.

\subsubsection*{\quad Task {\rm (1):} domain of validity}

In terms of the Weber function \eqref{eq:Weber}, the compound asymptotic series \eqref{eq:wAiry} was shown by Olver \cite[Eq.~(8.2)]
{MR0109898} to be uniformly valid for $\Re t \geq -1 +\delta$. By Ritt's theorem, preserving uniformity, the expansion can be repeatedly differentiated.

\subsubsection*{\quad Task {\rm (2):}\! computation of the expansion terms}

Based on \eqref{eq:fgHer}, an explicit integration evaluates \eqref{eq:liouville} and \eqref{eq:zeta0} to \eqref{eq:xiHer} and  \eqref{eq:zeta0Her}. The function $\phi(t)$ in \eqref{eq:LGpsi} becomes
\[
\phi(t) = -\frac{3t^2+2}{4S(t)^6} =: \frac{2Q_1(t)}{S(t)^6},
\]
so that the functions $F_k(t)$ in \eqref{eq:FE} take, by induction, the form $F_k(t) =Q_k(t)/S(t)^{3k+3}$,
with certain rational polynomials $Q_k \in \Q[t]$ of degree $k+1$ such that $Q_{2k-1}$ is even and $Q_{2k}$ is odd. In fact, this follows inductively from the recurrence relation \eqref{eq:qkHer}, which 
is a direct implication of \eqref{eq:FE} when observing $S(t)S'(t) = t$. 

In particular, $F_{2k-1}(t)$ and $S(t) F_{2k}(t)$ extend as rational (and thus meromorphic) functions of $t$. To prepare
for the integration that yields $E_k$ in \eqref{eq:FE}, we decompose into partial fractions which, based on the parities of the polynomials $Q_k$, take the form
\begin{subequations}
\begin{align}
F_{2k}(t) S(t) &= \sum_{j=1}^{3k+1} \beta_{2k,j} \left( \frac{1}{(t+1)^j}+\frac{(-1)^{j-1}}{(t-1)^j}\right),\label{eq:FevenHer}\\*[2mm]
F_{2k-1}(t) S(t) &= \frac{1}{S(t)} \sum_{j=1}^{3k-1}  \beta_{2k-1,j} \left( \frac{1}{(t+1)^j}+\frac{(-1)^{j}}{(t-1)^j}\right),\label{eq:FoddHer} 
\end{align}
\end{subequations}
with certain rational numbers $\beta_{k,j}$. From the discussion in Section~\ref{subsect:gen}, we know that 
\[
E_k(t) = \int F_k(t) f(t)^{1/2}\,dt = \int F_k(t) S(t) \,dt
\]
actually gives, for even index $2k$, a function $E_{2k}(t)$ which is meromorphic for all choices of the integration constant, whereas for odd index $2k-1$ there is a unique integration constant such that $S(t) E_{2k-1}(t)$ becomes meromorphic. 

Thus, when integrating \eqref{eq:FevenHer}, we have necessarily $\beta_{2k,1}=0$ since otherwise $E_{2k}(t)$ would exhibit a logarithmic singularity at $t=1$. Specifying the arbitrary integration constant by the condition that $\lambda_{2k}:=E_{2k}(\infty)=0$, 
we obtain the rational function
\[
E_{2k}(t) = \sum_{j=2}^{3k+1} \frac{\beta_{2k,j}}{1-j} \left( \frac{1}{(t+1)^{j-1}}+\frac{(-1)^{j-1}}{(t-1)^{j-1}}\right) = \frac{P_{2k}(t)}{S(t)^{6k}},
\]
with a certain even polynomial $P_{2k}(t)$ of degree less than $6k$. When plugged into the differential relation $E_{2k}'(t) = F_{2k}(t) S(t)$, 
we get the differential equation \eqref{eq:pkHer},
\[
(t^2-1) P_{2k}'(t) - 6k t P_{2k}(t) = Q_{2k}(t).
\]
Since the kernel of the differential operator acting on the left is spanned by the polynomial $(t^2-1)^{3k}$ of degree $6k$,
the polynomial solution $P_{2k}$ (having degree less than $6k$) is uniquely determined by the differential equation.\footnote{It is important to note that the differential equation \eqref{eq:pkHer} constitutes an {\em overdetermined} linear system for the coefficients of a polynomial solution $P_k$. Hence, one has to justify the {\em existence} of a polynomial solution before one can work with the ansatz $E_k(t) = P_k(t)/S(t)^{3k}$, implicitly showing thereby that the right-hand side $Q_k$ satisfies the necessary one-dimensional compatibility condition of the solvability of the linear system. We have done so by using the partial fraction decomposition and noting that $\beta_{2k,1}=0$. The discussion in the first paragraph in \cite[p.~141]{MR0109898} skips over this subtle point.} Because $Q_{2k} \in \Q[t]$ has degree $2k+1$, we see that also $P_{2k}\in \Q[t]$ and its degree is $2k$.\footnote{We note that the first concrete instance $P_2$ can be obtained without any further computation by comparing the differential equation \eqref{eq:pkHer} with \eqref{eq:qkHer}, which gives by uniqueness $P_2(t) = -\frac{1}{2} Q_1(t)$.}

In the case of an odd index $2k-1$, using the integral formula\footnote{Which is, up to an affine transformation, \cite[Eq.~(1.15.3.8)]{MR1054647} applied to $_2 F_1(0,b;c;x)=1$.} (where for $j=1,2,\ldots$ the hypergeometric term is actually a polynomial of degree $j-1$ in $t$)
\begin{equation}\label{hyperHer}
\int \frac{dt}{S(t) (t\pm1)^j} = \pm \frac{S(t)\,_2F_1\big(1,-j+1;3/2;(1\mp t)/2\big)}{(t\pm1)^j},
\end{equation}
we infer from integrating \eqref{eq:FoddHer} that the meromorphic function $S(t)E_{2k-1}(t)$ is, in fact, a rational function of the form
\[
S(t) E_{2k-1}(t) = \frac{P_{2k-1}(t)}{S(t)^{6k-4}},
\]
with a certain odd polynomial $P_{2k-1}(t)$ of degree $\leq 6k-3$. When plugged into the differential relation $E_{2k-1}'(t) = F_{2k-1}(t) S(t)$, we get the differential equation  \eqref{eq:pkHer} once again, 
\[
(t^2-1) P_{2k-1}'(t) - (6k-3) t P_{2k-1}(t) = Q_{2k-1}(t).
\]
This time, the kernel of the differential operator acting on the left is spanned by the algebraic function $(t^2-1)^{3k-3/2}$,
so that the polynomial solution $P_{2k-1}$ is uniquely determined by the differential equation. Since $Q_{2k} \in \Q[t]$, we see that also $P_{2k}\in \Q[t]$. From \eqref{eq:EkHer}, writing $\lambda_{2k-1}:= E_{2k-1}(\infty)$, we get
\begin{equation}\label{eq:pk2lambdaHer}
P_{2k-1}(t) = \lambda_{2k-1} t^{6k-3} + \text{terms of lower degree}.
\end{equation}
Since we establish $\lambda_{2k-1}\neq 0$ in the next step of the proof, the degree of $P_{2k-1}$ is $6k-3$.

\subsubsection*{\quad Task {\rm (3):}\protect\footnote{\label{fn:lambda}Computationally, and for the purposes of this paper, we
could have worked with computing  the $\lambda_{2k-1}$ recursively from \eqref{eq:pk2lambdaHer} while expanding \eqref{eq:cu} to match the order of truncation of the compound asymptotic series \eqref{eq:wAiry}. Therefore, we defer the proof of the specific formulas of this part to Appendix~\ref{app:connect}; see also Footnotes~\ref{fn:Dunster}/\ref{fn:lambdaLag}.} matching to the particular recessive solution} Using
a connection formula for the Weber function $U(a,z)$, we will prove in Appendix~\ref{app:weber} that the particular so\-lu\-tion~$w(t)$, which is defined in \eqref{eq:wHer}, possesses a constant $c$ in its Airy expansion \eqref{eq:wAiry} that evaluates (subject to the indeterminacy discussed in Footnote~\ref{fn:simeq}) to
\begin{equation}\label{eq:cHer}
c \simeq 2^{1/2} u^{-1/12}.
\end{equation}
By writing $2^{1/2} = 4^{1/4}$, this gives the pre-factor in \eqref{eq:expanHer}. Likewise, from \eqref{eq:lambdaOddWeber}, we read off the values of $\lambda_{2k-1}$ displayed in \eqref{eq:lambdaOddHer} and get, in particular, $0\neq \lambda_{2k-1}\in \Q$.
\end{proof}

\section{Application to Laguerre polynomials of large orders and parameter}\label{subsect:LagGen}

In addressing the wave functions $\phi_n^{(\alpha)}(x)$ associated with the Laguerre polynomials, as given 
in~\eqref{eq:LaguerrePhiN}, we follow the parametrization in \cite[§1]{MR3874026} and write 
\begin{subequations}\label{eq:gfLag}
\begin{equation}
w(t) = t^{1/2} \phi_n^{(\alpha)}(ut),\quad \alpha = (a^2-1) u \quad (a>0), \quad u = n+1/2.\label{eq:wLag}
\end{equation}
From the confluent hypergeometric equation \cite[Table~18.8.1]{MR2723248} satisfied by the Laguerre polynomials, we infer that $w(t)$ satisfies an equation of the form \eqref{eq:dw} with 
\begin{equation}\label{eq:fgLag}
f(t) = \frac{(t-t_0)(t-t_1)}{4t^2},\quad g(t) = -\frac{1}{4t^2},
\end{equation}
possessing the turning points\footnote{For $a=1$, the turning point $t_0$ coalesces with the double pole of $f$ at $t=0$.} $t_0< t_1$ where
\begin{equation}
t_0 = (a-1)^2, \quad t_1 = (a+1)^2.
\end{equation}
\end{subequations}
The larger one, $t_1$, takes the role of $t_*$ in Section~\ref{subsect:gen}, whereas $t_0$ limits the domain of validity of the Airy expansion. In the following theorem, we restrict ourselves,  for convenience, to domains in the complex plane that are certain half-spaces. Maximal domains of uniformity (and of analytic continuation) are displayed in \cite[Figure~3(a)]{MR0990876} and \cite[Figure~6]{MR3874026}.

\begin{theorem}\label{thm:genLag} {\em (I)} With $u=n+1/2$, $\alpha=(a^2-1)u$, $t_0=(a-1)^2$, $t_1 = (a+1)^2$, there holds the compound asymptotic series\/\footnote{\label{fn:Dunster}Note that the pre-factor here is much more specific than in the expansion \cite[Eq.~(5.17)]{MR3874026} where an unresolved exponential series $\exp\big(\sum_{k=1}^\infty \lambda_{2k-1}/u^{2k-1}\big)$ enters. While missing the formula \eqref{eq:lambdaOdd} for of $\lambda_{2k-1}$, a recursive computation via \eqref{eq:pk2lambda} suffices to compute as many concrete values of the limits $\lambda_{2k-1}$ as would be needed when truncating the series in \eqref{eq:expanLag}  at some positive power of $u^{-1}$.}
\begin{equation}\label{eq:expanLag}
\begin{aligned}
\phi_n^{(\alpha)}(ut) &\sim u^{-1/3} \left(\frac{4\zeta}{(t-t_0)(t-t_1)}\right)^{1/4} \\*[2mm]
&\qquad\qquad \times\left(\Ai(u^{2/3}\zeta) \sum_{k=0}^\infty \frac{A_k(\zeta)}{u^{2k}} + \frac{\Ai'(u^{2/3}\zeta)}{u^{4/3}} \sum_{k=0}^\infty \frac{B_k(\zeta)}{u^{2k}}\right),
\end{aligned}
\end{equation}
which is uniformly valid as $n\to\infty$ when $0< a_0 \leq a \leq a_1 < \infty$ and when $t$ belongs to the complex half-plane $\Re t \geq t_0 + \delta$, with $a_0,a_1,\delta$ being any fixed positive numbers. Preserving uniformity, the expansion can be repeatedly differentiated w.r.t. the variable $t$.

\smallskip
{\em (II)}
Writing\/\footnote{Here, $S(t)$ is positive for $t>t_1$, meromorphic at $\infty$ and has a cut along the interval $[t_0,t_1]$.} $S(t) = \sqrt{(t-t_0)(t-t_1)}$, the transformation $\zeta=\zeta(t)$ is given, as $t>t_1$, by 
\begin{equation}\label{eq:xiLag}
\frac43\zeta^{3/2} =  S(t) - \frac{t_0+t_1}{2}\arcosh\frac{2t-(t_0+t_1)}{t_1-t_0} + \sqrt{t_0t_1} \arcosh\frac{2t^{-1}-(t_0^{-1}+t_1^{-1})}{t_1^{-1}-t_0^{-1}},
\end{equation}
extending analytically to $\Re t > t_0$ with $\zeta(t_1)=0$ and mapping the real interval $t_0 < t <\infty$ strictly increasing to 
\begin{equation}\label{eq:intvalLag}
-(a\wedge 1)^{4/3} \left(\frac{3\pi}{2}\right)^{2/3} < \zeta < \infty.
\end{equation}

{\em (III)} The coefficients $A_k(\zeta)$,~$B_k(\zeta)$ are analytic for $\Re t > t_0$. They are defined by expanding the asymptotic series \eqref{eq:dunster}, with the functions $E_k(t)$ given by
\begin{equation}\label{eq:EkLag}
E_k(t) = \frac{P_k(t)}{S(t)^{3k}},
\end{equation}
where $P_k$ are polynomials  such that $P_{2k} \in \Q[a^2][t]$ has degree $4k$ in $t$ and $P_{2k-1} \in \Q[a^{-2},a^2][t]$ has degree $6k-3$ in $t$. Subject to the conditions on their degree, the polynomials $P_k$ are uniquely determined by the
differential recurrence relation
\begin{subequations}
\begin{gather}
S(t)^2 P_{k}'(t) - 3k S(t)S'(t) P_{k}(t) = Q_{k}(t), \label{eq:pkLag}\\*[1mm]
Q_{k+1}(t) = \big(3(k+1)t S(t)S'(t) -S(t)^2\big)Q_k(t) - t S(t)^2 Q_k'(t) - t \sum_{j=1}^{k-1}Q_j(t)Q_{k-j}(t),\label{eq:qkLag}\\*[1mm]
Q_1(t) = -\big(t^3 - (3a^2-1)(a^2-3)t + 2(a^2+1)(a^2-1)^2\big)/4.
\end{gather}
\end{subequations}
The limits $\lambda_k :=  E_k(\infty)$ are $\lambda_{2k}=0$ and\/\footnote{This formula generalizes the explicit numerical values computed in \cite[Eq.~(2.25)]{MR3874026} for $k=1,2,3$. In fact, by observing \eqref{eq:pk2lambda}, it can be used as a sanity check for the computation of the polynomials $P_{2k-1}$.}
\begin{equation}\label{eq:lambdaOdd}
\lambda_{2k-1} =  \frac{B_{2k}(1/2)(1+1/a^{4k-2})}{4k(2k-1)}.
\end{equation}

{\em (IV)} 
The first instances evaluate to
\begin{subequations}\label{eq:ABLag}
\begin{align}
P_1(t) &= \frac{(a^2 - 1)^2 ((a^2 - 1)^2 - 4 a^2) - 3 (a^2 + 1)^3 t + 
   3 (a^4 + 6 a^2 + 1) t^2 - (a^2 + 1) t^3}{48a^2},\label{eq:p1}\\*[2mm]
P_2(t) &=  t \big(t^3 - t(3a^2-1)(a^2-3) + 2 (a^2+1)(a^2-1)^2\big)/4,\label{eq:p2}
\end{align}
and \eqref{eq:A1B0gen} can thus be written in the form\/\footnote{Note that $S(t)\zeta^{1/2}$ is analytic at $t=t_1$.}
\begin{align}
A_0(\zeta)=1, \quad A_1(\zeta) &= -\frac{455}{4608\zeta^3} + \frac{7}{48S(t)\zeta^{3/2}}\frac{P_1(t)}{S(t)^2} + \frac{P_1(t)^2+2P_2(t)}{2S(t)^6},\\*[2mm]
B_0(\zeta) &= -\frac{5}{48\zeta^2} + \frac{1}{S(t)\zeta^{1/2}} \frac{P_1(t)}{S(t)^2}.
\end{align}
\end{subequations}
\end{theorem}

\begin{proof} We organize the proof along the Tasks (1)--(3) of the general method in  Section~\ref{subsect:gen}.

\subsubsection*{Task {\rm (1):} domain of validity}
Writing $w(t)$ in terms of the Whittaker function $W_{\kappa,\mu}(z)$ as in \eqref{eq:Whittaker}, Dunster \cite[Eq.~(6.17)]{MR0990876} (see also \cite[Eq.~(5.17)]{MR3874026})  proved the uniform validity of the compound asymptotic series  \eqref{eq:wAiry}  as $n\to\infty$ when $0< a_0 \leq a \leq a_1 < \infty$ and when $t$ belongs to the complex half-plane $\Re t \geq t_0 + \delta$ (with $a_0,a_1,\delta$ being any fixed positive numbers). By Ritt's theorem, preserving uniformity, the expansion can be repeatedly differentiated.

\subsubsection*{Task {\rm (2):} computation of the expansion terms}

Based on \eqref{eq:gfLag}, an explicit integration evaluates \eqref{eq:liouville} and \eqref{eq:zeta0} to \eqref{eq:xiLag} and  \eqref{eq:intvalLag}. The function $\phi(t)$ in \eqref{eq:LGpsi} becomes  (see \cite[Eq.~(2.8)]{MR3874026}) 
\begin{equation}\label{eq:phiLag}
\phi(t) = -\frac{t\big(t^3 - (3a^2-1)(a^2-3)t + 2(a^2+1)(a^2-1)^2\big)}{S(t)^6} =: \frac{4t Q_1(t)}{S(t)^6},
\end{equation}
so that the functions $F_k(t)$ as defined in \eqref{eq:FE} take the form
\[
F_k(t) =\frac{2 t Q_k(t)}{S(t)^{3k+3}},
\]
with certain rational polynomials $Q_k \in \Q[a^2][t]$ of degree $2k+1$ in $t$. In fact, using 
\begin{align*}
S(t)S'(t) &= t - \frac{t_0+t_1}{2} = t-(a^2+1) \in \Q[a^2][t],\\*[1mm]
S(t)^2 &= t^2 - 2 (a^2 + 1) t + (a^2 - 1)^2 \in \Q[a^2][t],
\end{align*}
this follows inductively from the recurrence relation \eqref{eq:qkLag}, which is a straightforward implication of \eqref{eq:FE}.

In particular, $F_{2k-1}(t)$ and $S(t) F_{2k}(t)$ extend as rational (and thus meromorphic) functions of $t$. To prepare
for the integration that yields $E_k$ in \eqref{eq:FE}, we decompose into partial fractions which, 
based on the symmetry with respect to $a \leftrightarrow -a$ (which exchanges $t_0$ and $t_1$), take the form
\begin{subequations}
\begin{align}
\frac{F_{2k}(t) S(t)}{2t} &= \sum_{j=1}^{3k+1} \left( \frac{\beta_{2k,j}(-a)}{(t-t_0)^j}+\frac{\beta_{2k,j}(a)}{(t-t_1)^j}\right),\label{eq:Feven}\\*[2mm]
\frac{F_{2k-1}(t) S(t)}{2t} &= \frac{1}{S(t)} \sum_{j=1}^{3k-1}  \left( \frac{\beta_{2k-1,j}(-a)}{(t-t_0)^j}+\frac{\beta_{2k-1,j}(a)}{(t-t_1)^j}\right),\label{eq:Fodd} 
\end{align}
\end{subequations}
with certain rational functions $\beta_{k,j}(a)$ of the parameter $a$. From the discussion in Section~\ref{subsect:gen}, we know that
\[
E_k(t) = \int F_k(t) f(t)^{1/2}\,dt = \int \frac{F_k(t) S(t)}{2t} \,dt
\]
actually gives, for even index $2k$, a function $E_{2k}(t)$ which is meromorphic for all choices of the integration constant, whereas for odd index $2k-1$ there is a unique integration constant such that $S(t) E_{2k-1}(t)$ becomes meromorphic.

Thus, when integrating \eqref{eq:Feven}, we have necessarily $\beta_{2k,1}(a)\equiv0$ since otherwise $E_{2k}(t)$ would exhibit a logarithmic singularity at $t=t_1$. Specifying the arbitrary integration constant by the condition that $\lambda_{2k}:=E_{2k}(\infty)=0$, we obtain the rational function
\begin{equation}\label{eq:Eeven}
E_{2k}(t) = \sum_{j=2}^{3k+1} \frac{1}{1-j} \left( \frac{\beta_{2k,j}(-a)}{(t-t_0)^{j-1}}+\frac{\beta_{2k,j}(a)}{(t-t_1)^{j-1}}\right) = \frac{P_{2k}(t)}{S(t)^{6k}},
\end{equation}
with a certain polynomial $P_{2k}(t)$ of degree less than $6k$. When plugged into the differential relation $E_{2k}'(t) = F_{2k}(t) S(t)/(2t)$,
we get the differential equation \eqref{eq:pkLag}, 
\[
S(t)^2 P_{2k}'(t) - 6k S(t)S'(t) P_{2k}(t) = Q_{2k}(t).
\]
Since the kernel of the differential operator acting on $P_{2k}$ is spanned by the polynomial $S(t)^{6k}$ of degree $6k$,
the polynomial solution $P_{2k}$  (having degree less than $6k$) is uniquely determined by this differential equation.\footnote{The discussion in \cite[p.~1448]{MR3874026} skips over the details of showing the existence of the polynomials $P_k$ and how to actually compute them.} Because $Q_{2k} \in \Q[a^2,t]$ has degree $4k+1$ in $t$, we see that $P_{2k}$ has degree $4k$ in $t$. Observing that the only divisions when solving the differential equation for the coefficients of $P_{2k}$, starting with the largest power and working backward, are by numbers of the form $j-6k$ (independent of $a$), we see that also $P_{2k} \in \Q[a^2][t]$.\footnote{We note that the first concrete instance $P_2$ can be obtained without any further computation by comparing the differential equation \eqref{eq:pkLag} with \eqref{eq:qkLag}, which gives by uniqueness $P_2(t) = -t Q_1(t)$; cf. \cite[Eq.~(2.23)]{MR3874026}.}

In the case of an odd index $2k-1$, using the integral formula\footnote{Which is, up to an affine transformation, \cite[Eq.~(1.15.3.8)]{MR1054647} applied to $_2 F_1(0,b;c;x)=1$.} (where for $j=1,2,\ldots$ the hypergeometric term is actually a polynomial of degree $j-1$ in $t$)
\begin{equation}\label{hyperLag}
\int \frac{dt}{S(t) (t-t_1)^j} = \frac{2S(t)\,_2F_1\big(-j+1,1,3/2;(t-t_0)/(t_1-t_0)\big)}{(t_0-t_1) (t-t_1)^j}
\end{equation}
and its counterpart with $t_0$ and $t_1$ interchanged, we infer from integrating \eqref{eq:Fodd} that the meromorphic function $S(t)E_{2k-1}(t)$ is, in fact, a rational function of the form
\begin{equation}\label{eq:Eodd}
S(t) E_{2k-1}(t) = \frac{P_{2k-1}(t)}{S(t)^{6k-4}},
\end{equation}
with a certain polynomial $P_{2k-1}(t)$ of degree $\leq 6k-3$. When plugged into the differential relation $E_{2k-1}'(t) = F_{2k-1}(t) S(t)/(2t)$ 
we get the differential equation \eqref{eq:pkLag},
\begin{equation}\label{eq:diffpOdd}
S(t)^2 P_{2k-1}'(t) - (6k-3) S(t)S'(t) P_{2k-1}(t) = Q_{2k-1}(t).
\end{equation}
This time, the kernel of the differential operator acting on the left is spanned by the algebraic function $S(t)^{6k-3}$
so that the polynomial solution $P_{2k-1}$ is uniquely determined by the differential equation. From \eqref{eq:EkLag}, writing
$\lambda_{2k-1}:=E_{2k-1}(\infty)$, we get
\begin{equation}\label{eq:pk2lambda}
P_{2k-1}(t) = \lambda_{2k-1} t^{6k-3} + \text{terms of lower degree}.
\end{equation}
Since we establish $0 \neq \lambda_{2k-1} \in \Q[a^{-2}]$ in the next step of the proof, the degree of $P_{2k-1}$ is actually $6k-3$. Observing that the only divisions when solving the differential equation for the rest of the coefficients of $P_{2k-1}$, starting with the next-to-largest power and working backward, are by numbers of the form $j-(6k-3)$ (independent of $a$), we infer from $\lambda_{2k-1} \in \Q[a^{-2}]$ and $Q_{2k-1}\in \Q[a^2][t]$ that $P_{2k-1} \in \Q[a^{-2},a^2][t]$.

\subsubsection*{Task {\rm (3):}\protect\footnote{\label{fn:lambdaLag}The remark of Footnote~\ref{fn:lambda} applies here, mutatis mutandis, too.} matching the particular recessive solution} Using a connection formula for the Whittaker function $W_{\kappa,\mu}(z)$, we prove in Appendix~\ref{app:whittaker} that the particular solution $w(t)$, as defined in \eqref{eq:wLag},
possesses a constant $c$ in the Airy expansion \eqref{eq:wAiry} which evaluates, subject to the indeterminacy discussed in Footnote~\ref{fn:simeq}, to
\begin{equation}\label{eq:cuLag}
c \simeq u^{-1/3}.
\end{equation}
By canceling $t^{1/2}$ on both sides, this gives the pre-factor in \eqref{eq:expanLag}. Likewise, from \eqref{eq:lambdaOddWhittaker}, we read off the values of $\lambda_{2k-1}$ displayed in \eqref{eq:lambdaOdd}  and, in particular, $0\neq \lambda_{2k-1}\in \Q[a^{-2}]$.
\end{proof}

\section{Local expansions with a uniform tail bound}\label{subsect:expansions}

In continuing the discussion of Section~\ref{subsect:gen}, we consider a particular solution 
$\Phi_u(t)$ of \eqref{eq:dw}, recessive for $t\to+\infty$ and having a factor $c\simeq1$ in the expansion \eqref{eq:wAiry},
\begin{equation}\label{eq:Phi1}
\Phi_u(t) \sim \left(\frac{\zeta}{f(t)}\right)^{1/4}\left( \Ai(u^{2/3} \zeta) \sum_{k=0}^\infty \frac{A_k(\zeta)}{u^{2k}} +  \frac{\Ai'(u^{2/3} \zeta)}{u^{4/3}} \sum_{k=0}^\infty \frac{B_k(\zeta)}{u^{2k}} \right) \quad (u\to\infty),
\end{equation}
which we assume to be uniformly valid when $t$ belongs to the complex half-plane $\Re t \geq t_0 + \delta$, with $\delta$ being any fixed positive number. In particular, the quantities $f(t), \zeta, A_k(\zeta), B_k(\zeta)$ are analytic for $\Re t > t_0$.
The coefficients $A_k, B_k$ satisfy the recurrence \eqref{eq:AB}, with the arbitrary constants of integration fixed by specifying that the limits in \eqref{eq:limAB} are given by the exponential expansions \eqref{eq:limABexp}, subject to the conditions in \eqref{eq:lambda}.

With $f_1:=f'(t_*)>0$, the local power series of $f$ and $g$ at the turning point $t_*$ can be written in the form
\begin{subequations}\label{eq:powerZetaGen}
\begin{equation}\label{eq:powerf}
f_1^{-1} f(t) =  (t-t_*) + f_2 (t-t_*)^2 + f_3 (t-t_*)^3 + \cdots, \quad g(t) = g_0 + g_1(t-t_*) + \cdots,
\end{equation}
converging for $|t-t_*| < \rho:= t_*-t_0$. 

Since the power series of $f_1^{-1} f$ starts with the term $1\cdot (t-t_*)$, the induced power series of $\zeta, \psi(\zeta)$ and $(\zeta/f(t))^{1/4}$ can be rescaled with a real power $f_1$, such that all of their coefficients are {\em independent} of $f_1$ and, in fact, {\em rational polynomials} in  $f_2,f_3,\ldots$ and $g_0,g_1,\ldots$\,:
\begin{align}
f_1^{-1/3} \zeta &= (t-t_*) + \frac{f_2}{5} (t-t_*)^2 + \frac{25f_3-8f_2^2}{175} (t-t_*)^3\label{eq:powerZetaScaled}\\*[0mm]
&\qquad + \frac{148f_2^3-550f_2f_3+875f_4}{7875}(t-t_*)^4 + \cdots,\notag\\*[2mm]
f_1^{2/3} \psi(\zeta) &= \frac{- 9 f_2^2 + 15f_3 + 35 g_0}{35}\\*[0mm]
&\qquad + \frac{56 
f_2^3 - 140 f_2 f_3+ 100 f_4 - 
 60 f_2 g_0 + 75 g_1}{75}(t-t_*) + \cdots,\notag\\*[1mm]
f_1^{1/6} \left(\frac{\zeta}{f(t)}\right)^{1/4} &= 1-\frac{f_2}{5} (t-t_*) + \frac{9f_2^2-15f_3}{70}(t-t_*)^2 + \cdots.\label{eq:powerZetaFactor}
\end{align}
\end{subequations}
From the recurrence \eqref{eq:AB}, written for $A_{k+1}$ in the form 
\[
A_{k+1}(\zeta) = -\frac{1}{2}B_k'(\zeta) + \frac{1}{2} \int_0^\zeta \psi(v)B_k(v)\,dv + c_{k+1}
\]
to include the constants $c_{k+1}$ of integration (which are chosen from matching \eqref{eq:Phi1} as $t\to\infty$ via \eqref{eq:limAB}, \eqref{eq:limABexp} and are thus {\em inaccessible} to using local expansions only), we get inductively that the rescaled coefficients
\[
f_1^{k} A_k(\zeta),\quad f_1^{k+2/3} B_k(\zeta)
\]
expand, as a function of $t$, at $t=t_*$ into power series with coefficients that are independent of $f_1$, while being rational polynomials in the coefficients $f_2,f_3,\ldots$, $g_0,g_1,\ldots$ as well as in the integration constants $c_{k+1}$. 

We now zoom into the vicinity of the turning point $t_*$ by a linear `stretching transformation'%
\begin{subequations}\label{eq:ts}
\begin{align}
t &= t_* + h_u s,\\
\intertext{subject to the condition} 
u^{2/3}\zeta &= s + O(h_u),
\end{align}
which by \eqref{eq:powerZetaScaled} is uniquely solved by
\begin{equation}\label{eq:hu}
h_u = f_1^{-1/3} u^{-2/3}, \quad u^2 =f_1^{-1} h_u^{-3}.
\end{equation}
\end{subequations}

With these preparations, we can formulate and prove the following theorem, which matches a local expansion at the turning point to a uniform tail bound (to the right).

\begin{theorem}\label{thm:localExpan} In addition to the notation and assumptions above, let us assume that $f(t) \geq c$ for $t\geq t_* + \rho/2$, with some $c>0$. Then, there holds,
for any fixed non-negative integer $m$ and any  fixed real $s_0$, as $u\to\infty$ through positive real values,
\begin{equation}\label{eq:PhiAsymp}
f_1^{1/6}\Phi_u(t_* + h_u s) =   \Ai(s) \sum_{k=0}^m p_k(s) h_u^k + \Ai'(s) \sum_{k=0}^m q_k(s) h_u^k  + h_u^{m+1} O(e^{-s}),
\end{equation}
uniformly for $s_0 \leq s < \infty$. Here, $p_k(s)$, $q_k(s)$ are polynomials in $s$ which have coefficients that are rational polynomial expressions in $f_2,f_3,\ldots$, $g_0,g_1,\ldots$ and $c_1,c_2,\ldots$ -- in fact, up to $m=2$, they are actually independent of the $c_{k+1}$\/{\em :}
\begin{subequations}\label{eq:pq2}
\begin{align}
p_0(s) &= 1, & q_0(s) &= 0, \\*[1mm] p_1(s) &= -\frac{sf_2}{5}, & q_1(s) &= \frac{s^2f_2}{5},\\*[1mm]
p_2(s) &= \frac{15s^2(3f_2^2-5f_3) + 7 s^5 f_2^2}{350},& q_2(s) &= \frac{-(s^3+3)(3 f_2^2-5 f_3) +35 g_0}{35}.
\end{align}
\end{subequations}
 Preserving uniformity, the expansion \eqref{eq:PhiAsymp} can be repeatedly differentiated w.r.t. $s$.
\end{theorem}

\begin{proof} We write 
\[
\Phi_u(t_* + h_u s) = E_m(u;s) + R_m(u;s),
\]
where $E_m$ denotes the sum of the expansion terms in \eqref{eq:PhiAsymp} and $R_m$ is the remainder to be estimated.
We split the range of $s$ into the two intervals
\[
\text{(I):} \; s_0 \leq s \leq \rho h_u^{-1}/2, \quad \text{(II):} \; \rho h_u^{-1}/2 < s < \infty.
\]
In the interval (I), the power series
\eqref{eq:powerZetaScaled} and  \eqref{eq:powerZetaFactor} become the uniformly convergent series
\begin{align*}
u^{2/3} \zeta &= s  + \frac{f_2}{5}s^2 h_u   + \frac{25f_3-8f_2^2}{175} s^3 h_u^2 +\cdots, \\*[2mm]
\left(\frac{\zeta}{f(t)}\right)^{1/4} &= f_1^{-1/6}\left(1-\frac{f_2}{5} s  h_u + \frac{9f_2^2-15f_3}{70} s^2 h_u^2 + \cdots\right),
\end{align*}
in powers of $h_u$. Likewise 
\[
\frac{A_k(\zeta(t))}{u^{2k}} = f_1^k A_k(\zeta(t)) \cdot h_u^{3k}, \qquad u^{-4/3}\frac{B_k(\zeta(t))}{u^{2k}} = f_1^{k+2/3} B_k(\zeta(t)) \cdot h_u^{3k+2}
\]
turn into series in powers of $h_u$ which converge uniformly in the interval (I). 
If we plug these uniformly convergent series into the uniform asymptotic expansion \eqref{eq:Phi1}, we get by Taylor expanding $\Ai(u^{2/3}\zeta)$ and $\Ai'(u^{2/3}\zeta)$ (using the Airy differential equation) a compound asymptotic series of the form
\begin{equation}\label{eq:PhiAsymp2}
\Phi_u(t_* + h_u s) \sim f_1^{-1/6} \left( \Ai(s) \sum_{k=0}^\infty p_k(s) h_u^k + \Ai'(s) \sum_{k=0}^\infty q_k(s) h_u^k \right),
\end{equation}
which is uniformly valid as $u\to\infty$ when $s$ belongs to the interval (I).
By construction, the polynomials $p_k, q_k$ inherit the properties of their coefficients from the expansions which were discussed at the beginning of this section. We observe that $q_0 = 0$ because the second sum, with the $B_k$ coefficients in \eqref{eq:Phi1}, starts with the power $h_u^2$, while $p_0=1$ is inherited from the normalization $A_0(\zeta)=1$. Since the constants of integration $c_{k+1}$ only effect the coefficients $A_{k+1}, B_{k+1}$ for $k=1,2,\ldots$, the expansion terms up to the power $h_u^2$ are, in fact, independent of the $c_k$. 

By truncating \eqref{eq:PhiAsymp2}, we get
\[
R_m(u;s) = h_u^{m+1} O(p_{m+1}(s) \Ai(s)) + h_u^{m+1} O(q_{m+1}(s) \Ai'(s)),
\]
uniformly for $s$ in the interval (I). Now, the super-exponential decay (see \eqref{eq:airy2}) of the Airy function $\Ai(s)$ and its derivative, as $s\to+\infty$, absorbs the polynomial growth of $p_{m+1}$ and $q_{m+1}$ into a common decay estimate $O(e^{-s})$. This implies the asserted uniform bound
\[
 R_m(u;s) = h_u^{m+1} O(e^{-s}).
\]
in the interval (I). 

On the other hand, in the interval (II) of the range of $s$, we infer from \eqref{eq:airy2}, by absorbing the polynomial growth of the coefficients $p_{k}$ and $q_{k}$ for $k=0,1,\ldots,m$ into  the super-exponential decay of the Airy function $\Ai(s)$ and its derivative,
\[
E_m(u;s) = O\Big(e^{-s(1+2/\rho)}\Big) = h_u^{m+1} O(e^{-s})
\]
since $e^{-h_u^{-1}} = O(h_u^{m+1})$. We now show that also
\begin{equation}\label{eq:PhiUEst}
\Phi_u(t_*+h_u s) = h_u^{m+1} O(e^{-s})
\end{equation}
uniformly in the interval (II) of the range of $s$, so that all terms in \eqref{eq:PhiAsymp} are absorbed in the asserted error term. On taking the leading order terms in \eqref{eq:Phi1}, \eqref{eq:airy2}, using the assumed lower bound on $f(t)$ 
and \eqref{eq:hu}, we get
\[
\Phi_u(t_* + h_u s) = h_u^{1/4} O(e^{-u \xi}),
\]
uniformly as $u\to\infty$ when $s$ belongs to the interval (II). By the definition of $\xi$ in \eqref{eq:liouville} and the lower bound on $f(t)$ we have
\[
\xi \geq c_1 h_u s 
\] 
for some $c_1 >0$. Hence, for some $c_2 > 0$ 
\[
\Phi_u(t_* + h_u s) = h_u^{1/4} O\big(e^{-c_2 s h_u^{-1/2}}\big) = h_u^{1/4} e^{-c_2 \rho h_u^{-3/2}/4} O\big(e^{-c_2 s h_u^{-1/2}/2}\big),
\]
where we used that $s$ belongs to the interval (II). On restricting ourselves to $u$ large enough such that 
\[
c_2 h_u^{-1/2}/2 \geq 1
\]
and observing $e^{-c_2 \rho h_u^{-3/2}/4} = O(h_u^{m+3/4})$, we get the asserted estimate \eqref{eq:PhiUEst}.
\end{proof}

\subsection{Expansions of  Hermite polynomials} 
From \eqref{eq:fgHer}, we get 
\[
2^{-1} f(t) = (t-1) + \frac{1}{2}(t-1)^2, \quad g(t) = 0,
\]
so that the the power series \eqref{eq:powerZetaScaled} becomes
\begin{equation}\label{eq:powerHer}
2^{-1/3} \zeta = (t - 1) + \frac{1}{10} (t - 1)^2 - \frac{2}{175} (t - 1)^3 + \frac{37}{15750} (t - 1)^4 + \cdots\, \quad (|t - 1| < 2)
\end{equation}
with coefficients on the right that are rational numbers. Since Theorem~\ref{thm:genHer} shows that $\zeta$ is analytic in the half-plane $\Re t > -1$, the power series converges in the disk $|t-1|<2$.

We now get the following corollary to Theorems~\ref{thm:genHer} and \ref{thm:localExpan}.

\begin{corollary}\label{cor:Hermite}
With $n_+=n+1/2$ and the scaling parameters\footnote{The particular scaling $h=h_u/2$ was made with Remark~\ref{rem:Lag2Her2} in mind. In expanding on \eqref{eq:ts}, we note that by \eqref{eq:powerHer} this set of parameters yields the power series
\begin{equation}\label{eq:sHer}
u^{2/3} \zeta\big|_{x = \mu +\sigma s} = s +\frac{s^2}{5}h -\frac{8s^3}{175}h^2 + \frac{148s^4}{7875}h^3 + \cdots \quad(|h s|<1).
\end{equation}
}
\begin{equation}\label{eq:HerScaling}
\mu=\sqrt{2n_+}, \quad \sigma = 2^{-1/2}n_+^{-1/6}, \quad h = 4^{-1} n_+^{-2/3},
\end{equation}
there holds, for any fixed non-negative integer $m$ and any fixed real $s_0$, as $n\to\infty$,
\begin{equation}\label{eq:expanHer2}
\frac{\phi_n(\mu + \sigma s)}{2^{1/2}h^{1/8}} =   \Ai(s) \sum_{k=0}^m p_k(s) h^k + \Ai'(s) \sum_{k=0}^m q_k(s) h^k  + h^{m+1} O(e^{-s}),
\end{equation}
uniformly  for $s_0 \leq s < \infty$. The first instances of the rational polynomials $p_k, q_k \in \Q[s]$ are
$p_0(s) = 1$, $q_0(s) = 0$,
\begin{equation}\label{eq:pqHer}
\begin{aligned}
p_1(s) &= -\frac{s}{5}, & q_1(s) &= \frac{s^2}{5},\\*[1mm]
p_2(s) &= \frac{9 s^2}{70} + \frac{s^5}{50},& q_2(s) &= -\frac{9}{35} -\frac{3 s^3}{35},\\*[1mm]
p_3(s) &= -\frac{28}{225}-\frac{473 s^3}{3150}-\frac{31 s^6}{2625}, &
q_3(s) &= \frac{473 s}{1575}+\frac{169 s^4}{3150} + \frac{s^7}{750}.
\end{aligned}
\end{equation}
Preserving uniformity, the expansion \eqref{eq:expanHer2} can be repeatedly differentiated w.r.t. $s$.
\end{corollary}

\begin{proof}
By \eqref{eq:wHer}, we have $u=2n+1$ and the stretching transformation
\eqref{eq:ts} becomes, because of $h_u = 2h$,
\[
t= 1 + 2h s, \quad x=u^{1/2} t = \mu + \sigma s.
\]
From \eqref{eq:expanHer}, \eqref{eq:Phi1}, and \eqref{eq:PhiAsymp}, we get that the factor in front of the bracket in \eqref{eq:expanHer2} is
\[
u^{-1/12} 4^{1/4} 2^{-1/6} = 2^{1/2} h^{1/8}
\]
and the range of validity of the expansion \eqref{eq:expanHer2} follows from combining Theorem~\ref{thm:genHer} with Theorem~\ref{thm:localExpan}.

The expressions for $p_0,p_1,p_2$ and $q_0,q_1,q_2$ follow directly from \eqref{eq:pq2}, whereas $p_3$ and $q_3$ are obtained from plugging \eqref{eq:powerHer} and $t=1+2h s$ into \eqref{eq:ABHer} and re-expanding \eqref{eq:expanHer}.
\end{proof}

\begin{remark}\label{rem:Her_m01} In the literature, the case $m=0$ of \eqref{eq:expanHer2} can be found with a scaling parameter 
\[
\tilde \sigma = 2^{-1/2} n^{-1/6}
\]
instead of $\sigma =2^{-1/2} (n+1/2)^{-1/6}$, which gives, as $n\to\infty$,
\[
2^{-1/4} n^{1/12} \phi_n(\mu + \tilde \sigma s) = \Ai(s) + n^{-2/3} O(e^{-s}),
\]
uniformly valid for $s$ being bounded from below. Except for exhibiting a slightly weaker error term of the form $n^{-2/3} O(e^{-s/2})$, this result was proved by Johnstone and Ma \cite[Prop.~1]{MR3025686}, who also emphasized the differentiability w.r.t. $s$. If only bounded $s$ are to be considered, a predecessor is the classical Plancherel--Rotach formula \cite[Eq.~(8.22.14)]{MR0372517}
 \[
2^{-1/4} n^{1/12} \phi_n(\mu + \tilde \sigma s) = \Ai(s) + O(n^{-2/3}) \quad (n\to\infty).
\] 
We note that for $m\geq 1$ the use of $\tilde\sigma$ instead of $\sigma$ would pollute the expansion \eqref{eq:expanHer2} with terms of order $n^{-(2k+1)/3}$ for $k=1,2,\ldots$\,; likewise if we  used (a multiple of) $\tilde h=n^{-2/3}$ as the expansion parameter instead of $h=\tfrac14 (n-1/2)^{-2/3}$.
\end{remark}

\subsection{Expansions of Laguerre polynomials} From \eqref{eq:fgLag}, we get, for $|t-t_1| < t_1 = (a+1)^2$,
\begin{subequations}
\begin{align}
f(t) &= \frac{a}{(a+1)^4} \bigg( (t-t_1) + \sum_{j=2}^\infty \frac{(-1)^{j}\big(j(a^2-1)^2-(a+1)^4\big)}{4a(a+1)^{2j}}(t-t_1)^j\bigg),\\*[1mm]
g(t) &= \sum_{j=0}^\infty \frac{(-1)^{j-1}(j+1)}{4(a+1)^{2j+4}}(t-t_1)^j,
\end{align}
\end{subequations}
so that the power series \eqref{eq:powerZetaGen} becomes
\begin{multline}\label{eq:powerLag}
\frac{(a+1)^{4/3}}{a^{1/3}} \zeta = (t - t_1) + \frac{1 - 6 a + a^2}{20 a (a + 1)^2} (t - t_1)^2 \\*[2mm] - \frac{1 + 13 a - 62 a^2 + 13 a^3 + a^4}{350a^2 (a + 1)^4} (t - t_1)^3 \\*[2mm]+ \frac{37 + 434 a + 3607 a^2 - 15724 a^3 + 3607 a^4 + 434 a^5 + 37 a^6}{126000a^3 (a + 1)^6} (t - t_1)^4 + \cdots \quad (|t-t_1| < 4a)
\end{multline}
where the coefficients on the right belong to $\Q[a^{-1}(a+1)^{-2},a]$. Since  $\zeta$ is analytic in the half-plane $\Re t > t_0$, the power series converges in the disk $|t-t_1|<t_1-t_0 = 4a$.

We now get the following corollary to Theorems~\ref{thm:genLag} and \ref{thm:localExpan}. 

\begin{corollary}\label{cor:Laguerre1}
With $u=n+1/2$, $\alpha=(a^2-1)u$ and the scaling parameters
\begin{equation}\label{eq:LagScaling}
\mu_u=(a+1)^2 u, \quad \sigma_u = \frac{(a+1)^{4/3}}{a^{1/3}}u^{1/3}, \quad h_u =\frac{(a+1)^{4/3}}{a^{1/3}} u^{-2/3},
\end{equation}
there holds, as $n\to\infty$,
\begin{equation}\label{eq:expanLag3}
\phi_n^{(\alpha)}(\mu_u + \sigma_u s) = \frac{h_u^{1/2}}{a+1}  \left( \Ai(s) \sum_{k=0}^m p_k(s) h_u^k + \Ai'(s) \sum_{k=0}^m q_k(s) h_u^k  + h_u^{m+1} O(e^{-s}) \right),
\end{equation}
uniformly  for $s_0 \leq s < \infty$ and $0< a_0 \leq a \leq a_1 < \infty$, with $m$ being any fixed non-negative integer and
$s_0$, $0<a_0<a_1$ being any real numbers. The $p_k, q_k$ are polynomials in the polynomial space $\Q[a^{-1},(a+1)^{-1},a][s]$ -- the first instances are\/\footnote{We refrain from displaying the lengthy expressions of the polynomials $p_3(s), q_3(s)$, since we will give a much more compact symmetric version of Corollary~\ref{cor:Laguerre1} in Section~\ref{subsect:symmetric}, see \eqref{eq:pqLagTau}.} $p_0(s) = 1$, $q_0(s) = 0$,
\begin{align*}
p_1(s) &= -\frac{(a^2+4 a+1)s}{20 a (a+1)^2}, \qquad q_1(s) = \frac{(a^2-6 a+1)s^2}{20 a (a+1)^2},\\*[1mm]
p_2(s) &= \frac{\left(9 a^4+40 a^3+114 a^2+40 a+9\right) s^2}{1120 a^2 (a+1)^4} + \frac{\left(a^2-6 a+1\right)^2 s^5}{800 a^2 (a+1)^4},\\*[1mm]
q_2(s) &= -\frac{9 a^4+12 a^3+2 a^2+12 a+9+\left(3 a^4+18 a^3-130 a^2+18 a+3\right) s^3}{560 a^2 (a+1)^4}.
\end{align*}
Preserving uniformity, the expansion \eqref{eq:expanLag3} can be repeatedly differentiated w.r.t. $s$.
\end{corollary}

\begin{proof} 
Here, by \eqref{eq:gfLag}, it is $x^{1/2} \phi_n^{(\alpha)}(x)\big|_{x=ut}$ which solves a differential equation \eqref{eq:dw}, so that we have to apply Theorem~\ref{thm:localExpan} accordingly. By \eqref{eq:wLag}, we have $u=n+1/2$ and the stretching transformation
\eqref{eq:ts} becomes
\[
t= 1 + h_u s, \quad x=u t = \mu_u + \sigma_u s.
\]
From \eqref{eq:wLag}, \eqref{eq:expanLag}, \eqref{eq:Phi1}, and \eqref{eq:PhiAsymp}, we get  the factor in front of the bracket in \eqref{eq:expanHer2} as
\[
u^{1/2-1/3} f_1^{-1/6} = \frac{a+1}{a^{1/4}}h_u^{-1/4}
\]
from combining Theorem~\ref{thm:genLag} with \ref{thm:localExpan} we get, as $n\to\infty$,
\begin{multline}\label{eq:expanLag2}
x^{1/2} \phi_n^{(\alpha)}(x)\Big|_{x=\mu_u + \sigma_u s} \\*[2mm]
= \frac{a+1}{a^{1/4}} h_u^{-1/4} \left( \Ai(s) \sum_{k=0}^m \tilde p_k(s) h_u^k + \Ai'(s) \sum_{k=0}^m \tilde q_k(s) h_u^k  + h_u^{m+1} O(e^{-s}) \right),
\end{multline}
uniformly for $s_0 \leq s < \infty$ and $0< a_0 \leq a \leq a_1 < \infty$, with $m$ being any fixed non-negative integer and
$s_0$, $0<a_0<a_1$ being any real numbers. The $\tilde p_k, \tilde q_k$ are polynomials in $\Q[a^{-1},(a+1)^{-1},a][s]$ -- the first instances follow directly from \eqref{eq:pq2}\/{\rm :}%
\begin{align*}
\tilde p_0(s) &= 1, \quad \tilde q_0(s) = 0,\\*[1mm] 
\tilde p_1(s) &= -\frac{(a^2-6 a+1)s}{20 a (a+1)^2}, \quad
\tilde q_1(s) = \frac{(a^2-6 a+1)s^2}{20 a (a+1)^2},\\*[1mm]
\tilde p_2(s) &= \frac{3 \left(3 a^4+4 a^3-46 a^2+4 a+3\right) s^2}{1120 a^2 (a+1)^4} + \frac{\left(a^2-6 a+1\right)^2 s^5}{800 a^2 (a+1)^4},\\*[1mm]
\tilde q_2(s) &= -\frac{9 a^4+12 a^3+2 a^2+12 a+9+\left(3 a^4+4 a^3-46 a^2+4 a+3\right) s^3}{560 a^2 (a+1)^4}.
\end{align*}
Now, by multiplying \eqref{eq:expanLag2} with
\[
x^{-1/2} \Big|_{x=\mu_u + \sigma_u s} = u^{-1/2} \big((a+1)^{2} + h_u s\big)^{-1/2},
\]
which  expands into the uniformly convergent power series
\[
x^{-1/2} \Big|_{x=\mu_u + \sigma_u s} = \frac{u^{-1/2}}{a+1}\sum_{k=0}^\infty \binom{-1/2}{k} \frac{s^k h_u^k}{(a+1)^{2k}}
\qquad (s_0 \leq s \leq \rho h_u^{-1}/2)
\]
in the interval (I) of the proof of Theorem~\ref{thm:localExpan}, while staying bounded in the complementary interval (II) $s > \rho h_u^{-1}/2$, we get \eqref{eq:expanLag3} and its domain of uniform validity. 
\end{proof}

\begin{remark}\label{rem:Lag_m01} 
In his  2001 study of the Tracy--Widom limit laws of the real and complex Wishart ensembles, the leading order of the expansion \eqref{eq:expanLag2} was obtained by Johnstone  in the form \cite[Eq.~(5.1)]{MR1863961},
\[
n^{-1/6} x^{1/2} \phi_n^{(\alpha)}(x) \Big|_{x=\mu_u + \sigma_u s} 
\begin{cases}
\to \frac{(a+1)^{2/3}}{a^{1/6}} \Ai(s) &\quad \text{pointwise in $s\in\R$,} \\*[2mm]
= O(e^{-s}) & \quad \text{uniformly in $[s_0,\infty)$, $s_0\in\R$.}
\end{cases}
\]
Johnstone considered  parameters $\alpha=p-n$ of the Laguerre polynomials   with a fixed limit
\[
\frac{p}{n} \to \gamma \in (0,\infty) \quad (n\to\infty),
\]
a setting which fits into the uniformity of \eqref{eq:expanLag2} for bounded $a$, and he also emphasized the differentiability w.r.t. $s$ as being an important property when lifting the inherited kernel estimates to Fredholm determinants. In 2006, El Karoui \cite[§3]{MR2294977} added a rate estimate of the form $n^{-2/3} O(e^{-s/2})$, which corresponds to the case $m=0$ of the expansion~\eqref{eq:expanLag2}.
\end{remark}

\begin{remark}\label{rem:Lag_m01_alpha} In the classical literature on the Laguerre polynomials,  the cases $m=0,1$ 
of the expansion \eqref{eq:expanLag2} have been studied for fixed parameter $\alpha>-1$ and bounded $s$. Here,
\[
a^2 = \frac{\alpha}{n+1/2} +1 \to 1 \quad (n\to\infty),
\]
which fits into the uniformity of \eqref{eq:expanLag2} w.r.t. to a bounded quantity $a$. Re-expanding the case $m=0$ of \eqref{eq:expanLag2} in terms of powers of $n^{-1/3}$, gives, as $n\to\infty$,
\begin{equation}\label{eq:PlancherelLag}
(-1)^n e^{-x/2} L_n^{(\alpha)}(x) \Big|_{x=4n+2\alpha+2+2(2n)^{1/3} s} = 2^{-\alpha-1/3} n^{-1/3} \big(\!\Ai(s) + O(n^{-2/3})\big),
\end{equation}
uniformly for bounded $s$ and fixed $\alpha$, which is  the classical Plancherel--Rotach formula \cite[Eq.~(8.22.11)]{MR0372517}. Re-expanding the case $m=1$ of \eqref{eq:expanLag2}, using the expansion parameter 
\[
\nu = 2n+\alpha+1,
\]
gives, as $n\to\infty$,
\begin{multline}
(-1)^n e^{-x/2} L_n^{(\alpha)}(x)\Big|_{x=2\nu + 2\nu^{1/3} s}  \\*[2mm]
= 2^{-\alpha} \nu^{-1/3} \Big(\Ai(s)- \frac{(5\alpha+3)s\Ai(s) + 2s^2 \Ai'(s)}{10} \nu^{-2/3} + O(\nu^{-4/3}) \Big),
\end{multline}
uniformly for bounded $s$ and fixed $\alpha$, which is a 1949 result of Tricomi's \cite[Eq.~(40)]{MR0036355}.\footnote{Without reference to Tricomi, using the same parametrization as for the Plancherel--Rotach formula \eqref{eq:PlancherelLag}, and expanding into powers of $n^{-1/3}$, 
Choup \cite[Lemma~2.1]{MR2233711} gets for bounded $s$ and fixed $\alpha$ that
\begin{multline}
(-1)^n e^{-x/2} L_n^{(\alpha)}(x)\Big|_{x=4n+2\alpha+2+2(2n)^{1/3} s}  \\*[2mm] = 2^{-\alpha-1/3} n^{-1/3} \Big(\Ai(s)- \frac{(5\alpha+3)s\Ai(s) + 2s^2 \Ai'(s)}{10} (2n)^{-2/3}
- \frac{\alpha+1}{3}\big(\Ai(s) + s\Ai'(s)\big) (2n)^{-1} + O(n^{-4/3}) \Big),
\end{multline}
where we have corrected an erroneous polynomial factor of the $O(n^{-1})$ term.}
\end{remark}

\subsection{Symmetrized expansions of Laguerre polynomials}\label{subsect:symmetric}

The expansion in Corollary~\ref{cor:Laguerre1} unveils additional structure and simplified expressions of the polynomials $p_k$, $q_k$, if we respect a  hidden symmetry of the wave functions $\phi_n^{(\alpha)}(x)$: upon introducing
\begin{equation}\label{eq:LaguerreSymmetricParameters}
p := n + \alpha, \quad p_+ = p + 1/2, \quad n_+ = n+1/2,
\end{equation}
we get that the parameter $a$, which we have used so far to write $\alpha = (a^2-1)n_+$, satisfies
\[
a = p_+^{1/2}/n_+^{1/2}.
\]
Using Whittaker's confluent hypergeometric function $W_{\kappa,\mu}(z)$ (see Appendix~\ref{app:whittaker}), we get, by~\eqref{eq:Whittaker}, the representation
\begin{equation}\label{eq:LaguerreWhittaker}
\phi_{n,p}(x):=\phi_n^{(p-n)}(x) = \frac{x^{-1/2} W_{(p+n+1)/2,(p-n)/2}(x)}{\sqrt{p!\,n!}} \qquad (x>0).
\end{equation}
The general symmetry \cite[Eq.~(13.14.31)]{MR2723248}, which follows from Kummer's transformations,
\[
W_{\kappa,\mu}(z) = W_{\kappa,-\mu}(z),
\]
implies that the right-hand side of \eqref{eq:LaguerreWhittaker} is invariant under the permutation\footnote{Note that this symmetry reflects (and generalizes) the corresponding symmetry of a complex $n$-dimensional Wishart distribution with $p$ degrees of freedom, see the discussion of \eqref{eq:WishartSymmetry}.} $n \leftrightarrow p$ and can be taken as a general definition of $\phi_{n,p}(x)=\phi_n^{(p-n)}(x)$ for any real $p, n > -1$, so that
\begin{equation}\label{eq:Kummer}
\phi_{n,p}(x) = \phi_{p,n}(x)\quad (x>0).
\end{equation}
Because Corollary~\ref{cor:Laguerre1} was proved on the basis of the differential equation, which is satisfied by the Whittaker function, it generalizes to the more comprehensive definition of $\phi_{n,p}(x)$.

Now, if we restrict ourselves to $n_+, p_+ >0$, the invariance with respect to the permutation $n \leftrightarrow p$ corresponds to the invariance with respect to the transformation $a  \leftrightarrow  a^{-1}$. It turns out that the scaling parameters in \eqref{eq:LagScaling} are already symmetric in $n$ and $p$ (cf. \cite[p.~315]{MR1863961} and recall $u=n_+$),
\[
\mu_u =\mu:= \big(\sqrt{n_+}+\sqrt{p_+}\,\big)^2,\quad \sigma_u=\sigma:= \big(\sqrt{n_+}+\sqrt{p_+}\,\big) \left(\frac{1}{\sqrt{n_+}} + \frac{1}{\sqrt{p_+}}\right)^{1/3},
\]
and the expansion parameter in \eqref{eq:LagScaling} is proportional to a likewise symmetric quantity,
\begin{subequations}\label{eq:htau}
\begin{equation}
h_u = 4a h, \quad h := \frac{1}{4} \left(\frac{1}{\sqrt{n_+}} + \frac{1}{\sqrt{p_+}}\right)^{4/3} \asymp (n\wedge p)^{-2/3}.
\end{equation}
Finally, we introduce the symmetric quantity 
\begin{equation}\label{eq:tau}
\tau := \frac{4}{(a+1)(a^{-1}+1)} =
\frac{4}{\big(\sqrt{n_+}+\sqrt{p_+}\,\big) \left(\dfrac{1}{\sqrt{n_+}} + \dfrac{1}{\sqrt{p_+}}\right)},
\end{equation}
\end{subequations}
and note that the constant pre-factors in $h$ and $\tau$ have been chosen so that 
\begin{equation}\label{eq:tauhrationale}
0 < \tau \leq 1, \quad \mu + \sigma s = u t_1 (1+\tau h s),\quad \sigma = \tau h \mu.  
\end{equation}
This way, the expansion \eqref{eq:expanLag3} takes the form
\[
\frac{\phi_{n,p}(\mu + \sigma s)}{\tau^{1/2} h^{1/2}}
= \Ai(s) \sum_{k=0}^m p_k(s)(4a)^k h^k + \Ai'(s) \sum_{k=0}^m q_k(s)(4a)^k h^k  + h^{m+1} O(e^{-s}),
\]
where the polynomial coefficients $p_k(s)(4a)^k$, $q_k(s)(4a)^k$ must be invariant with respect to the transformation $a  \leftrightarrow  a^{-1}$. Inductively, from the proof of Corollary~\ref{cor:Laguerre1}, those coefficients can be written in the form
\[
\frac{4^k r_k(a;s)}{(a+1)^{2k}} = \tau^k a^{-k} r_k(a;s), \quad r_k \in \Q[a,s],
\]
where $r_k$ has degree $2k$ w.r.t. $a$. The invariance implies that $a^{-k} r_k(a;s)$ is a  linear combination of terms of the form $a^j + a^{-j}$ with coefficients in $\Q[s]$ ($j=0,1,\ldots,k)$. Recalling that the Chebyshev polynomials $T_j\in \Z[x]$ transmit the representations \cite[Eq.~(18.5.1)]{MR2723248}
\[
a^j + a^{-j} = 2T_j\left((a+a^{-1})/2\right),
\]
and substituting $(a+a^{-1})/2 = 2\tau^{-1}-1$, we finally obtain the form
\[
\tau^k a^{-k} r_k(a;s) = \tilde r_k(\tau;s), \quad \tilde r_k \in \Q[\tau,s],
\]
where $\tilde r_k$ has degree $k$ w.r.t. $\tau$. Summarizing, we have proved the following corollary. 

\begin{corollary}\label{cor:Laguerre2}
For real $n,p$, with $n_+=n+1/2>0$, $p_+=p+1/2>0$ and the induced symmetric parameters\footnote{In expanding on \eqref{eq:ts}, we note that by \eqref{eq:powerLag} this set of parameters yields the power series
\begin{multline}\label{eq:sLag}
u^{2/3} \zeta\big|_{x = \mu +\sigma s} = s-\frac{(2\tau-1)s^2}{5}h + \frac{(43 \tau ^2-18 \tau-8)s^3}{175}   h^2 \\*[0mm]
-\frac{ (1384 \tau ^3-551 \tau ^2-212 \tau -148)s^4}{7875}h^3 + \cdots \quad(|h s|<1).
\end{multline}
}
\begin{subequations}\label{eq:LagScaling3}
\begin{align}
\mu &= \big(\sqrt{n_+}+\sqrt{p_+}\,\big)^2,\hspace*{-1cm}& \sigma&= \big(\sqrt{n_+}+\sqrt{p_+}\,\big) \bigg(\frac{1}{\sqrt{n_+}} + \frac{1}{\sqrt{p_+}}\bigg)^{1/3},\\*[2mm]
h &= \frac{1}{4} \bigg(\frac{1}{\sqrt{n_+}} + \frac{1}{\sqrt{p_+}}\bigg)^{4/3},\hspace*{-1cm} & \tau &= \frac{4}{\big(\sqrt{n_+}+\sqrt{p_+}\,\big) \left(\dfrac{1}{\sqrt{n_+}} + \dfrac{1}{\sqrt{p_+}}\right)},
\end{align}
\end{subequations}
there holds for the generalized function $\phi_{n,p}(x)$, defined as in \eqref{eq:LaguerreSymmetricParameters} in terms of the Whittaker confluent hypergeometric function $W_{\kappa,\mu}(z)$, the expansion
\begin{equation}\label{eq:expanLag4}
\frac{\phi_{n,p}(\mu + \sigma s)}{\tau^{1/2} h^{1/2}} =  \Ai(s) \sum_{k=0}^m p_k(\tau,s) h^k + \Ai'(s) \sum_{k=0}^m q_k(\tau,s) h^k  + h^{m+1} O(e^{-s}),
\end{equation}
as the expansion parameter goes to zero,
\[
h \asymp (n\wedge p)^{-2/3}\to0^+.
\]
The expansion is uniformly valid for $s_0 \leq s < \infty$ and $0< \tau_0 \leq \tau \leq 1$, with $m$ being any fixed non-negative integer and
$s_0, \tau_0$ being any real numbers. The $p_k(\tau,s), q_k(\tau,s)$ are polynomials in $\Q[\tau,s]$ of degree $k$ w.r.t. $\tau$ -- the first instances are $p_0(\tau,s) = 1$, $q_0(\tau,s) = 0$,
\begin{equation}\label{eq:pqLagTau}
\begin{aligned}
p_1(\tau,s) &= -\frac{\tau +2}{10}s, \\*[1mm]
q_1(\tau,s) &= -\frac{2 \tau -1}{5}s^2,\\*[1mm]
p_2(\tau,s) &= \frac{13 \tau ^2+4 \tau +36}{280}s^2+ \frac{(2 \tau -1)^2 }{50}s^5,\\*[1mm]
q_2(\tau,s) &= \frac{\tau ^2+24 \tau -36}{140} + \frac{20 \tau ^2-3 \tau -6}{70}s^3,\\*[1mm]
p_3(\tau,s) &= \frac{\tau ^3-14 \tau ^2+112 \tau -112}{900} - \frac{803 \tau ^3+1838 \tau ^2-3244 \tau +3784}{25200}s^3\\*[1mm]
& \qquad  -\frac{(2 \tau -1) \left(614 \tau ^2-209 \tau -124\right)}{10500}s^6,\\*[1mm]
q_3(\tau,s) &= -\frac{37 \tau ^3-158 \tau ^2+3244 \tau -3784}{12600}s - \frac{2758 \tau ^3-437 \tau ^2-44 \tau -676}{12600}s^4\\*[1mm]
&\qquad -\frac{(2 \tau -1)^3}{750} s^7.
\end{aligned}
\end{equation}
Preserving uniformity, the expansion \eqref{eq:expanLag4} can be repeatedly differentiated w.r.t. $s$.
\end{corollary}

\begin{remark}\label{rem:Lag2Her} Using the generalized function $\phi_{n,p}(x)$, as defined in \eqref{eq:LaguerreSymmetricParameters},
there is yet another interrelation between the Hermite and the Laguerre polynomials, which does not need to distinguish between odd and even order as in \eqref{eq:Lag2Her} -- namely, by \cite[Eqs.~(13.18.16/17)]{MR2723248},
\begin{equation}\label{eq:Lag2Her2}
\phi_n(x) = x^{1/2} \phi_{(n-1)/2,n/2}(x^2),
\end{equation}
where, for negative $x$, the multivalued natures of $x^{1/2}$ and $\phi_{(n-1)/2,n/2}(x^2)$ cancel each other.\footnote{In fact, we could have used \eqref{eq:Lag2Her2} to define the Hermite wave functions $\phi_n$ for all real $n>-1$.} Using this interrelation, we can reclaim Corollary~\ref{cor:Hermite} by re-expanding the result of Corollary~\ref{cor:Laguerre2} as follows. With the notation of both corollaries but discriminating  the same-letter quantities of the Laguerre case by using a hat, a routine calculation shows
\begin{gather*}
\hat h = 4h + 32h^4 + O(h^7),\quad  \tau = 1 - 4h^3 + O(h^6),\\*[1mm]
 \tau^{1/2} \hat h^{1/2} x^{1/2}\Big|_{x=\mu+\sigma s} = 2^{1/2}h^{1/8} \left(1+sh-\frac{s^2}{2}h^2 +\frac{s^3+4}{2} h^3+O\big(h^4\big)\right)
\end{gather*}
and, if we define $\hat s$  by $(\mu+\sigma s)^2 =\hat \mu + \hat \sigma \hat s$, we get 
\[
\hat s = s + s^2h + h^2 + O(h^5).
\]
Now, upon using these expansions together with \eqref{eq:expanLag4} and \eqref{eq:pqLagTau} to expand, up to $m=3$, the right hand side of
\[
\phi_n(\mu+\sigma s) =  \tau^{1/2} \hat h^{1/2} x^{1/2}\Big|_{x=\mu+\sigma s} \cdot \frac{\phi_{(n-1)/2,n/2}(x^2) }{ \tau^{1/2} \hat h^{1/2}}\Big|_{x^2=\hat\mu+\hat\sigma \hat s},
\]
gives \eqref{eq:expanHer2} and, by uniqueness, the polynomials $p_0,\ldots,p_3$, $q_0,\ldots,q_3$ as displayed in \eqref{eq:pqHer}. 
\end{remark}

\begin{remark}\label{rem:Lag2Her2} In continuing the discussion of the previous remark (once more using hats to discriminate the Laguerre quantities from the Hermite ones), we observe that
yet another, much shorter way of reclaiming Corollary~\ref{cor:Hermite} from Corollary~\ref{cor:Laguerre2} would be possible if the latter result was uniformly valid for all $0< \tau \leq 1$ (which we conjecture to be the case). For given $n$, we could then use the limit relation
\[
\lim_{p\to \infty} \hat \sigma^{1/2} \phi_{n,p}(\hat\mu + \hat\sigma s) = \sigma^{1/2} \phi_n(\sigma+\mu s),
\]
which follows straightforwardly from Theorem~\ref{thm:phiTrans}, to pass from \eqref{eq:expanLag4} to \eqref{eq:expanHer2}
by taking the limit $p\to\infty$, while observing
\[
\lim_{p\to\infty}  \tau^{1/2} \hat h^{1/2} \hat \sigma^{1/2}  = 2^{1/2} h^{1/8}\sigma^{1/2},\quad \lim_{p\to\infty} \hat h = h,\quad \lim_{p\to\infty} \tau = 0.
\]
In particular, by the uniqueness of the expansion terms, this argument would give 
\[
p_k(s) = p_k(\tau,s)|_{\tau=0},\qquad  q_k(s) = q_k(\tau,s)|_{\tau=0},
\]
which is indeed true for the polynomials displayed in \eqref{eq:pqHer} and \eqref{eq:pqLagTau}.
\end{remark}

\part*{Appendices}

\appendix
\renewcommand{\sectionname}{}
\section{Finite rank perturbations of the Airy kernel determinant}\label{app:airykernel} 

We recall the calculations for the hard-to-soft edge transition that we discussed in \cite[§3]{arxiv:2301.02022}, turning them into a template by replacing the concrete coefficients with general parameters. Specifically, we assume a kernel expansion (after applying, e.g., the simplifying transformations in Lemmas  \ref{lem:KhGUEexpan}/\ref{lem:KhLUEexpan}) of the form
\[
K_{(h)}(x,y) = K_{\Ai}(x,y) + \sum_{j=1}^3  K_j(x,y)h^j + h^4 O\big(e^{-(x+y)}\big),
\]
which holds uniformly in $t_0 \leq x,y < \infty$ as $h\to 0^+$, $t_0$ being any fixed real number, and which can be repeatedly
differentiated w.r.t. $x,y$ preserving uniformity. The  kernels $K_1,K_2,K_3$  are assumed to be symmetric finite rank kernels of the form\footnote{The following discussion is kept general enough to accommodate arbitrary expansions with finite-rank kernels $K_1,\ldots,K_m$
that are  linear combinations of terms of the form $\Ai^{(j)}\otimes \Ai^{(k)}$ -- see \eqref{eq:E_hGen} and \eqref{eq:DhMinor}.} 
\begin{equation}\label{eq:K1K3}
\begin{aligned}
K_1 &= a_{01}(\Ai\otimes\Ai' + \Ai'\otimes\Ai), \\*[1mm]
K_2 &= a_{00} \Ai\otimes \Ai\, + \,a_{03} (\Ai\otimes \Ai''' + \Ai'''\otimes \Ai) + a_{12} (\Ai'\otimes \Ai'' + \Ai''\otimes \Ai'),\\*[1mm]
K_3 &= a_{02} (\Ai\otimes \Ai'' + \Ai''\otimes \Ai) + a_{05} (\Ai\otimes \Ai^{\rm V} + \Ai^{\rm V}\otimes \Ai) + a_{11} \Ai'\otimes \Ai' \\*[1mm]
& \qquad +a_{14} (\Ai'\otimes \Ai^{\rm IV} + \Ai^{\rm IV}\otimes \Ai') +a_{23} (\Ai''\otimes \Ai''' + \Ai'''\otimes \Ai'').
\end{aligned}
\end{equation}
Then, by \cite[Theorem~2.1]{arxiv:2301.02022}, the Fredholm determinant of $K_{(h)}$ on $(t,\infty)$ satisfies an expansion 
\begin{equation}\label{eq:detexpand1}
\det\big(I-K_{(h)}\big)\big|_{L^2(t,\infty)} = F_2(t) + F_2(t) \sum_{j=1}^3 d_j(t) h^j + h^4 O(e^{-2t}),
\end{equation}
uniformly for $t_0 \leq t < \infty$ as $h\to0^+$. Here, the $d_j(t)$ are certain smooth functions depending on the kernels $K_1,\ldots,K_j$, with right tail bounds 
\[
d_j(t) = O(e^{-2t}) \quad (t\to \infty).
\] 
Instead of calculating the coefficients $d_j(t)$ from the trace formulas given in \cite[Eq.~(2.10)]{arxiv:2301.02022} and \cite[Theorem~A.1]{2306.03798} (as we did in \cite[§3.3]{arxiv:2301.02022}), we present a computationally simpler approach here. For fixed $t$, the proof of \cite[Theorem~2.1]{arxiv:2301.02022} gives 
\begin{equation}\label{eq:detexpand2}
\begin{aligned}
\det\big(I-K_{(h)}\big)\big|_{L^2(t,\infty)} &= F_2(t) \det\left.\left({\mathbf I} - {\mathbf E_h}  \right)\right|_{L^2(t,\infty)} + O(h^4),\\*[0mm]
{\mathbf E_h} &:= ({\mathbf I}-{\mathbf K}_{\Ai})^{-1} \sum_{j=1}^3 {\mathbf K}_j h^j,
\end{aligned}
\end{equation}
as $h\to 0^+$, where we denote the induced integral operators of trace class, acting on $L^2(t,\infty)$, in boldface. By \eqref{eq:K1K3},
the operator ${\mathbf E_h}$
has at most rank $r=16$ (counting the number of different terms of the form $a_{j k}\Ai^{(j)} \otimes \Ai^{(k)}$ that enter the kernels $K_1$, $K_2$, $K_3$) and can be recast in the form
\begin{equation}\label{eq:E_hGen}
{\mathbf E_h} =  \sum_{\mu=1}^r h^{q_\mu} a_{j_\mu k_\mu} ({\mathbf I}-{\mathbf K}_{\Ai})^{-1} \Ai^{(j_\mu)} \otimes \Ai^{(k_\mu)},
\end{equation}
with certain integers $1\leq q_\mu \leq 3$ and $0\leq j_\mu,k_\mu \leq 5$.

Now, the operator determinant of a rank $r$ perturbation of the identity is known to be expressible in terms of a $r\times r$ matrix determinant (see, e.g., \cite[Theorem~I.3.2]{MR1744872}), 
\begin{align}
D_h(t) & := \det\left.\left({\mathbf I} - {\mathbf E_h}  \right)\right|_{L^2(t,\infty)} = \det_{\mu,\nu=1}^r \left(\delta_{\mu\nu} - h^{q_\mu} a_{j_\mu k_\mu}  u_{j_\mu k_\nu}(t)\right),\notag
\intertext{where we use, as in \cite[§3.3.1]{arxiv:2301.02022} and \cite[p.~63]{MR2787973}, the notation}
u_{jk}(t) &:= \left.\left\langle({\mathbf I}-{\mathbf K}_{\Ai})^{-1} \Ai^{(j)},\Ai^{(k)}\right\rangle\right|_{L^2(t,\infty)}=  \tr\left.\left({\mathbf I}-{\mathbf K}_{\Ai})^{-1} \Ai^{(j)}\otimes \Ai^{(k)} \right)\right|_{L^2(t,\infty)}.\label{eq:ujk}
\end{align}
Note that the symmetry of the kernel $K_{\Ai}$ induces the symmetry $u_{jk}(t) = u_{kj}(t)$.
Calculating the finite-dimensional determinant $D_h(t)$ by von Koch's expansion into minors (see, e.g., \cite[Eq.~(5.3)]{MR1744872}),
\begin{equation}\label{eq:DhMinor}
D_h(t) = 
1 + \sum_{m=1}^r(-1)^m \!\!\! \sum_{1\leq p_1< \cdots < p_m\leq r} h^{q_{p_1}+\,\cdots\,+\,q_{p_m}} \left(\prod_{\mu=1}^m a_{j_{p_\mu}k_{p_\mu}}\right)
\det_{\mu,\nu=1}^m u_{j_{p_\mu}k_{p_\nu}},
\end{equation}
gives, if truncated at $O(h^4)$,
\begin{multline*}
D_h(t) = 1 - 2  a_{01} u_{01}(t) h - \bigg(a_{00} u_{00}(t)+2 a_{03} u_{03}(t)+2 a_{12} u_{12}(t) + a_{01}^2  
\begin{vmatrix}
 u_{00}(t) & u_{01}(t)\\
 u_{10}(t) & u_{11}(t) 
\end{vmatrix} \bigg)h^2 \\*[2mm] -  \bigg(2 a_{02} u_{02}(t)+2 a_{05} u_{05}(t)+a_{11} u_{11}(t)+2 a_{14} u_{14}(t)+2 a_{23} u_{23}(t) \\*[2mm]
 + 2 a_{01} a_{03} 
\begin{vmatrix}
 u_{00}(t) & u_{03}(t) \\
 u_{10}(t) & u_{13}(t) 
\end{vmatrix}
 +2 a_{01} a_{12} 
\begin{vmatrix}
 u_{11}(t) & u_{12}(t) \\ 
  u_{01}(t) & u_{02}(t)
\end{vmatrix}
\bigg)h^3 +O(h)^4.
\end{multline*}
By uniqueness, we thus get from comparing the expansions \eqref{eq:detexpand1} and \eqref{eq:detexpand2} that, in fact,
\begin{equation}\label{eq:G1G3}
\begin{aligned}
d_1(t) &= - 2  a_{01} u_{01}(t), \\*[0mm]
d_2(t) &=  - a_{00} u_{00}(t)-2 a_{03} u_{03}(t)-2 a_{12} u_{12}(t) - a_{01}^2  
\begin{vmatrix}
 u_{00}(t) & u_{01}(t)\\
 u_{10}(t) & u_{11}(t) 
\end{vmatrix},\\*[2mm]
d_3(t) &= -  2 a_{02} u_{02}(t)-2 a_{05} u_{05}(t)-a_{11} u_{11}(t)-2 a_{14} u_{14}(t)-2 a_{23} u_{23}(t) \\*[1mm]
&\qquad\qquad   - 2 a_{01} a_{03} 
\begin{vmatrix}
 u_{00}(t) & u_{03}(t) \\
 u_{10}(t) & u_{13}(t) 
\end{vmatrix}
 -2 a_{01} a_{12} 
\begin{vmatrix}
 u_{11}(t) & u_{12}(t) \\ 
  u_{01}(t) & u_{02}(t)
\end{vmatrix}.
\end{aligned}
\end{equation}
Note that, without the simplifying transformations of Lemmas~\ref{lem:KhGUEexpan}/\ref{lem:KhLUEexpan}), if the method were applied directly to the kernel expansions in Lemmas~\ref{lem:GUEkernexpan}/\ref{lem:LUEkernexpan} after trading, by repeated application of \eqref{eq:airyrule}, the polynomial factor for higher derivatives of the Airy function, we would get terms that exhibit a much higher complexity -- namely,
\begin{gather*} d_1 = [\text{4:  $1\times 1$ minors}], \quad 
d_2 = [\text{12: $1\times 1$ minors}] + [\text{10: $2\times 2$ minors}],\\*[1mm]
d_3 = [\text{24:  $1\times 1$ minors}] + [\text{51:  $2\times 2$ minors}] + [\text{10:  $3\times 3$ minors}].
\end{gather*}
Here, the numbers before the colon give the tallies of different $k\times k$ minors.

Upon advancing the set of formulas of the Tracy--Widom theory \cite{MR1257246}, which represents the distribution $F_2$ in terms of the Hastings--McLeod solution of Painlevé II, Shinault and Tracy  showed through an explicit inspection of each single case in the range $0\leq j+k \leq 8$ that the functions $u_{jk}$ are, at least for that range of indices, linear combinations of the form \cite[p.~68]{MR2787973}
\[
u_{jk} = p_{1} \frac{F_2'}{F_2} + p_{2} \frac{F_2''}{F_2} + \cdots + p_{j+k+1} \frac{F_2^{(j+k+1)}}{F_2},
\]
with coefficients (which depend, of course, on $j$, $k$) that are certain  polynomials  $p_{l} \in \Q[t]$. By `reverse engineering' the remarks of Shinault and Tracy about validating their table, we have found an algorithm to compile such a table in the first place, see \cite[Appendix~B]{arxiv:2301.02022}.\footnote{Since the multilinear structure of $u_{jk}$ comes from integrating their derivatives within $\Q[t][q,q']$, where~$q$ is the Hastings--McLeod solution of Painlevé II, it amounts for `integrability' in a very literal sense.\label{fn:integrable}} This algorithm can also be applied to terms $T$ that are {\em nonlinear} rational polynomials
of the~$u_{jk}$, resulting for a given positive integer $n$ either in a linear combination of the form
\[
T=p_{1} \frac{F_2'}{F_2} + p_{2} \frac{F_2''}{F_2} + \cdots + p_{n} \frac{F_2^{(n)}}{F_2}, \quad p_1,\ldots,p_n \in \Q[t],
\]
or a message that such a {\em linear $F_2$-form of order $n$} does not exist for the term $T$. Extensive tabulations reported in 
\cite[Footnote~43]{arxiv:2301.02022} have let us to conjecture that all rational linear combinations of (finite) minors of $(u_{jk})_{j,k=0}^\infty$ -- that is, of determinants of the form
\[
\begin{vmatrix}
u_{j_1k_1} & \cdots & u_{j_1k_m} \\
\vdots & & \vdots \\
u_{j_mk_1} & \cdots & u_{j_mk_m}
\end{vmatrix},
\]
which are exactly those entering the formula \eqref{eq:DhMinor} for $D_h(t)$ -- are linear $F_2$-forms of order $$n=j_1 + \cdots + j_m + k_1 + \cdots + k_m + m.$$ This conjecture predicts, just by looking at the indices entering the expressions \eqref{eq:G1G3}, that the coefficient functions $d_1$, $d_2$, $d_3$ are linear $F_2$-forms of order $n=2,4,6$.

Indeed, by either a direct application of the algorithm presented in \cite[Appendix~B]{arxiv:2301.02022} or by combining the tabulated expressions (displayed in \cite[p.~68]{MR2787973}) for the $1\times 1$ minors
\[
u_{00}(t),\; u_{01}(t),\; u_{02}(t),\; u_{03}(t),\; u_{05}(t),\; u_{11}(t),\; u_{12}(t),\; u_{14}(t),\; u_{23}(t)
\]
 with those (displayed in \cite[Eqs.~(B.7/B.8)]{arxiv:2301.02022}) for the three $2\times 2$ minors 
\[
\begin{vmatrix}
 u_{00}(t) & u_{01}(t)\\
 u_{10}(t) & u_{11}(t) 
\end{vmatrix}, \quad
\begin{vmatrix}
 u_{00}(t) & u_{03}(t) \\
 u_{10}(t) & u_{13}(t) 
\end{vmatrix}, \quad
\begin{vmatrix}
 u_{11}(t) & u_{12}(t) \\ 
  u_{01}(t) & u_{02}(t)
\end{vmatrix},
\]
we obtain the specific representations
\begin{subequations}\label{eq:G1G3nice}
\begin{align}
F_2(t) d_1(t) &= -a_{01} F_2''(t), \\*[2mm]
F_2(t) d_2(t) &= -\frac{6a_{00} +7 a_{03} - 3a_{12} + a_{01}^2 }{6}F_2'(t) - \frac{2 a_{03}-a_{01}^2}{3}\, tF_2''(t)\\*[1mm]
&\qquad -\frac{a_{03}+3a_{12} + a_{01}^2}{12} F_2^{(4)}(t),\notag\\*[2mm]
\intertext{}\notag\\*[-13mm]
F_2(t) d_3(t) &= -\frac{30 a_{02}+101 a_{05} -15 a_{11}-65 a_{14} +20 a_{23} +9 a_{01} a_{03}+5 a_{01} a_{12}}{45}\,t F_2'(t)\\*[1mm]
&\hspace*{-0.5cm}-\frac{23 a_{05}-5 a_{14}+5 a_{23} -18 a_{01} a_{03}-10 a_{01} a_{12}}{45}\, t^2 F_2''(t)\notag\\*[1mm]
&\hspace*{-1.5cm}
-\frac{20 a_{02}+59 a_{05}+20 a_{11} +55 a_{14}-10 a_{23} -9 a_{01} a_{03}+5 a_{01} a_{12}}{60} F_2'''(t)\notag\\*[1mm]
&\hspace*{-2.5cm}
-\frac{a_{05}+2 a_{14}+a_{23}+a_{01} a_{12}}{9}\, t F_2^{(4)}(t)
-\frac{a_{05}+5 a_{14}+10 a_{23} + 9 a_{01} a_{03}-5 a_{01} a_{12}}{360} F_2^{(6)}(t).\notag
\end{align}
\end{subequations}
\begin{remark}\label{rem:u01}
As shown in \cite[Remark~3.1]{arxiv.2206.09411}, a simple twofold differentiation suffices to get
\[
F_2''(t) = 2 F_2(t) \cdot u_{01}(t).
\]
So, there is no need to harness the power of the Tracy--Widom theory when inferring the first equation
in \eqref{eq:G1G3nice} from the first one in \eqref{eq:G1G3}.
\end{remark}

\section{Large parameter L$\beta$E-to-G$\beta$E transition law}\label{sect:Laguerre2Gauss}

For real Wishart matrices,  limit laws for large degrees of freedom have been studied in multivariate statistical analysis quite early on. Hsu \cite[Lemma 2]{MR5576} observed in 1941 that such laws follow from applying the multivariate central limit theorem to sample covariance matrices. Namely, for a matrix $X_{n,p}$  drawn from the standard $n$-variate Wishart  distribution of $p$ degrees of freedom, the multivariate central limit theorem (cf., e.g., \cite[Theorem~3.4.4]{MR1990662}) gives, in distribution,
\begin{equation}\label{eq:multiCLT}
\frac{X_{n,p} - p I}{\sqrt{2p}} \;\to\, \GOE_n \quad (p\to\infty).
\end{equation}
By continuity, the limit can be lifted to eigenvalue distributions; see, e.g., \cite[pp.~125--126]{MR145620} and \cite[Corollary 13.3.2]{MR1990662}. Muirhead \cite[Corollary~9.5.7]{MR652932} sketched a direct approach for establishing that limit of the joint eigenvalue distribution, which we generalize to $\beta$-ensembles here.

We recall from Section~\ref{sect:intro} that, for the $n$-dimensional Gaussian and Laguerre $\beta$-ensembles (with parameter~$\alpha>-1$), the joint probability densities of the eigenvalues are
\begin{subequations}\label{eq:jointEVpdf}
\begin{align}
p_{\GbE}(x_1,\ldots,x_n) &= \frac{1}{Z} |\Delta(x)|^\beta \prod_{k=1}^n  e^{-\gamma x_k^2},\\*[1mm]
p_{\LbE,\alpha}(x_1,\ldots,x_n)&= \frac{1}{Z_\alpha} |\Delta(x)|^\beta \prod_{k=1}^n  x_k^\alpha e^{-\gamma x_k} \chi_{\{x_k > 0\}}.
\end{align}
\end{subequations}
We suppress the dependence on $\beta,\gamma$, $n$ from the notation, focusing on the dependence on the parameter $\alpha$ of the Laguerre ensemble. 
The quantity $\gamma$ is introduced to accommodate different scalings used in the literature; for 
$\beta=1,2,4$, which are considered in this paper, we follow Forrester and Rains \cite{MR1842786} in choosing
$\gamma=\gamma_\beta$ as in \eqref{eq:alphabetagamma}.

By Selberg's integral (see, e.g., \cite[p.~59]{MR2760897}), the normalizations (`partition functions') are
\begin{subequations}\label{eq:Z}
\begin{align}
Z &=(2\pi)^{n/2} \frac{n! }{\Gamma(\beta/2)^n}\prod_{k=0}^{n-1} \frac{\Gamma((k+1)\beta/2)}{(2\gamma)^{(1+k\beta)/2}},\\*[1mm]
Z_\alpha &= \frac{n!}{\Gamma(\beta/2)^n}\prod_{k=0}^{n-1}\frac{ \Gamma(\alpha+1+k\beta/2)\Gamma((k+1)\beta/2)}{\gamma^{\alpha+1+k\beta}}.
\end{align}
\end{subequations}

Recalling that,  for $\beta=1$, we have $\alpha=(p-n-1)/2$ and $\gamma=1/2$, the centering and scaling parameters in  Hsu's matrix central limit law \eqref{eq:multiCLT} therefore correspond to the choices 
\[
\mu_\alpha = p = \frac{\alpha}{\gamma} + n + 1 ,\quad \sigma_\alpha = \sqrt{2p} = \sqrt\frac{2\alpha}{\gamma} + o(1)
\]
of the centering and scaling parameters in the following L$\beta$E-to-G$\beta$E transition law.

\begin{theorem}\label{thm:Laguerre2Gauss} For given  $n, \beta,\gamma$, consider, as $\alpha\to\infty$, some scaling parameters of the form
\[
\mu_\alpha = \frac{\alpha}{\gamma} + c_\mu + o(1), \quad \sigma_\alpha = \sqrt\frac{2\alpha}{\gamma} + c_\sigma + o(1),
\]
where $c_\mu, c_\sigma$ are real constants (which may, of course, depend on $n,\beta,\gamma$). Then, there holds, locally uniform in $x_1,\ldots,x_n$, the Laguerre-to-Gaussian transition
\[
\sigma_\alpha^n \cdot p_{\LbE,\alpha}(\mu_\alpha + \sigma_\alpha x_1, \ldots, \mu_\alpha + \sigma_\alpha x_n) = p_{\GbE}(x_1, \ldots, x_n)(1+ o(1)) \quad (\alpha\to\infty).
\]
Furthermore, by Scheffé's theorem,\footnote{See, e.g., \cite[Theorem~16.12]{MR1324786}, \cite[Corollary~2.30]{MR1652247}. In the literature on $L^p$ spaces, Scheffé's theorem is named after F. Riesz.} the a.e. convergence of some probability densities to yet another probability density implies that the densities converge in $L^1$ and that the induced probability measures converge in total variation.
\end{theorem}
\begin{proof} For brevity, we write $\xi = \mu_\alpha + \sigma_\alpha x$ and $\xi_k = \mu_\alpha + \sigma_\alpha x_k$.
The claim follows from the following three asymptotic relations, as $\alpha\to\infty$ -- the first two being elementary and the last one  an application of Stirling's formula:
\begin{subequations}
\begin{gather}
|\Delta(\xi)| = |\Delta(x)| \,\sigma_\alpha^{n(n-1)/2} \sim |\Delta(x)| \prod_{k=0}^{n-1} \left(\frac{2\alpha}{\gamma}\right)^{k/2},\label{eq:DeltaLim}\\*[0mm]
\left(\frac{\gamma e}{\alpha}\right)^\alpha \xi^\alpha e^{-\gamma \xi} \sim e^{-\gamma x^2},\label{eq:eLim}\\*[2mm]
\sigma_\alpha \left(\frac{2\alpha}{\gamma}\right)^{k\beta/2}\left(\frac{\alpha}{e}\right)^\alpha \frac{\gamma^{1+k\beta} }{\Gamma(\alpha+1+k\beta/2)} \sim \frac{(2\gamma)^{(1+k\beta)/2}}{\sqrt{2\pi}},\label{eq:StirLim}
\end{gather}
\end{subequations}
where the second holds locally uniform in $x\in\R$. Here, we have used $\sigma_\alpha \sim (2\alpha/\gamma)^{1/2}$ and the fact that, when Taylor expanding the logarithm in 
\[
\alpha\log\left(\frac{\mu_\alpha +\sigma_\alpha x}{\alpha}\right) - \gamma (\mu_\alpha + \sigma_\alpha x) = -\alpha -\gamma x^2 + o(1) \quad (\alpha\to\infty),
\]
the terms involving the constants $c_\mu,c_\sigma$ cancel.
\end{proof}

\begin{remark} The convergence in total variation as stated in Theorem~\ref{thm:Laguerre2Gauss} can already be found in the work of Jiang and Li \cite[Theorem~1(ii)]{MR3413957}. Additionally, due to a more elaborate proof technique, those authors were able to extend the result to the simultaneous limits 
\[
n\to\infty, \quad \alpha/n^3 \to \infty,
\]
see \cite[Theorem~1(i)]{MR3413957}.
Subsequently, for real Wishart matrices with $n\to \infty$, a smooth phase transition between 
the  limiting cases  $p/n^3 \to 0$ and $p/n^3 \to \infty$ was established in \cite{MR3959632}.
\end{remark}

For $\beta=2$, using the correlation kernels of  GUE and LUE, Theorem~\ref{thm:Laguerre2Gauss} is also implied by the following limit relation between the Laguerre and Hermite `wave functions'.

\begin{theorem}\label{thm:phiTrans} For a given $n$, consider, as $\alpha\to\infty$, some scaling parameters of the form
\[
\mu_\alpha = \alpha + c_\mu + o(1), \quad \sigma_\alpha = \sqrt{2\alpha} + c_\sigma + o(1),
\]
where $c_\mu, c_\sigma$ are real constants (which may, of course, depend on $n$). Then, there holds, locally uniform in $x$,
\[
\sigma_\alpha^{1/2} \phi_n^{(\alpha)}(\mu_\alpha + \sigma_\alpha x) = \phi_n(x) (1 + o(1)) \quad (\alpha\to\infty).
\]
\end{theorem} 
\begin{proof} For brevity, we write $\xi =\mu_\alpha + \sigma_\alpha x = \sqrt{2\alpha} \, x_\alpha + \alpha$. Then,
we have $\lim_{\alpha\to \infty} x_\alpha = x$ and
\[
\begin{gathered}
\left(\frac{e}{\alpha}\right)^{\alpha/2} \xi^{\alpha/2} e^{-\xi/2} \sim e^{-x^2/2},\\*[2mm]
\frac{\sigma_\alpha^{1/2}}{\sqrt{\Gamma(n+\alpha+1)}} \left( \frac{\alpha}{2}\right)^{n/2} \left(\frac{\alpha}{e}\right)^{\alpha/2} \sim \frac{1}{\pi^{1/4} \sqrt{2^n}},\\*[2mm]
n! \left( \frac{2}{\alpha}\right)^{n/2} (-1)^n L_n^{(\alpha)}\big(\sqrt{2\alpha} \,x_\alpha + \alpha\big) \sim H_n(x).
\end{gathered}
\]
The first of these relations is the square root of \eqref{eq:eLim} with $\gamma=1$, the second one the square root of \eqref{eq:StirLim} with $\beta=2$, $k=n$ and $\gamma=1$ and the last one is the limit relation \cite[Eq.~(18.7.26)]{MR2723248} between Laguerre and Hermite polynomials.\footnote{A  peruse of the proofs in, e.g., \cite{MR0510617,MR0005178} shows that this limit relation, usually stated for fixed $x$ as
\[
\lim_{\alpha\to\infty} n! \left( \frac{2}{\alpha}\right)^{n/2} (-1)^n L_n^{(\alpha)}\big(\sqrt{2\alpha} \,x + \alpha\big) = H_n(x),
\]
holds, in fact,  locally uniform in $x$.} Multiplication of these three asymptotic relations proves the assertion.
\end{proof}

\begin{corollary}\label{cor:LUE2GUE} With the notation introduced in Theorems~\ref{thm:softGUE} and \ref{thm:softLUE}, as well as with the definition \eqref{eq:nprime} for $n'$ and $p'$, given $n$, there holds the convergence of probability distributions
\begin{equation}\label{eq:LUE2GUE}
\lim_{p\to\infty} E_\beta\big(n,p;\mu_{n',p'} + \sigma_{n',p'} t\big)  = E_\beta\big(n;\mu_{n'} + \sigma_{n'} t\big) \qquad (\beta=1,2,4),
\end{equation}
uniformly  in $t$. 
\end{corollary}
\begin{proof} We recall from  \eqref{eq:alphabetagamma} the choice $\gamma=\gamma_\beta$ and
that $p = 2(\alpha+1)/\beta + n - 1$.
Upon writing
\[
x = \mu_{n'} + \sigma_{n'} t, \quad \xi = \mu_{n',p'} + \sigma_{n',p'} t =: \tilde \mu_\alpha + \tilde \sigma_\alpha x,
\]
we get, for $\alpha\to\infty$,
\[
\tilde \mu_\alpha = \frac{\alpha}{\gamma} - \frac{2}{3} n' + \frac{2-\beta}{2\gamma} + O(\alpha^{-1/2}),\quad \tilde \sigma_\alpha = \sqrt{\frac{2\alpha}{\gamma}} + \frac{4}{3} \sqrt{2n'} + O(\alpha^{-1/2}),
\]
so that Theorem~\ref{thm:Laguerre2Gauss} implies the convergence in distribution of the scaled largest eigenvalue of the $\LbE$ to the one of the $\GbE$. 
Since the limit distribution is continuous, the convergence of distributions is uniform (cf., e.g., \cite[Lemma~2.11]{MR1652247}). 
\end{proof}

\section{`Histogram adjusted' expansions of probability distributions}\label{app:histo}

We consider, for sufficiently small $h>0$, a family of probability distributions $F_h(t)$ with a limit law 
\[
\lim_{h\to0^+} F_h(t) = F(t),
\]
assuming that the expansion
\begin{equation}\label{eq:EVanillaexpansion}
F_h(t) = F(t) + E_1(t) h + E_2(t) h^2 + E_3(t) h^3 + O(h^4) \quad (h\to 0^+)
\end{equation}
holds locally uniform in $t$ and that, preserving uniformity, the expansion can be repeatedly differentiated w.r.t. the variable $t$. 
When checking, as we did in Figures~\ref{fig:GUESimulation}, \ref{fig:LUESimulation}, \ref{fig:GOEGSESimulation} and \ref{fig:LOELSESimulation} such an expansion against a histogram of sampling data obtained from sampling $F_h(t)$, we have to exercise some care for the following reasons. 

A histogram of bin size $H=\eta h$ (for some fixed $\eta >0$)\footnote{In Figures~\ref{fig:GUESimulation}, \ref{fig:LUESimulation}, \ref{fig:GOEGSESimulation} and \ref{fig:LOELSESimulation}, for the $\GbE$ and $\LbE$, we have six different values of $\eta$ that satisfy $1.6 < \eta < 5.1$.} gives, up to  random sampling errors, the values 
\begin{equation}\label{eq:centralDiff}
\frac{1}{H} \int_{t-H/2}^{t+H/2} d F_h(s) = \frac{F_h(t+H/2) - F_h(t-H/2)}{H}
\end{equation}
on equispaced gridpoints $t$. Being central differences, these values are not just simply equal to the probability density
\[
F_h'(t) = F'(t) + E_1'(t) h + E_2'(t) h^2 + E_3'(t) h^3 + O(h^4) \quad (h\to 0^+),
\]
but differ from it by an error of order $O(H^2)=O(h^2)$. Hence, if we do not adjust for this difference, the histogram cannot be used to validate the expansion terms beyond the first order term $E_1'$.

To specify the necessary adjustments of the higher order coefficients, a Taylor expansion of the central difference formula \eqref{eq:centralDiff}, followed by inserting the expansion \eqref{eq:EVanillaexpansion}, shows that
\begin{equation}
\begin{aligned}
\frac{1}{H} \int_{t-H/2}^{t+H/2} d F_h(s) &=  F_h'(t) + \frac{H^2}{24} F_h'''(t) + O(H^4) \\*[1mm]
&= F'(t) + E_1'(t) h + \tilde E_2'(t)h^2 + \tilde E_3'(t) h^3 + O(h^4),
\end{aligned}
\end{equation}
locally uniform in $t$, where
\begin{equation}
\tilde E_2(t) := E_2(t) + \frac{\eta^2}{24} F''(t), \quad \tilde E_3(t) := E_3(t) + \frac{\eta^2}{24} E_1''(t),
\end{equation}
are the  `histogram adjusted' variants of the expansion terms $E_2, E_3$. As noted before, the term $E_1$ does not need to be adjusted.

\section{Using connection formulas in turning point analysis}\label{app:connect}
We establish yet another, very convenient formula for the constant $c$ in the Airy expansion~\eqref{eq:wAiry} which uses suitable companion solutions of the differential equation \eqref{eq:dw} and their connection formulas.\footnote{Such a procedure was used, en passant, in \cite{MR4312523,MR0109898}.}

In continuing the discussion of Section~\ref{subsect:gen}, we consider a particular solution 
$\Phi(t)$ of \eqref{eq:dw}, recessive for $t\to+\infty$ and having a factor $c\simeq1$ in the expansion \eqref{eq:wAiry},
\begin{subequations}\label{eq:Phis}
\begin{equation}\label{eq:Phi}
\Phi(t) \sim \left(\frac{\zeta}{f(t)}\right)^{1/4}\left( \Ai(u^{2/3} \zeta) \sum_{k=0}^\infty \frac{A_k(\zeta)}{u^{2k}} +  \frac{\Ai'(u^{2/3} \zeta)}{u^{4/3}} \sum_{k=0}^\infty \frac{B_k(\zeta)}{u^{2k}} \right) \quad (u\to\infty),
\end{equation}
which is valid at least locally uniform for $t \in D_\rho$ with $D_\rho \subset D$ being any disk centered at the turning point $t_*$. If we replace $u$ by $u e^{\mp \pi i}$ in the compound asymptotic series on the right, we get two formal companion solutions of the differential equation \eqref{eq:dw}, which can be shown (see \cite[Theorem~2.1]{MR3735704}) to represent actual solutions $\Phi_\pm(t)$ in the disk $D_\rho \subset D$; their analytic continuations are recessive for $\zeta \to \infty e^{\pm 2\pi i/3}$ in the $\zeta$-plane. Upon introducing the corresponding companion solutions of the Airy differential equation,
\[
\Ai_\pm(\zeta) := \Ai(e^{\mp 2\pi i/3} \zeta) \quad (\zeta\in\C)
\]
which are recessive for $\zeta \to \infty e^{\pm 2\pi i/3}$, we thus get
\begin{equation}\label{eq:PhiPM}
\Phi_\pm(t) \sim \left(\frac{\zeta}{f(t)}\right)^{1/4}\left( \Ai_\pm(u^{2/3} \zeta) \sum_{k=0}^\infty \frac{A_k(\zeta)}{u^{2k}} +  \frac{\Ai_\pm'(u^{2/3} \zeta)}{u^{4/3}} \sum_{k=0}^\infty \frac{B_k(\zeta)}{u^{2k}} \right) \quad (u\to\infty),
\end{equation}
\end{subequations}
which is valid at least locally uniform for $t\in D_\rho$. 

The pair $\Ai_\pm$ constitutes a fundamental system of solutions of the Airy equation (the Wronskian evaluates to $1/(2\pi i)$, see \cite[Eq.~(9.2.9)]{MR2723248}) and there holds the connection formula \cite[Eq.~(9.2.12)]{MR2723248}
\begin{equation}\label{eq:AiryConnect}
\Ai(t) = i e^{-\pi i/6} \Ai_+(t) - i e^{\pi i/6} \Ai_-(t).
\end{equation}
The factor $i$ is introduced here and in the following for convenience.
Using the expansions in $D_\rho$, these properties are inherited by the solutions $\Phi$ and $\Phi_\pm$ of the differential equation~\eqref{eq:dw}, and we can arrange for
\begin{equation}\label{eq:PhiConnect}
\Phi(t) =  i e^{-\pi i/6}\Phi_+(t) - i e^{\pi i/6}  \Phi_-(t)
\end{equation}
by specifying a representative of the constant $c\simeq 1$ that is implicit in \eqref{eq:Phi}. 

After these preparations, we can state and prove the following theorem.

\begin{theorem}\label{thm:connect} With the assumptions and notation of Section~\ref{subsect:gen}, let $w(t), w_\pm(t)$ be solutions of 
the differential equation \eqref{eq:dw} such that the asymptotic expansions
\begin{subequations}\label{eq:wExpand}
\begin{align}
w(t) &\sim c \left(\frac{\zeta}{f(t)}\right)^{1/4}\left( \Ai(u^{2/3} \zeta) \sum_{k=0}^\infty \frac{A_k(\zeta)}{u^{2k}} +  \frac{\Ai'(u^{2/3} \zeta)}{u^{4/3}} \sum_{k=0}^\infty \frac{B_k(\zeta)}{u^{2k}} \right), \\*[2mm]
w_\pm(t) &\sim c_\pm \left(\frac{\zeta}{f(t)}\right)^{1/4}\left( \Ai_\pm(u^{2/3} \zeta) \sum_{k=0}^\infty \frac{A_k(\zeta)}{u^{2k}} +  \frac{\Ai_\pm'(u^{2/3} \zeta)}{u^{4/3}} \sum_{k=0}^\infty \frac{B_k(\zeta)}{u^{2k}} \right),
\end{align}
\end{subequations}
are uniformly valid as $u\to\infty$ along the real axis when the real argument $t$ satisfies $t \geq t_0 + \delta$, with $\delta>0$ being any fixed positive number. Here, $c, c_\pm \neq 0$ are assumed to be nonzero constants (which may depend on $u$), so that the pair $w_\pm$ constitutes a fundamental system, and there is a connection formula
\begin{equation}\label{eq:wConnect}
w(t) = i \lambda_+ w_+(t) - i \lambda_- w_-(t),
\end{equation}
with certain connection coefficients $\lambda_\pm$ (which may depend on $u$). Then, in terms of the
connection coefficients $\lambda_\pm$, we have the relations\/\footnote{Rather conveniently, only the original variable $t$ of the problem enters the limit in \eqref{eq:cu_no_xi}.}
\begin{align}
c_{\pm}^{-1} &\simeq e^{\pm  \pi i/6} \lambda_\pm  c^{-1},\label{eq:cPM}\\*[1mm]
c & \simeq 2\sqrt{\pi}\, u^{1/6}  \lim_{t\to +\infty} f(t)^{1/4} \big(\lambda_\pm w_\pm(t) w(t) \big)^{1/2},\label{eq:cu_no_xi}
\end{align}
where $\simeq$ denotes equality up to an error that is smaller than any positive power of $u^{-1}$ for large $u$. The limits $\lambda_{2k-1}$, as defined in \eqref{eq:lambda}, can be obtained from expanding 
\begin{equation}\label{eq:lambdaGen}
\lim_{t\to+\infty} e^{-u\xi} \sqrt{\frac{\lambda_\pm w_\pm(t)}{w(t)}} \sim \exp\left(\sum_{k=1}^\infty \frac{\lambda_{2k-1}}{u^{2k-1}}\right) \quad (u\to\infty).
\end{equation}
\end{theorem}

\begin{proof}
Comparing the expansions \eqref{eq:Phis} with \eqref{eq:wExpand} gives 
\[
w(t)\simeq c\,\Phi(t), \quad w_\pm(t)\simeq c_\pm \Phi_\pm(t) \quad (t>t_0),
\]
so that by  the connection formulas \eqref{eq:PhiConnect} and \eqref{eq:wConnect}
\[
i e^{-\pi i/6}\Phi_+(t) - i e^{\pi i/6}  \Phi_-(t) \simeq i c_+  \lambda_+ c^{-1} \Phi_+(t) - i c_-  \lambda_- c^{-1} \Phi_-(t) \quad (t>t_0).
\]
Since the pair of solutions $\Phi_\pm$ forms a fundamental system of \eqref{eq:dw}, we read off \eqref{eq:cPM}.

Similar to the derivation of \eqref{eq:cu},
\[
c \sim 2\sqrt{\pi} u^{1/6}  \exp\left(\sum_{k=1}^\infty \frac{\lambda_{2k-1}}{u^{2k-1}}\right) \lim_{t\to +\infty} e^{u\xi} f(t)^{1/4} w(t) \quad (u\to\infty),
\]
we get, recalling that the expansions of $\Phi_{\pm}$ are obtained by replacing $u$ with $u e^{\mp \pi i}$,
\[
c_{\pm} \sim 2\sqrt{\pi} u^{1/6} e^{\mp \pi i/6}  \exp\left(-\sum_{k=1}^\infty \frac{\lambda_{2k-1}}{u^{2k-1}}\right) \lim_{t\to +\infty} e^{-u\xi} f(t)^{1/4} w_\pm(t) \quad (u\to\infty),
\]
cf. also \cite[Eqs.~(5.15/19)]{MR3874026}. Hence, by using \eqref{eq:cPM},
\[
c \sim 2\sqrt{\pi} u^{1/6}  \exp\left(-\sum_{k=1}^\infty \frac{\lambda_{2k-1}}{u^{2k-1}}\right) \lim_{t\to +\infty} e^{-u\xi} f(t)^{1/4} \lambda_\pm w_\pm(t) \quad (u\to\infty).
\]
In taking the geometric mean of this expansion with the previous one \eqref{eq:cu}, the exponential terms cancel and we obtain the asserted formula \eqref{eq:cu_no_xi}. Finally, we can solve for the exponential of the $\lambda$-series in \eqref{eq:cu} and get \eqref{eq:lambdaGen}.
\end{proof}

\subsection{Application to Weber's parabolic cylinder function \for{toc}{$U(a,z)$}\except{toc}{\boldmath$U(a,z)$\unboldmath}}\label{app:weber}

Because of \eqref{eq:Weber}, we consider Weber's parabolic cylinder function $U(a,z)$ (see \cite[§12.2]{MR2723248} for its definition), which is an entire function of the parameter $a$ and of the argument $z$. By its differential equation \cite[Eq.~(12.2.2)]{MR2723248}, we get, once more, that
\begin{subequations}\label{eq:weberGen}
\begin{equation}
w(t) := U(-u/2, (2u)^{1/2} t)\quad (u>0)
\end{equation}
is a particular recessive (as $t\to+\infty$) solution of the differential equation \eqref{eq:dw}, with $f,g$ as displayed in \eqref{eq:fgHer}. Hence, the two companion solutions, with $u$ being replaced by $ue^{\mp\pi i}$, are 
\[
w_\pm(t) := U(u/2, \mp i \sqrt{2u} t) \quad (u>0).
\]
In studying complex values of $u$, and not just $u>0$, Olver \cite[Eq.~(8.11)]{MR0109898} proved, in particular, that the assumptions of Theorem~\ref{thm:connect} are met. The coefficients in the connection formula \eqref{eq:wConnect} are known to be (see \cite[Eq.~(3.26)]{MR0240343} or \cite[Eq.~(12.2.18)]{MR2723248})
\[
\lambda_\pm = \frac{\Gamma((u+1)/2)}{\sqrt{2\pi}} e^{\mp \pi i (u+1)/4}
\]
and the leading order asymptotic for $t\to+\infty$ is given by (see \cite[Eq.~(7.5a/b)]{MR0240343}, \cite[Eq.~(12.9.1)]{MR2723248})
\begin{align*}
w(t) &\sim e^{-u t^2/2} (2ut^2)^{(u-1)/4},\\*[1mm]
w_\pm(t) &\sim e^{u t^2/2} (2ut^2)^{-(u+1)/4} e^{\mp \pi i(u+1)/4}.
\end{align*}
Hence, using $f(t)\sim t^2$ as $t\to +\infty$, the formula \eqref{eq:cu_no_xi} for the constant $c$ in  \eqref{eq:wExpand} evaluates straightforwardly
to\footnote{For the reasons discussed in Footnote~\ref{fn:olverPre}, this is much simpler than the pre-factor in \cite[Eq.~(8.11)]{MR0109898}.} 
\begin{equation}
c \simeq 2^{1/2} \pi^{1/4} u^{-1/12} \sqrt{\Gamma((u+1)/2)}.
\end{equation}
\end{subequations}
For the particular case of the wave functions $\phi_n(ut)$ associated with the Hermite polynomials, where  $u=2n+1$ and thus
\[
\Gamma((u+1)/2) = n!,
\]
a rescaling by $1/\pi^{1/4} \sqrt{n!}$ as in \eqref{eq:Weber} gives the value of the associated constant shown in~\eqref{eq:cHer}.

In the same fashion, we evaluate formula \eqref{eq:lambdaGen} to compute the values of the limits $\lambda_{2k-1}$. 
From \eqref{eq:xiHer}, writing $\xi = \frac{2}{3}\zeta^{3/2}$, we infer (cf.~\cite[Eq.~(3.5)]{MR0109898})
\[
\xi = \frac{t^2}{2} - \frac{1}{2}\log2t - \frac{1}{4} +  O(t^{-2}) \quad (t\to+\infty),
\]
so that to leading order
\[
e^{-u\xi} \sim e^{-ut^2/2} (4et^2)^{u/4} \quad (t\to+\infty).
\]
On using the exponential form \eqref{eq:stirling} of the Stirling-type expansion, we get
\begin{equation}\label{eq:lambdaOddWeber}
\begin{aligned}
\exp\left(\sum_{k=1}^\infty \frac{\lambda_{2k-1}}{u^{2k-1}}\right)  &\sim \lim_{t\to+\infty} e^{-u\xi} \sqrt{\frac{\lambda_\pm w_\pm(t)}{w(t)}} = \sqrt{\frac{\Gamma(u/2+1/2)}{\sqrt{2\pi}(u/2e)^{u/2}}}\\*[2mm]
&\sim \exp\left(\sum_{k=1}^\infty \frac{2^{2k-1}B_{2k}(1/2)}{4k(2k-1)u^{2k-1}}\right),
\end{aligned}
\end{equation}
from which we read off the values of $\lambda_{2k-1}$ which are displayed in \eqref{eq:lambdaOddHer}.

\subsection{Application to  Whittaker's confluent hypergeometric function \for{toc}{$W_{\kappa,\mu}(z)$}\except{toc}{\boldmath $W_{\kappa,\mu}(z)$\unboldmath}}\label{app:whittaker} 
The function~\eqref{eq:gfLag} associated with the Laguerre polynomials, and studied in Section~\ref{subsect:LagGen}, can be expressed in the form \cite[Eq.~(13.18.17)]{MR2723248} 
\begin{equation}\label{eq:Whittaker}
\begin{gathered}
t^{1/2} \phi_n^{(\alpha)}(ut) = \frac{u^{-1/2}W_{\kappa,\mu}(ut)}{\sqrt{\Gamma(\kappa-\mu+1/2)\Gamma(\kappa+\mu +1/2)}},\\*[1mm]
u = n+1/2,\quad \alpha = (a^2-1)u, \quad \kappa = u(a^2+1)/2, \quad \mu = u(a^2-1)/2,
\end{gathered}
\end{equation}
in terms of Whittaker's confluent hypergeometric function $W_{\kappa,\mu}(z)$. Being defined in \cite[§13.14]{MR2723248}, $W_{\kappa,\mu}(z)$ is a many-valued analytic function of $z$ with a branch point at $z=0$ and its principal branch has a cut in the $z$-plane along the interval $(-\infty,0]$. Except when $z=0$, each branch is an entire function of the parameters $\kappa$ and $\mu$,
respecting the symmetry (implied by Kummer's transformation) \cite[Eq.~(13.14.31)]{MR2723248} 
\[
W_{\kappa,\mu}(z) = W_{\kappa,-\mu}(z).
\]
As discussed in Section~\ref{subsect:LagGen}, or by reference to Whittaker's differential equation \cite[Eq.~(13.14.1)]{MR2723248},
we have that
\begin{subequations}\label{eq:whittakerGen}
\begin{equation}
w(t) := W_{\kappa,\mu}(ut),\quad \kappa = u(a^2+1)/2, \quad \mu = u(a^2-1)/2, \quad u>0,
\end{equation}
satisfies the differential equation \eqref{eq:dw} with $f,g$ as in \eqref{eq:fgLag}. Hence,  
the two companion solutions with $u$ being replaced by $ue^{\mp\pi i}$ are 
\[
w_\pm(t) := W_{-\kappa,-\mu}(ut e^{\mp\pi i}).
\]
Since we are  in the parameter regime\footnote{The regime $1-\delta \leq |\mu|/\kappa \leq 1+\delta$ is studied in \cite{MR0592550}, $|\mu|/\kappa \geq 1+\delta$ in \cite{MR4312523}. The cases with $\kappa < 0$ are covered by analytic continuation and connection formulas.}
\[
0 \leq \frac{|\mu|}{\kappa} = \frac{|a^2-1|}{a^2+1} \leq 1 - \delta,
\]
for some $\delta>0$ when $0 < a_1 \leq a \leq a_1<\infty$ as in Theorem~\ref{thm:genLag}, the results of Dunster \cite[Eqs.~(6.17--19)]{MR0990876} apply and show, in particular, that the assumptions of Theorem~\ref{thm:connect} are met.
The coefficients in the connection formula \eqref{eq:wConnect} are (see \cite[Eq.~(2.21b)]{MR0240343}, \cite[Eq.~(2.9)]{MR0990876}, or \cite[Eq.~(2.5.19)]{MR0107026})
\[
\lambda_\pm = \frac{\Gamma(\kappa-\mu+1/2)\Gamma(\kappa+\mu+1/2)}{2\pi} e^{\mp \kappa \pi i}
\]
and the leading order asymptotic for $t\to+\infty$ is given by (see \cite[Eq.~(7.1b)]{MR0240343}, \cite[Eq.~(13.19.3)]{MR2723248})
\begin{align*}
w(t) &\sim e^{-u t/2} (ut)^{\kappa},\\*[1mm]
w_\pm(t) &\sim e^{u t/2} (ut)^{-\kappa} e^{\pm \kappa \pi i}.
\end{align*}
Hence, using $f(t)\sim 2^{-1/2}$ as $t\to +\infty$, the formula \eqref{eq:cu_no_xi} for the constant $c$ in the Airy expansion \eqref{eq:wExpand} evaluates straightforwardly
to
\begin{equation}
c \simeq  u^{1/6} \sqrt{\Gamma(\kappa-\mu+1/2)\Gamma(\kappa+\mu+1/2)}.
\end{equation}
\end{subequations}
A rescaling as in \eqref{eq:Whittaker} gives the value of the constant displayed in \eqref{eq:cuLag}.

In the same fashion, we evaluate formula \eqref{eq:lambdaGen} to compute the values of the limits $\lambda_{2k-1}$. From \eqref{eq:xiLag}, writing $\xi = \frac{2}{3}\zeta^{3/2}$, we infer (cf.~\cite[Eq.~(5.9)]{MR3874026})
\[
\xi(t) = \frac{t}{2} - \frac{1}{2}(a^2+1)(\log(t) +1) + a^2 \log(a) + O(t^{-1}) \quad (t\to+\infty),
\]
so that to leading order
\[
e^{-u\xi} \sim e^{-ut/2} (et)^{u/2} (e t/a^2) ^{a^2u/2} \quad (t\to+\infty).
\]
On using $\kappa-\mu = u$ and $\kappa + \mu = a^2 u$, and the exponential form \eqref{eq:stirling} of the Stirling-type expansion of $\Gamma(z+1/2)$ for large $z$,
we get 
\begin{equation}\label{eq:lambdaOddWhittaker}
\begin{aligned}
\exp\left(\sum_{k=1}^\infty \frac{\lambda_{2k-1}}{u^{2k-1}}\right)  &\sim \lim_{t\to+\infty} e^{-u\xi} \sqrt{\frac{\lambda_\pm w_\pm(t)}{w(t)}} = \sqrt{\frac{\Gamma(u+1/2)}{\sqrt{2\pi}(u/e)^u}\frac{\Gamma(a^2u+1/2)}{\sqrt{2\pi}(a^2u/e)^{a^2u}}}\\*[2mm]
&\sim \exp\left(\sum_{k=1}^\infty \frac{B_{2k}(1/2)(1+1/a^{4k-2})}{4k(2k-1)u^{2k-1}}\right) \quad (u\to\infty),
\end{aligned}
\end{equation}
from which we read off the values of $\lambda_{2k-1}$ which are displayed in \eqref{eq:lambdaOdd}.

\let\oldaddcontentsline\addcontentsline
\renewcommand{\addcontentsline}[3]{}%

\subsection*{Acknowledgements} 
We thank Peter Forrester for suggesting the general $\beta$-form of Theorem~\ref{thm:Laguerre2Gauss} and for highlighting its $\beta=1$ case in the book of Muirhead \cite{MR652932}.

\let\addcontentsline\oldaddcontentsline
\addtocontents{toc}{\vspace{1\baselineskip}}
\bibliographystyle{spmpsci}
\bibliography{paper}

\end{document}